\documentclass[oneside,reqno]{amsart}%
\usepackage[T1]{fontenc}
\usepackage[utf8]{inputenc}
\usepackage{amsthm}
\usepackage{amstext}
\usepackage{amssymb}
\usepackage{esint}
\usepackage[numbers]{natbib}
\usepackage[unicode=true,pdfusetitle,
bookmarks=true,bookmarksnumbered=false,bookmarksopen=false,
breaklinks=false,pdfborder={0 0 1},backref=section,colorlinks=false]%
{hyperref}
\usepackage{mathrsfs}
\usepackage{amsthm}
\usepackage{amstext}
\usepackage{datetime}
\usepackage{mathrsfs}
\usepackage{amsthm}
\usepackage{amstext}
\usepackage{amsfonts}
\usepackage{comment}
\usepackage{amsmath}
\usepackage{graphicx}%
\providecommand{\U}[1]{\protect\rule{.1in}{.1in}}
\makeatletter
\numberwithin{equation}{section}
\numberwithin{figure}{section}
\numberwithin{equation}{section}
\numberwithin{figure}{section}
\numberwithin{equation}{section}
\numberwithin{figure}{section}
\theoremstyle{plain}
\newtheorem{theorem}{Theorem}[section]
\newtheorem{corollary}[theorem]{Corollary}
\newtheorem*{corollary*}{Corollary}
\newtheorem{lemma}[theorem]{Lemma}
\newtheorem{convention}[theorem]{Convention}

\newtheorem{proposition}[theorem]{Proposition}

\renewenvironment{proof}[1][Proof]{\textbf{#1.} }{\ \rule{0.5em}{0.5em}}
\theoremstyle{definition}
\newtheorem{definition}[theorem]{Definition}
\newtheorem{comments}{Comments}
\newtheorem{notation}[theorem]{Notation}

\newtheorem{example}[theorem]{Example}
\theoremstyle{remark}
\newtheorem{remark}[theorem]{Remark}

\theoremstyle{plain}

\theoremstyle{definition}
\newtheorem{assumption}[]{Assumption}

\numberwithin{equation}{section}

\makeatother
\excludecomment{comments}
\begin{document}
\author[Driver]{Bruce K. Driver}
\author[Tong]{Pun Wai Tong}

\address[Driver]{Department of Mathematics, University of California, San Diego, La
Jolla, California 92093 USA.}
\email{bdriver@ucsd.edu}
\urladdr{www.math.ucsd.edu/$\sim$bdriver/}

\address[Tong]{Department of Mathematics, University of California, San Diego, La
Jolla, California 92093 USA. }
\email{p1tong@ucsd.edu}
\urladdr{www.math.ucsd.edu/$\sim$p1tong/}

\thanks{The research of Bruce K. Driver was supported in part by NSF Grant
DMS-0739164. The author also greatfully acknowledges the generosity and
hospitality of the Imperial College Mathematics department where the author
was a visiting Nelder fellow in the Fall of 2014.}
\thanks{The research of Pun Wai Tong was supported in part by NSF Grant DMS-0739164}
\title{Powers of Symmetric Differential Operators I.}
\date{October 28, 2015}

\begin{abstract}
Let $L$ be a linear symmetric differential operators on $L^{2}\left(
\mathbb{R}\right)  $ whose domain is the Schwartz test function space,
$\mathcal{S}.$ For the majority of this paper, it is assumed that the
coefficient of $L$ are polynomial functions on $\mathbb{R}.$ We will give
criteria on the polynomial coefficients of $L$ which guarantees that $L$ is
essentially self-adjoint, $\bar{L}\geq-CI$ for some $C<\infty,$ and that
$\mathcal{S}$ is a core for $\left(  \bar{L}+C\right)  ^{r}$ for all $r\geq0.$
Given another polynomial coefficient differential operator, $\tilde{L},$ we
will further give criteria on the coefficients $L$ and $\tilde{L}$ which
implies operator comparison inequalities of the form $\left(  \overline
{\tilde{L}}+\tilde{C}\right)  ^{r}\leq C_{r}\left(  \bar{L}+C\right)  ^{r}$
for all $0\leq r<\infty.$ The last inequality generalized to allow for an
added parameter, $\hbar>0,$ in the coefficients is used to provide a large
class of operators satisfying the hypotheses in \cite{BruceDriver20152nd}
where a strong form of the classical limit of quantum mechanics is shown to
hold. 

\end{abstract}
\maketitle
\tableofcontents


\section{Introduction\label{sec.1}}

Let $L^{2}\left(  \mathbb{R}\right)  :=L^{2}\left(  \mathbb{R},dx\right)  $ be
the Hilbert space of square integrable complex valued functions on
$\mathbb{R}$ relative to Lebesgue measure, $dx.$ The inner product on
$L^{2}\left(  \mathbb{R}\right)  $ is taken to be
\begin{equation}
\left\langle u,v\right\rangle :=\int_{\mathbb{R}}u\left(  x\right)  \bar
{v}\left(  x\right)  dx\text{ for }u,v\in L^{2}\left(  \mathbb{R}\right)
\label{e.1.1}%
\end{equation}
and the corresponding norm is $\left\Vert u\right\Vert =\sqrt{\left\langle
u,u\right\rangle }.$ [Note that we are using the mathematics convention that
$\left\langle u,v\right\rangle $ is linear in the first variable and conjugate
linear in the second.] .

\begin{notation}
\label{not.1.2}Let $C^{\infty}\left(  \mathbb{R}\right)  =C^{\infty
}(\mathbb{R},\mathbb{C})$ denote the smooth functions from $\mathbb{R}$ to
$\mathbb{C},$ $C_{c}^{\infty}\left(  \mathbb{R}\right)  $ denote those $f\in
C^{\infty}\left(  \mathbb{R}\right)  $ which have compact support, and
$\mathcal{S}:=\mathcal{S}\left(  \mathbb{R}\right)  \subset C^{\infty}\left(
\mathbb{R}\right)  $ be the subspace of Schwartz test functions, i.e. those
$f\in C^{\infty}\left(  \mathbb{R}\right)  $ such that $f$ and its derivatives
vanish at infinity faster than $\left\vert x\right\vert ^{-n}$ for all
$n\in\mathbb{N}.$
\end{notation}

\begin{notation}
\label{not.1.3}Let $C^{\infty}\left(  \mathbb{R}\right)  =C^{\infty
}(\mathbb{R},\mathbb{C})$ $.$ Also, let $\partial:C^{\infty}\left(
\mathbb{R}\right)  \rightarrow C^{\infty}\left(  \mathbb{R}\right)  $ denote
the differentiation operator, i.e. $\partial f\left(  x\right)  =f^{\prime
}\left(  x\right)  =\frac{d}{dx}f\left(  x\right)  .$
\end{notation}

\begin{notation}
\label{not.1.4}Given a function $f:\mathbb{R}\rightarrow\mathbb{C},$ we let
$M_{f}g:=fg$ for all functions $g:\mathbb{R}\rightarrow\mathbb{C},$ i.e.
$M_{f}$ denotes the linear operator given by multiplication by $f.$ Notice
that if $f\in C^{\infty}\left(  \mathbb{R}\right)  $ then we may view $M_{f}$
as a linear operator from $C^{\infty}\left(  \mathbb{R}\right)  $ to
$C^{\infty}\left(  \mathbb{R}\right)  .$
\end{notation}

For the purposes of this paper, a $d^{\text{th}}$\textbf{--order linear
differential operator} on $\mathbb{C^{\infty}\left(  R\right)  }$ with
$d\in\mathbb{N}$ is an operator $L:C^{\infty}\left(  \mathbb{R}\right)
\rightarrow C^{\infty}\left(  \mathbb{R}\right)  $ which may be expressed as
\begin{equation}
L=\sum_{k=0}^{d}M_{a_{k}}\partial^{k}=\sum_{k=0}^{d}a_{k}\partial^{k}
\label{e.1.2}%
\end{equation}
for some $\left\{  a_{k}\right\}  _{k=0}^{d}\subset C^{\infty}\left(
\mathbb{R},\mathbb{C}\right)  .$ The\textbf{ symbol} of $L,$ $\sigma
=\sigma_{L},$ is the function on $\mathbb{R}\times\mathbb{R}$ defined by
\begin{equation}
\sigma_{L}\left(  x,\xi\right)  :=\sum_{k=0}^{d}a_{k}\left(  x\right)  \left(
i\xi\right)  ^{k}. \label{e.1.3}%
\end{equation}

\begin{remark}
\label{rem.1.5}The action of $L$ on $C_{c}^{\infty}\left(  \mathbb{R}\right)
$ completely determines the coefficients, $\left\{  a_{k}\right\}  _{k=0}%
^{d}.$ Indeed, suppose that $x_{0}\in\mathbb{R}$ and $0\leq k\leq d$ and let
$\varphi\left(  x\right)  :=\left(  x-x_{0}\right)  ^{k}\chi\left(  x\right)
$ where $\chi\in C_{c}^{\infty}\left(  \mathbb{R}\right)  $ such that $\chi
$\thinspace$=1$ in a neighborhood of $x_{0}.$ Then an elementary computation
shows $k!\cdot a_{k}\left(  x_{0}\right)  =\left(  L\varphi\right)  \left(
x_{0}\right)  .$ In particular if $L\varphi\equiv0$ for all $\varphi\in
C_{c}^{\infty}\left(  \mathbb{R}\right)  $ then $a_{k}\equiv0$ for $0\leq
k\leq d$ and hence $L\varphi\equiv0$ for all $\varphi\in C^{\infty}\left(
\mathbb{R}\right)  .$
\end{remark}

\begin{definition}
[Formal adjoint and symmetry]\label{def.1.6} Suppose $L$ is a linear
differential operator on $C^{\infty}\left(  \mathbb{R}\right)  $ as in Eq.
(\ref{e.1.2}). Then $L^{^{\dag}}:C^{\infty}\left(  \mathbb{R}\right)
\rightarrow C^{\infty}\left(  \mathbb{R}\right)  $ denote the \textbf{formal
adjoint} of $L$ given by the differential operator,
\begin{equation}
L^{\dag}=\sum_{k=0}^{d}\left(  -1\right)  ^{k}\partial^{k}M_{\bar{a}_{k}%
}\text{ on }C^{\infty}\left(  \mathbb{R}\right)  . \label{e.1.4}%
\end{equation}
Moreover $L$ is said to be \textbf{symmetric} if $L^{\dagger}=L$ on
$C^{\infty}\left(  \mathbb{R}\right)  .$
\end{definition}

\begin{remark}
\label{rem.1.7}Using Remark \ref{rem.1.5}, one easily shows $L^{\dag}$ may
alternatively be characterized as that unique $d^{\text{th}}$--order
differential operator on $C^{\infty}\left(  \mathbb{R}\right)  $ such that
\begin{equation}
\left\langle Lf,g\right\rangle =\left\langle f,L^{\dag}g\right\rangle \text{
for all }f,g\in C_{c}^{\infty}\left(  \mathbb{R}\right)  . \label{e.1.5}%
\end{equation}
From this characterization it is then easily verified that;

\begin{enumerate}
\item The dagger operation is an involution, in particular $L^{\dag\dag}=L$
and if $S$ is another linear differential operator on $C^{\infty}\left(
\mathbb{R}\right)  ,$ then $\left(  LS\right)  ^{\dag}=S^{\dag}L^{\dag}.$

\item $L$ is symmetric iff $\left\langle Lf,g\right\rangle =\left\langle
f,Lg\right\rangle $ for all $f,g\in C_{c}^{\infty}\left(  \mathbb{R}\right)
.$
\end{enumerate}
\end{remark}

Proposition \ref{pro.2.2} below shows if $\left\{  a_{k}\right\}  _{k=0}%
^{d}\subset C^{\infty}\left(  \mathbb{R},\mathbb{R}\right)  ,$ then
$L=L^{\dag}$ iff $d=2m$ is even and there exists $\left\{  b_{l}\right\}
_{l=0}^{m}\subset C^{\infty}\left(  \mathbb{R},\mathbb{R}\right)  $ such that
\begin{equation}
L=L\left(  \left\{  b_{l}\right\}  _{l=0}^{m}\right)  :=\sum_{l=0}^{m}%
(-1)^{l}\partial^{l}b_{l}(x)\partial^{l}. \label{e.1.6}%
\end{equation}
The factor of $\left(  -1\right)  ^{l}$ is added for later convenience. The
coefficients $\left\{  b_{l}\right\}  _{l=0}^{m}$ are uniquely determined by
$\left\{  a_{2l}\right\}  _{l=0}^{m}$ (the even coefficients in Eq.
(\ref{e.1.2})) and in turn the coefficients $\left\{  a_{k}\right\}
_{k=0}^{2m}$ are determined by the $\left\{  b_{l}\right\}  _{l=0}^{m},$ see
Theorem \ref{the.2.7} and Lemma \ref{lem.2.4} respectively. We say that $L$ is
written in \textbf{divergence form }when $L$ is expressed as in Eq.
(\ref{e.1.6}).

From now on let us assume that $\left\{  a_{k}\right\}  _{k=0}^{d}\subset
C^{\infty}\left(  \mathbb{R},\mathbb{R}\right)  $ and $L$ is given as in Eq.
(\ref{e.1.2}). For each $n\in\mathbb{N},$ $L^{n}$ is a $dn$ order differential
operator on$\mathbb{\ }C^{\infty}\left(  \mathbb{R}\right)  $ and hence there
exists $\left\{  A_{k}\right\}  _{k=0}^{2mn}\subset C^{\infty}\left(
\mathbb{R},\mathbb{C}\right)  $ such that
\begin{equation}
L^{n}=\sum_{k=0}^{dn}A_{k}\partial^{k}. \label{e.1.7}%
\end{equation}
If we further assume that $L$ is symmetric (so $d=2m$ for some $m\in
\mathbb{N}_{0})$, then by Remark \ref{rem.1.7} $L^{n}$ is a symmetric $2mn$ -
order differential operator. Therefore by Proposition \ref{pro.2.2}, there
exists $\left\{  B_{\ell}\right\}  _{\ell=0}^{mn}\subset C^{\infty}\left(
\mathbb{R},\mathbb{R}\right)  $ so that $L^{n}$ may be written in divergence
form as%
\begin{equation}
L^{n}=\sum_{\ell=0}^{mn}\left(  -1\right)  ^{\ell}\partial^{\ell}B_{\ell
}\partial^{\ell}. \label{e.1.8}%
\end{equation}
Information about the coefficients $\left\{  A_{k}\right\}  _{k=0}^{2mn}$ and
$\left\{  B_{\ell}\right\}  _{\ell=0}^{mn}$ in terms of the divergence form
coefficients $\left\{  b_{l}\right\}  _{l=0}^{m}$ of $L$ may be found in
Propositions \ref{pro.3.9} and Proposition \ref{pro.3.10} respectively.

Let $\mathbb{R}\left[  x\right]  $ be the space of polynomial functions in one
variable, $x,$ with real coefficients.

\begin{remark}
\label{rem.1.8}If the coefficients, $\left\{  a_{k}\right\}  _{k=0}^{d=2m},$
of $L$ in the Eq. (\ref{e.1.2}) are in $\mathbb{R}\left[  x\right]  ,$ then
$L$ and $L^{\dag}$ are both linear differential operator on $C^{\infty}\left(
\mathbb{R}\right)  $ which leave $\mathcal{S}$ invariant. Moreover by simple
integration by parts Eq. (\ref{e.1.5}) holds with $C_{c}^{\infty}\left(
\mathbb{R}\right)  $ replaced by $\mathcal{S},$ i.e. $\left\langle
Lf,g\right\rangle =\left\langle f,L^{\dag}g\right\rangle $ for all
$f,g\in\mathcal{S}.$
\end{remark}

\begin{notation}
\label{not.1.9}For the remainder of this introduction we are going to assume
$L$ is symmetric ($L=L^{\dag}),$ $L$ is given in divergence form as in Eq.
(\ref{e.1.6}) with $\left\{  b_{l}\right\}  _{l=0}^{m}\subset\mathbb{R}\left[
x\right]  ,$ and we now view $L$ as an operator on $L^{2}\left(
\mathbb{R},m\right)  $ with $\mathcal{D}\left(  L\right)  =\mathcal{S}\subset
L^{2}\left(  \mathbb{R},m\right)  .$ In other words, we are going to replace
$L$ by $L|_{\mathcal{S}}.$
\end{notation}

The main results of this paper will now be summarized in the next two subsections.

\subsection{Essential self-adjointness results\label{sec.1.1}}

\begin{theorem}
\label{the.1.10}Let $m\in\mathbb{N},$ $\left\{  b_{l}\right\}  _{l=0}%
^{m}\subset\mathbb{R}\left[  x\right]  $ with $b_{m}\left(  x\right)  \neq0$
and assume;

\begin{enumerate}
\item either $\inf_{x}b_{l}\left(  x\right)  >0$ or $b_{l}\equiv0$ and

\item $\deg\left(  b_{l}\right)  \leq\max\left\{  \deg\left(  b_{0}\right)
,0\right\}  $ whenever $1\leq l\leq m$. [The zero polynomial is defined to be
of degree $-\infty.$]
\end{enumerate}

If $L$ is the unbounded operator on $L^{2}\left(  m\right)  $ as in Notation
\ref{not.1.9}, then $L^{n}$ (for which $\mathcal{D}\left(  L^{n}\right)  $ is
still $\mathcal{S})$ is essentially self-adjoint for all $n\in\mathbb{N}.$
\end{theorem}

\begin{remark}
\label{rem.1.11}Notice that assumption 1 of Theorem \ref{the.1.10} implies
$\deg\left(  b_{l}\right)  $ is even and the leading order coefficient of
$b_{l}$ is positive unless $b_{l}\equiv0.$
\end{remark}

\begin{definition}
[Subspace Symmetry]\label{def.1.12} Let $S$ be a dense subspace of a Hilbert
space $\mathcal{K}$ and $A$ be a linear operators on $\mathcal{K}.$ Then $A$
is said to be symmetric on $S$ if $S\subseteq\mathcal{D}\left(  A\right)  $
and
\[
\left\langle A\psi,\psi\right\rangle _{\mathcal{K}}=\left\langle \psi
,A\psi\right\rangle _{\mathcal{K}}\text{ for all }\psi\in S.
\]
The equality is equivalent to say $A|_{S}\subseteq\left(  A|_{S}\right)
^{\ast}$ or $A\subseteq A^{\ast}$ if $\mathcal{D}\left(  A\right)  =S.$
\end{definition}

\begin{remark}
\label{rem.1.13}Using Remark \ref{rem.1.7}, it is easy to see that $L$ with
polynomial coefficients is symmetric on $C^{\infty}\left(  \mathbb{R}\right)
$ as in Definition \ref{def.1.6} if and only if $L$ is symmetric on
$\mathcal{S}$ as in Definition \ref{def.1.12}.
\end{remark}

We now introduce three different partial ordering on symmetric operators on a
Hilbert space.

\begin{notation}
\label{not.1.14}Let $S$ be a dense subspace of a Hilbert space, $\mathcal{K},$
and $A$ and $B$ be two densely defined operators on $\mathcal{K}.$

\begin{enumerate}
\item We write $A\preceq_{S}B$ if both $A$ and $B$ are symmetric on $S$
(Definition \ref{def.1.12}) and
\[
\left\langle A\psi,\psi\right\rangle _{\mathcal{K}}\leq\left\langle B\psi
,\psi\right\rangle _{\mathcal{K}}\text{ for all }\psi\in S.
\]

\item We write $A\preceq B$ if $A\preceq_{D\left(  B\right)  }B,$ i.e.
$D\left(  B\right)  \subset D\left(  A\right)  ,$ $A$ and $B$ are both
symmetric on $D\left(  B\right)  $, and
\[
\left\langle A\psi,\psi\right\rangle _{\mathcal{K}}\leq\left\langle B\psi
,\psi\right\rangle _{\mathcal{K}}\text{ for all }\psi\in D\left(  B\right)  .
\]

\item If $A$ and $B$ are non-negative (i.e. $0\preceq A$ and $0\preceq B$)
self adjoint operators on a Hilbert space $\mathcal{K},$ then we say $A\leq B$
if and only if $D\left(  \sqrt{B}\right)  \subseteq D\left(  \sqrt{A}\right)
$ and
\[
\left\Vert \sqrt{A}\psi\right\Vert \leq\left\Vert \sqrt{B}\psi\right\Vert
\text{ for all }\psi\in D\left(  \sqrt{B}\right)  .
\]

\end{enumerate}
\end{notation}

There is a sizable literature dealing with similar essential self-adjointness
in Theorem \ref{the.1.10}, see for example
\citep{Chernoff1973,Kato1973,Reed1975}. Suppose that $b_{2},$ $b_{1},$ and
$b_{0}$ are smooth real-valued functions of $x\in\mathbb{R}$ and $T$ is an
differential operator on $C_{c}^{\infty}\left(  \mathbb{R}\right)  \subseteq
L^{2}\left(  \mathbb{R}\right)  $ defined by,
\[
T=-\partial b_{2}\left(  x\right)  \partial+b_{0}\left(  x\right)  +i\left(
b_{1}\left(  x\right)  \partial+\partial b_{1}\left(  x\right)  \right)  .
\]
Kato \citep{Kato1973} shows $T^{n}$ is essentially self-adjoint for all
$n\in\mathbb{N}$ when $b_{2}=1,$ $b_{1}=0$ and $-a-b\left\vert x\right\vert
^{2}\preceq_{C_{c}^{\infty}\left(  \mathbb{R}\right)  }T$ for some constants
$a$ and $b.$ Chernoff \citep{Chernoff1973} gives the same conclusion under
certain assumptions on $b_{2}$ and $T.$ For example, Chernoff's assumptions
would hold if $b_{2}$, $b_{1}$ and $b_{0}$ are real valued polynomial
functions such that $\deg\left(  b_{2}\right)  \leq2$ and $b_{2}$ is positive
and $T$ is semi-bounded on $C_{c}^{\infty}\left(  \mathbb{R}\right)  .$ In
contrast, Theorem \ref{the.1.10} allows for higher order differential
operators but does not allow for non-polynomial coefficients. [However, the
methods in this paper can be pushed further in order to allow for certain
non-polynomial coefficients.]

There are also a number of results regarding essential self-adjointness in the
pseudo-differential operator literature, the reader may be referred to, for
example,
\citep{Folland1989,Nagase1988,Nagase1990,Shubin2001,Zworski2012,Yamazaki1992}.
In fact, our proof of Theorem \ref{the.1.10} will be an adaptation of an
approach found in Theorem 3.1 in \citep{Nagase1988}.

\subsection{Operator Comparison Theorems\label{sec.1.2}}

Motivated by considerations involved in taking the classical limit of quantum
mechanics in \citep{BruceDriver20152nd} and the important paper by
\citep{Hepp1974}, we will define a scaled version of $L$ (see Notation
\ref{not.1.15}) where for any $\hbar>0$ we make the following replacements in
Eq. (\ref{e.1.6}),
\begin{equation}
x\rightarrow\sqrt{\hbar}M_{x}\text{ and }\partial\rightarrow\sqrt{\hbar
}\partial. \label{e.1.9}%
\end{equation}
For reasons explained in Theorem \ref{the.A.2} of the appendix, we are lead to
consider a more general class of operators parametrized by $\hbar>0.$

\begin{notation}
\label{not.1.15}Let
\begin{equation}
\left\{  b_{l,\hbar}\left(  \cdot\right)  :0\leq l\leq m\text{ and }%
\hbar>0\right\}  \subset\mathbb{R}\left[  x\right]  , \label{e.1.10}%
\end{equation}
and then define
\begin{equation}
L_{\hbar}=L\left(  \left\{  \hbar^{l}b_{l,\hbar}(\sqrt{\hbar}\left(
\cdot\right)  )\right\}  _{l=0}^{m}\right)  =\sum_{l=0}^{m}(-\hbar
)^{l}\partial^{l}b_{l,\hbar}(\sqrt{\hbar}\left(  \cdot\right)  )\partial
^{l}\text{ on }\mathcal{S}. \label{e.1.11}%
\end{equation}

\end{notation}

We now record an assumption which is needed in a number of the results stated below.

\begin{assumption}
\label{ass.1}Let $m\in\mathbb{N}_{0}.$ We say $\left\{  b_{l,\hbar}\right\}
_{l=0}^{m}\subset\mathbb{R}\left[  x\right]  $ and $\eta>0$ satisfies
Assumption 1 if the following conditions hold.

\begin{enumerate}
\item For $0\leq l\leq m$, $b_{l,\hbar}(x)=\sum_{j=0}^{2m_{l}}\alpha
_{l,j}\left(  \hbar\right)  x^{j}$ is a real polynomial of $x$ where
$\alpha_{l,j}$ is a real continuous function on $\left[  0,\eta\right]  .$

\item For all $0<\hbar<\eta$,
\begin{equation}
2m_{l}=\deg(b_{l,\hbar})\leq\deg(b_{l-1,\hbar})=2m_{l-1}\text{ for }1\leq
l\leq m. \label{e.1.12}%
\end{equation}

\item We have,%
\begin{align}
c_{b_{m}}  &  :=\inf_{x\in\mathbb{R},0<\hbar<\eta}b_{m,\hbar}(x)>0\text{ and
}\label{e.1.13}\\
c_{\alpha}  &  :=\min_{0\leq l\leq m}\inf_{0<\hbar<\eta}\alpha_{l,2m_{l}%
}\left(  \hbar\right)  >0, \label{e.1.14}%
\end{align}
i.e. $b_{m,\hbar}(x)$ is uniformly in $x\in\mathbb{R}$ and $0<\hbar<\eta$
positive and leading orders, $\alpha_{l,2m_{l}}\left(  \hbar\right)  ,$ of all
$b_{l,\hbar}\in\mathbb{R}\left[  x\right]  $ are uniformly strictly positive.
\end{enumerate}
\end{assumption}

\begin{remark}
\label{rem.1.16}Conditions $\left(  1\right)  $ and $\left(  3\right)  $ of
Assumption \ref{ass.1} implies there exists $A\in\left(  0,\infty\right)  $ so
that
\[
\min_{0\leq l\leq m}\inf_{0<\hbar<\eta}\inf_{\left\vert x\right\vert \geq
A}b_{l,\hbar}(x)>0.
\]
Furthermore, if $k\geq1,$ $0\leq l_{1},\ldots,l_{k}\leq m,$ and $q_{\hbar
}\left(  x\right)  =b_{l_{1},\hbar}\left(  x\right)  \ldots b_{l_{k},\hbar
}\left(  x\right)  \in\mathbb{R}\left[  x\right]  ,$ then
\[
q_{\hbar}\left(  x\right)  =\sum_{i=0}^{2M}Q_{i}\left(  \hbar\right)  \cdot
x^{i}%
\]
where $M=m_{l_{1}}+\ldots+m_{l_{k}},$ each of the coefficients, $Q_{i}\left(
\hbar\right)  $ is uniformly bounded for $0<\hbar<\eta,$ and
\[
\inf_{0<\hbar<\eta}Q_{2M}\left(  \hbar\right)  =\inf_{0<\hbar<\eta}%
\alpha_{l_{1},m_{l_{1}}\ldots}\alpha_{l_{k},m_{l_{k}}}\left(  \hbar\right)
\geq c_{\alpha}^{k}>0.
\]
From these remarks one easily shows $\inf_{0<\hbar<\eta}\inf_{x\in\mathbb{R}%
}q_{\hbar}\left(  x\right)  >-\infty.$
\end{remark}

The second main goal of this paper is to find criteria on two symmetric
differential operators $L_{\hbar}$ and $\tilde{L}_{\hbar}$ so that for each
$n\in\mathbb{N},$ there exists $K_{n}<\infty$ such that
\begin{equation}
L_{\hbar}^{n}\preceq_{\mathcal{S}}K_{n}\left(  \tilde{L}_{\hbar}^{n}+I\right)
. \label{e.1.15}%
\end{equation}
(As usual $I$ denotes the identity operator here and $\preceq_{\mathcal{S}}$
is as in Notation \ref{not.1.14}.) For some perspective let us recall the
L\"{o}wner-Heinz inequality.

\begin{theorem}
[L\"{o}wner-Heinz inequality]\label{the.1.17}If $A$ and $B$ are two
non-negative self-adjoint operators on a Hilbert space, $\mathcal{K},$ such
that $A\leq B,$ then $A^{r}\leq B^{r}$ for $0\leq r\leq1.$
\end{theorem}

L\"{o}wner proved this result for finite dimensional matrices in
\citep{Loewner1934} and Heinz extended it to bounded operators in a Hilbert
space in \citep{Heinz1951}. Later, both Heinz in \cite{Heinz1951} and Kato in
\citep[Theorem 2 of ][]{Kato1952} extended the result for unbounded operators,
also see \citep[Proposition 10.14 of][]{Schmuedgen2012}. There is a large
literature on so called \textquotedblleft operator monotone
functions,\textquotedblright\ e.g.
\citep{Pedersen1972,Fujii2005,Elst2009,Ando1994} and
\citep[Theorem 18]{Pecaric2005}. It is well known (see \cite[Section
10.3]{Schmuedgen2012} for more background) that $f\left(  x\right)  =x^{r}$ is
not an operator monotone for $r>1,$ see \cite[Example 10.3]{Schmuedgen2012}
for example. This indicates that proving operator inequalities of the form in
Eq. (\ref{e.1.15}) is somewhat delicate. Our main result in this direction is
the subject of the next theorem.

\begin{theorem}
[Operator Comparison Theorem]\label{the.1.18}Suppose that $\tilde{L}_{\hbar}$
and $L_{\hbar}$ are two linear differential operators on $\mathcal{S}$ given
by
\[
\tilde{L}_{\hbar}=\sum_{l=0}^{m_{\tilde{L}}}(-\hbar)^{l}\partial^{l}\tilde
{b}_{l,\hbar}(\sqrt{\hbar}x)\partial^{l}\text{ and }L_{\hbar}=\sum
_{l=0}^{m_{L}}(-\hbar)^{l}\partial^{l}b_{l,\hbar}(\sqrt{\hbar}x)\partial^{l},
\]
with polynomial coefficients, $\left\{  \tilde{b}_{l,\hbar}(x)\right\}
_{l=0}^{m_{\tilde{L}}}$ and $\left\{  b_{l,\hbar}(x)\right\}  _{l=0}^{m_{L}}$
satisfying Assumption \ref{ass.1} with constants $\eta_{\tilde{L}}$ and
$\eta_{L}$ respectively. Let $\eta=\min\{\eta_{\tilde{L}},\eta_{L}\}.$ If we
further assume that $m_{\tilde{L}}\leq m_{L}$ and there exists $c_{1}$ and
$c_{2}$ such that
\begin{equation}
\left\vert \tilde{b}_{l,\hbar}\left(  x\right)  \right\vert \leq c_{1}\left(
b_{l,\hbar}\left(  x\right)  +c_{2}\right)  \text{~}\forall~0\leq l\leq
m_{\tilde{L}}~\text{and~}0<\hbar<\eta, \label{e.1.16}%
\end{equation}
then for any $n\in\mathbb{N}$ there exists $C_{1}$ and $C_{2}$ such that
\begin{equation}
\tilde{L}_{\hbar}^{n}\preceq_{\mathcal{S}}C_{1}\left(  L_{\hbar}^{n}%
+C_{2}\right)  \text{ for all }0<\hbar<\eta. \label{e.1.17}%
\end{equation}

\end{theorem}

\begin{corollary}
\label{cor.1.19}If $\left\{  b_{l,\hbar}(x)\right\}  _{l=0}^{m_{L}}$ and
$\eta>0$ satisfy Assumption \ref{ass.1}, then there exists $C\in\mathbb{R}$
such that $CI\preceq_{\mathcal{S}}L_{\hbar}$ for all $0<\hbar<\eta.$
\end{corollary}

\begin{proof}
Define $\tilde{L}_{\hbar}=I$, i.e. we are taking $m_{\tilde{L}}=0$ and
$\tilde{b}_{0,\hbar}\left(  x\right)  =1.$ It then follows from Theorem
\ref{the.1.18} with $n=1$ that there exists $C_{1},C_{2}\in\left(
0,\,\infty\right)  $ such that $I=\tilde{L}_{\hbar}\preceq_{\mathcal{S}}%
C_{1}L_{\hbar}+C_{1}C_{2}$ and hence $L_{\hbar}+C_{2}\succeq_{\mathcal{S}%
}C_{1}^{-1}I.$
\end{proof}

A similar result to Theorem \ref{the.1.18} may be found in the
\citep[Theorem 1.1 of][]{Elst2009}. The paper \citep{Elst2009} compares the
standard Laplacian $-\triangle$ with an operator of $H_{0}$ in the form of
$-\sum_{i,j}^{d}\partial_{i}c_{ij}\left(  x\right)  \partial_{j}$ with
coefficients $\left\{  {c_{ij}}\right\}  _{i,j=1}^{d}$ lying in a Sobolev
spaces $W^{m+1,\infty}\left(  \mathbb{R}^{d}\right)  $ for some $m\in
\mathbb{N}$ and $\mathcal{D}\left(  H_{0}\right)  =W^{\infty,2}\left(
\mathbb{R}^{d}\right)  .$ The theorem shows that if $H_{0}$ is a symmetric,
positive and subelliptic of order $\gamma\in(0,1]$ then $\overline{H_{0}}$ is
positive self-adjoint and for all $\alpha\in\left[  0,\frac{m+1+\gamma^{-1}%
}{2}\right]  ,$ there exists $C_{\alpha}$ such that
\[
\left(  -\Delta\right)  ^{2\alpha\gamma}\leq C_{\alpha}\left(  I+\overline
{H_{0}}\right)  ^{2\alpha}
\]

As a corollary of Theorem \ref{the.1.10} and aspects of the proof of Theorem
\ref{the.1.18} given in Section \ref{sec.6} below, we have the following
corollaries which are proved in subsection \ref{sec.6.3} below.

\begin{corollary}
\label{cor.1.20}Supposed $\left\{  b_{l,\hbar}\left(  x\right)  \right\}
_{l=0}^{m}\subset\mathbb{R}\left[  x\right]  $ and $\eta>0$ satisfies
Assumption \ref{ass.1}, $L_{\hbar}$ is the operator in the Eq. (\ref{e.1.11}),
and suppose that $C\geq0$ has been chosen so that $0\preceq_{\mathcal{S}%
}L_{\hbar}+CI$ for all $0<\hbar<\eta.$ (The existence of $C$ is guaranteed by
Corollary \ref{cor.1.19}.) Then for any $0<\hbar<\eta,$ $\bar{L}_{\hbar}+CI$
is a non-negative self-adjoint operator on $L^{2}\left(  m\right)  $ and
$\mathcal{S}$ is a core for $\left(  \bar{L}_{\hbar}+C\right)  ^{r}$ for all
$r\geq0.$
\end{corollary}

\begin{corollary}
\label{cor.1.21}Suppose that $\tilde{L}_{\hbar}$ and $L_{\hbar}$ are two
linear differential operators and $\eta>0$ as in Theorem \ref{the.1.18}. If
$C\geq0$ and $\tilde{C}\geq0$ are chosen so that $L_{\hbar}+C\succeq
_{\mathcal{S}}I$ and $\tilde{L}_{\hbar}+\tilde{C}\succeq_{\mathcal{S}}0$ (as
is possible by Corollary \ref{cor.1.19}), then $\overline{\tilde{L}_{\hbar}%
}+\tilde{C}$ and $\bar{L}_{\hbar}+C$ are non-negative self adjoint operators
and for each $r\geq0$ there exists $C_{r}$ such that%
\begin{equation}
\left(  \overline{\tilde{L}_{\hbar}}+\tilde{C}\right)  ^{r}\preceq
C_{r}\left(  \bar{L}_{\hbar}+C\right)  ^{r}~\forall~0<\hbar<\eta.
\label{e.1.18}%
\end{equation}

\end{corollary}

\begin{definition}
[Number operator]\label{def.1.22}The \textbf{number operator}, $\mathcal{N},$
on $L^{2}\left(  \mathbb{R}\right)  $ is defined as the closure of
\begin{equation}
-\frac{1}{2}\partial^{2}+\frac{1}{2}x^{2}-\frac{1}{2}\text{ on }\mathcal{S}.
\label{e.1.19}%
\end{equation}
For any $\hbar>0$ we let $\mathcal{N}_{\hbar}:=\hbar\mathcal{N}$ which
commonly known as the \textbf{harmonic oscillator Hamiltonian}.
\end{definition}

The properties of the number operator are very well known. The next theorem
summarizes two such basic properties which we need for this paper.

\begin{theorem}
\label{the.1.23}The number operator, $\mathcal{N},$ is non-negative (i.e.
$0\preceq\mathcal{N}$) self-adjoint operator on $L^{2}\left(  \mathbb{R}%
\right)  $.
\end{theorem}

\begin{proof}
The self-adjointness of $\mathcal{N}$ is an easy consequence of Theorem X.28
in \citep{Reed1975} on p.184 or may be seen as a consequence of Corollary
\ref{cor.1.20}. The positivity is a simple consequence of the fact that
$\mathcal{N}|_{\mathcal{S}}=a_{1}^{\dagger}a_{1}$ where $a_{1}$ and
$a_{1}^{\dagger}$ are annihilation and creation operators (for $\hbar=1)$ as
given in Eq. (\ref{e.A.1}) of the appendix.
\end{proof}

The next corollary is a direct consequence from Corollaries \ref{cor.1.20} and
\ref{cor.1.21} where $\overline{\tilde{L}_{\hbar}}=\mathcal{N}_{\hbar}$ in Eq.
(\ref{e.1.18}).

\begin{corollary}
\label{cor.1.24}Suppose $m\geq1,$ $\left\{  b_{l,\hbar}(\cdot)\right\}
_{l=0}^{m}\subset\mathbb{R}\left[  x\right]  $ and $\eta>0$ satisfy Assumption
\ref{ass.1}, and $L_{\hbar}$ is the operator on $\mathcal{S}$ defined in Eq.
(\ref{e.1.11}).If $C\geq0$ is chosen so that $I\preceq_{\mathcal{S}}L_{\hbar}+C$ (see
Corollary \ref{cor.1.19}), then;

\begin{enumerate}
\item $\bar{L}_{\hbar}+C$ is a non-negative self-adjoint operator on
$L^{2}\left(  m\right)  $ for all $0<\hbar<\eta.$

\item $\mathcal{S}$ is a core for $\left(  \bar{L}_{\hbar}+C\right)  ^{r}$ for
all $r\geq0$ and $0<\hbar<\eta.$

\item If we further supposed $\deg(b_{0,\hbar})\geq2$, then there exists
$C_{r}>0$ such that
\begin{equation}
\mathcal{N}_{\hbar}^{r}\preceq C_{r}\left(  \bar{L}_{\hbar}+C\right)  ^{r}
\label{e.1.20}%
\end{equation}
for all $0<\hbar<\eta$ and $r\geq0.$\footnote{By the spectral theorem one
shows $D\left(  \left\vert \bar{L}_{\hbar}\right\vert ^{r}\right)  =D\left(
\left(  \bar{L}_{\hbar}+C\right)  ^{r}\right)  .$}
\end{enumerate}
\end{corollary}

\section{A Structure Theorem for Symmetric Differential Operators\label{sec.2}%
}

\begin{remark}
\label{rem.2.1}It is useful to observe if $A$ and $B$ are two linear
transformation from a vector space, $V,$ to itself, then
\[
AB^{2}+B^{2}A=2BAB+\left[  B,\left[  B,A\right]  \right]  ,
\]
where $\left[  A,B\right]  :=AB-BA$ denotes the commutator of $A$ and $B.$
\end{remark}

\begin{proposition}
\label{pro.2.2}Suppose $\left\{  a_{k}\right\}  _{k=0}^{d}\subset C^{\infty
}\left(  \mathbb{R},\mathbb{R}\right)  $ and $L$ is the $d^{th}$ -order
differential operator on $C^{\infty}(\mathbb{R})$ as defined in Eq.
(\ref{e.1.2}). If $L$ is symmetric according to Definition \ref{def.1.6} (i.e.
$L=L^{\dag}$ where $L^{\dag}$ is as in Eq. (\ref{e.1.4})), then $d$ is even
(let $m=d/2)$ and there exists $\left\{  b_{l}\right\}  _{l=0}^{m}\subset
C^{\infty}\left(  \mathbb{R},\mathbb{R}\right)  $ such that
\begin{equation}
L=\sum_{l=0}^{m}\left(  -1\right)  ^{l}\partial^{l}M_{b_{l}}\partial^{l}
\label{e.2.1}%
\end{equation}
where $M_{b_{l}}$ is as in Notation \ref{not.1.4}. Moreover, $b_{m}=\left(
-1\right)  ^{m}a_{2m}=\left(  -1\right)  ^{m}a_{d}.$
\end{proposition}

\begin{proof}
Since $L=L^{\dag},$ we have
\begin{align}
L  &  =\frac{1}{2}(L+L^{\dag})=\frac{1}{2}\sum_{k=0}^{d}[a_{k}\partial
^{k}+(-1)^{k}\partial^{k}M_{a_{k}}]\label{e.2.2}\\
&  =\frac{1}{2}[a_{d}\partial^{d}+(-1)^{d}\partial^{d}M_{a_{d}}]+[\text{diff.
operator of order }d-1]. \label{e.2.3}%
\end{align}
If $d$ were odd, then $\left(  -1\right)  ^{d}=-1$ and hence (using the
product rule),
\begin{align*}
\frac{1}{2}[a_{d}\partial^{d}+(-1)^{d}\partial^{d}M_{a_{d}}]  &  =\frac{1}%
{2}[M_{a_{d}},\partial^{d}]\\
&  =[\text{diff. operator of order }d-1]
\end{align*}
which combined with Eq. (\ref{e.2.3}) would imply that $L$ was in fact a
differential operator of order no greater than $d-1$. This shows that $L$ must
be an even order operator.

Now knowing that $d$ is even, let $m:=d/2\in\mathbb{N}.$ From Eq.
(\ref{e.2.2}), we learn that
\begin{align*}
L  &  =\frac{1}{2}\sum_{k=0}^{2m}\left[  a_{k}\partial^{k}+\left(  -1\right)
^{k}\partial^{k}M_{a_{k}}\right] \\
&  =\frac{1}{2}\left[  a_{2m}\partial^{2m}+\partial^{2m}M_{a_{2m}}\right]  +R
\end{align*}
where $R$ is given by
\[
R=\frac{1}{2}\sum_{k=0}^{2m-1}\left[  M_{a_{k}}\partial^{k}+\left(  -1\right)
^{k}\partial^{k}M_{a_{k}}\right]  .
\]
Moreover by Remark \ref{rem.1.7}, $R$ is still symmetric. As in the previous
paragraph $R$ is in fact an even order differential operator and its order is
at most $2m-2.$ Using Remark \ref{rem.2.1} with $A=M_{a_{2m}},$ $B=\partial
^{m},$ and $V=C^{\infty}\left(  \mathbb{R}\right)  ,$ we learn that
\begin{align*}
\frac{1}{2}\left[  a_{2m}\partial^{2m}+\partial^{2m}M_{a_{2m}}\right]   &
=\partial^{m}M_{a_{2m}}\partial^{m}+\frac{1}{2}\left[  \partial^{m},\left[
\partial^{m},M_{a_{2m}}\right]  \right] \\
&  =\partial^{m}M_{a_{2m}}\partial^{m}+[\text{diff. operator of order at most
}2m-2].
\end{align*}
Combining the last three displayed equations together shows
\[
L=\partial^{m}a_{2m}\partial^{m}+S
\]
where $S=L-\partial^{m}a_{2m}\partial^{m}$ is a symmetric (by Remark
\ref{rem.1.7}) even order differential operator or at most $2m-2$. It now
follows by the induction hypothesis that
\[
S=\sum_{l=0}^{m-1}\left(  -1\right)  ^{l}\partial^{l}M_{b_{l}}\partial
^{l}\implies L=\sum_{l=0}^{m}\left(  -1\right)  ^{l}\partial^{l}M_{b_{l}%
}\partial^{l}%
\]
where $b_{m}:=\left(  -1\right)  ^{m}a_{2m}.$
\end{proof}

Our next goal is to record the explicit relationship between $\left\{
a_{k}\right\}  _{k=0}^{2m}$ in $\text{Eq. (\ref{e.1.2})}$ and $\left\{
b_{k}\right\}  _{k=0}^{m}$ in $\text{Eq. (}$\ref{e.2.1}).

\begin{convention}
\label{con.2.3}To simplify notation in what follows, for $k$, $l\in\mathbb{Z}%
$, let
\[
\dbinom{l}{k}:=%
\begin{cases}
\frac{l!}{k!(l-k)!} & \text{if }0\leq k\leq l\\
0 & \text{otherwise.}%
\end{cases}
\]

\end{convention}

\begin{lemma}
\label{lem.2.4}If $\left\{  a_{k}\right\}  _{k=0}^{2m}\cup\left\{
b_{l}\right\}  _{l=0}^{m}\subset C^{\infty}\left(  \mathbb{R},\mathbb{R}%
\right)  $ and
\begin{equation}
\sum_{l=0}^{m}\left(  -1\right)  ^{l}\partial^{l}b_{l}\left(  x\right)
\partial^{l}=\sum_{k=0}^{2m}a_{k}\left(  x\right)  \partial^{k}%
\text{\thinspace}, \label{e.2.4}%
\end{equation}
then
\begin{equation}
a_{k}:=\sum_{l=0}^{m}\dbinom{l}{k-l}\left(  -1\right)  ^{l}\partial
^{2l-k}b_{l}. \label{e.2.5}%
\end{equation}

\end{lemma}

\begin{proof}
By the product rule,
\[
\partial^{l}M_{b_{l}}=\sum_{r=0}^{l}\dbinom{l}{r}\left(  \partial^{l-r}%
b_{l}\right)  \partial^{r}\text{,}
\]
and therefore,
\begin{align*}
\sum_{l=0}^{m}\left(  -1\right)  ^{l}\partial^{l}M_{b_{l}}\partial^{l}  &
=\sum_{l=0}^{m}\sum_{r=0}^{l}\dbinom{l}{r}\left(  -1\right)  ^{l}\left(
\partial^{l-r}b_{l}\right)  \partial^{l+r}\\
&  =\sum_{k=0}^{2m}\left[  \sum_{l=0}^{m}\sum_{r=0}^{m}\dbinom{l}{r}\left(
-1\right)  ^{l}\left(  \partial^{l-r}b_{l}\right)  1_{k=l+r}\right]
\partial^{k}\\
&  =\sum_{k=0}^{2m}\left[  \sum_{l=0}^{m}\dbinom{l}{k-l}\left(  -1\right)
^{l}\left(  \partial^{2l-k}b_{l}\right)  \right]  \partial^{k}\text{.}%
\end{align*}
Combining this result with Eq. (\ref{e.2.4}) gives the identities in Eq.
(\ref{e.2.5}).
\end{proof}

Let us observe that the binomial coefficient of $a_{l}$ is zero unless $0\leq
k-l\leq l$, i.e. $l\leq k\leq2l$. To emphasize this restriction, we may write
Eq. (\ref{e.2.5}) as
\begin{equation}
a_{k}=\sum_{l=0}^{m}1_{l\leq k\leq2l}\dbinom{l}{k-l}\left(  -1\right)
^{l}\partial^{2l-k}b_{l}. \label{e.2.6}%
\end{equation}

Taking $k=2p$ in Eq. (\ref{e.2.6}) and multiplying the result by $\left(
-1\right)  ^{p}=\left(  -1\right)  ^{-p}$ leads to the following corollary.

\begin{corollary}
\label{cor.2.5}For $0\leq p\leq m,$
\begin{equation}
\left(  -1\right)  ^{p}a_{2p}=\sum_{l=0}^{m}1_{p\leq l\leq2p}\dbinom{l}%
{2p-l}\left(  -1\right)  ^{\left(  l-p\right)  }\partial^{2(l-p)}b_{l}.
\label{e.2.7}%
\end{equation}

\end{corollary}

We will see in Theorem \ref{the.2.7} below that the relations in Eq.
(\ref{e.2.7}) may be used to uniquely write the $\left\{  b_{l}\right\}
_{l=0}^{m}$ in terms of linear combinations for the $\left\{  a_{2k}\right\}
_{k=0}^{m}.$ In particular this shows if the operator $L$ described in Eq.
(\ref{e.1.2}) is symmetric then $\left\{  b_{l}\right\}  _{l=0}^{m}$ is
completely determined by the $a_{k}$ with $k$ even.

\subsection{The divergence form of $L$\label{sec.2.1}}

\begin{notation}
\label{not.2.6}For $r,s,n\in\mathbb{N}_{0}$ and $0\leq r,s\leq m$, let
\[
C_{n}(r,s)=\sum\dbinom{k_{1}}{2r-k_{1}}\dbinom{k_{2}}{2k_{1}-k_{2}}%
\dots\dbinom{k_{n-1}}{2k_{n-2}-k_{n-1}}\dbinom{s}{2k_{n-1}-s}
\]
where the sum is over $r<k_{1}<k_{2}<\dots<k_{n-1}<s.$ We also let
\begin{equation}
K_{m}(r,s)=\sum_{n=1}^{m-1}(-1)^{n}C_{n}(r,s). \label{e.2.8}%
\end{equation}
In particular, $C_{n}\left(  0,s\right)  =C_{n}\left(  m,s\right)
=K_{m}\left(  0,s\right)  =K_{m}\left(  m,s\right)  =0$ for all $0\leq s\leq
m.$
\end{notation}

\begin{theorem}
\label{the.2.7}If Eq. (\ref{e.2.4}) holds then
\begin{equation}
\left(  -1\right)  ^{r}b_{r}=a_{2r}+\sum_{r<s\leq m}K_{m}(r,s)\partial
^{2(s-r)}a_{2s}\text{ }\forall~0\leq r\leq m. \label{e.2.9}%
\end{equation}

\end{theorem}

\begin{proof}
For $x\in\mathbb{R}$ let $\mathbf{b}\left(  x\right)  $ and $\mathbf{a}\left(
x\right)  $ denote the column vectors in $\mathbb{R}^{m+1}$ defined by
\begin{align*}
\mathbf{b}\left(  x\right)   &  =(\left(  -1\right)  ^{0}b_{0}\left(
x\right)  ,\left(  -1\right)  ^{1}b_{1}\left(  x\right)  \dots,\left(
-1\right)  ^{m}b_{m}\left(  x\right)  )^{\operatorname{tr}}\text{ and}\\
\mathbf{a}\left(  x\right)   &  =(a_{0}\left(  x\right)  ,a_{2}\left(
x\right)  ,a_{4}\left(  x\right)  ,\dots,a_{2m}\left(  x\right)  )^{tr}.
\end{align*}
Further let $U$ be the $(m+1)\times(m+1)$ matrix with entries $\{U_{r,k}%
\}_{r,k=0}^{m}$ which are linear constant coefficient differential operators
given by
\[
U_{r,k}:=1_{r<k\leq2r}\dbinom{k}{2r-k}\partial^{2(k-r)}.
\]
Notice that by definition, $U_{r,k}=0$ unless $k>r$ and $U_{0,k}=0$ for $0\leq
k\leq m.$ Hence $U$ is nilpotent and $U^{m}=0$. Further observe that Eq.
(\ref{e.2.7}) may be written as
\begin{align*}
a_{2r}  &  =\left(  -1\right)  ^{r}b_{r}+\sum_{r<k\leq m}\dbinom{k}%
{2r-k}\left(  -1\right)  ^{k}\partial^{2(k-r)}b_{k}\\
&  =\left(  -1\right)  ^{r}b_{r}+\sum_{r<k\leq m}U_{r,k}\left(  -1\right)
^{k}b_{k}%
\end{align*}
or equivalently stated \textbf{$\mathbf{a}$}$=(I+U)\mathbf{b.}$ As $U$ is
nilpotent with $U^{m}=0,$ this last equation may be solved for $\mathbf{b}$
using
\begin{equation}
\mathbf{b}=(I+U)^{-1}\mathbf{a}=\mathbf{a}+\sum_{n=1}^{m-1}(-U)^{n}\mathbf{a}.
\label{e.2.10}%
\end{equation}
In components this equation reads,
\begin{equation}
\left(  -1\right)  ^{r}b_{r}=a_{r}+\sum_{n=1}^{m-1}\left(  -1\right)  ^{n}%
\sum_{s=0}^{m}U_{r,s}^{n}a_{2s} \label{e.2.11}%
\end{equation}
However, with the aid of Lemma \ref{lem.2.8} below and the definition of
$K_{m}\left(  r,s\right)  $ in Eq. (\ref{e.2.8}) it follows that
\[
\sum_{n=1}^{m-1}(-1)^{n}U_{r,s}^{n}=\sum_{n=1}^{m-1}(-1)^{n}C_{n}%
(r,s)\partial^{2(s-r)}=K_{m}(r,s)\partial^{2(s-r)}
\]
which combined with Eq. (\ref{e.2.11}) and the fact that $K_{m}\left(
r,s\right)  =0$ unless $0<r<s\leq m$ proves Eq. (\ref{e.2.9}).
\end{proof}

\begin{lemma}
\label{lem.2.8}Let $1\leq n\leq m$ and $0\leq r,s\leq m$, then $U^{m}=0$ and
\begin{equation}
U_{r,s}^{n}=C_{n}(r,s)\partial^{2(s-r)}. \label{e.2.12}%
\end{equation}

\end{lemma}

\begin{proof}
By definition of matrix multiplication,
\begin{align*}
U_{r,s}^{n}=  &  \sum_{k_{1},\dots,k_{n-1}=1}^{m}1_{r<k_{1}\leq2r}%
\dbinom{k_{1}}{2r-k_{1}}\partial^{2(k_{1}-r)}1_{k_{1}<k_{2}\leq2k_{1}}%
\dbinom{k_{2}}{2k_{1}-k_{2}}\partial^{2(k_{2}-k_{1})}\dots\\
&  \qquad\qquad\dots1_{k_{n-1}<k_{n}\leq2k_{n-1}}\dbinom{k_{n}}{2k_{n-1}%
-k_{n}}\partial^{2(k_{n}-k_{n-1})}1_{k_{n}=s}\\
=  &  \sum_{r<k_{1}<k_{2}<\dots<k_{n-1}<s}\dbinom{k_{1}}{2r-k_{1}}%
\dbinom{k_{2}}{2k_{1}-k_{2}}\dots\dbinom{k_{n-1}}{2k_{n-2}-k_{n-1}}\dbinom
{s}{2k_{n-1}-s}\partial^{2(s-r)}\\
=  &  C_{n}(r,s)\partial^{2(s-r)}.
\end{align*}

\end{proof}

\section{The structure of $L^{n}$\label{sec.3}}

In this section let us fix a $2m$ -- order symmetric differential operator,
$L,$ acting on $C^{\infty}(\mathbb{R})\ $which can be written as in both of
the equations (\ref{e.1.2}) and (\ref{e.2.1}) where the coefficients,
$\left\{  a_{k}\right\}  _{k=0}^{2m}$ and $\left\{  b_{l}\right\}  _{l=0}^{m}$
are all real valued smooth functions on $\mathbb{R}.$ If $n\in\mathbb{N},$
$L^{n}$ is a $2mn$ -- order symmetric linear differential operator on
$C^{\infty}\left(  \mathbb{R}\right)  $ and hence there exits $\left\{
A_{k}\right\}  _{k=0}^{2mn}\subset C^{\infty}\left(  \mathbb{R},\mathbb{R}%
\right)  $ and (using Proposition \ref{pro.2.2}) $\left\{  B_{l}\right\}
_{l=0}^{mn}\subset C^{\infty}\left(  \mathbb{R},\mathbb{R}\right)  $ such
that
\begin{equation}
L^{n}=\sum_{k=0}^{2mn}A_{k}\partial^{k}=\sum_{\ell=0}^{mn}\left(  -1\right)
^{\ell}\partial^{\ell}B_{\ell}\partial^{\ell}. \label{e.3.1}%
\end{equation}

Our goal in this section is to compute the coefficients $\left\{
A_{k}\right\}  _{k=0}^{2mn}$ in terms of the coefficients $\left\{
b_{l}\right\}  _{l=0}^{m}$ defining $L$ as in Eq. (\ref{e.2.1}). It turns out
that it is useful to compare $L^{n}$ to the operators which is constructed by
writing out $L^{n}$ while pretending that the coefficients $\left\{
a_{k}\right\}  _{k=0}^{2m}$ or $\left\{  b_{l}\right\}  _{l=0}^{m}$ are
constant. This is explained in the following notations.

\begin{notation}
\label{not.3.1}For $n\in\mathbb{N}$ and $m\in\mathbb{N},$ let $\Lambda
_{m}:=\left\{  0,1,\dots,m\right\}  \subset\mathbb{N}_{0}$ and for
$\mathbf{j=}\left(  j_{1},\dots,j_{n}\right)  \in\Lambda_{m}^{n},$ let
$\left\vert \mathbf{j}\right\vert =j_{1}+j_{2}+\dots+j_{n}.$ If $\mathbf{k=}%
\left(  k_{1},\dots,k_{n}\right)  \in\Lambda_{m}^{n}$ is another multi-index,
we write $\mathbf{k\leq j}$ to mean $k_{i}\leq j_{i}$ for $1\leq i\leq n.$ [We
will use this notation when $m=\infty$ as well in which case $\Lambda_{\infty
}=\mathbb{N}_{0}.]$
\end{notation}

\begin{notation}
\label{not.3.2}Given $n\in\mathbb{N},$ and $L$ as in Eq. (\ref{e.2.1}), let
$\left\{  \mathcal{B}_{\ell}\right\}  _{\ell=0}^{mn}$ be $C^{\infty}\left(
\mathbb{R},\mathbb{R}\right)  $ -- functions defined by
\begin{equation}
\mathcal{B}_{\ell}:=\sum_{\mathbf{j}\in\Lambda_{m}^{n}}1_{\left\vert
\mathbf{j}\right\vert =\ell}b_{j_{1}}\dots b_{j_{n}} \label{e.3.2}%
\end{equation}
and $\mathcal{L}_{\mathcal{B}}^{\left(  n\right)  }$ be the differential
operator given by
\begin{equation}
\mathcal{L}_{\mathcal{B}}^{\left(  n\right)  }:=\sum_{\ell=0}^{nm}\left(
-1\right)  ^{\ell}\partial^{\ell}\mathcal{B}_{\ell}\partial^{\ell}.
\label{e.3.3}%
\end{equation}
It will also be convenient later to set $\mathcal{B}_{k/2}\equiv0$ when $k$ is
an odd integer.
\end{notation}

\begin{example}
\label{exa.3.4}If $m=1$ and $n=2,$ then
\begin{align*}
L  &  =-\partial b_{1}\partial+b_{0}\text{ and }\\
L^{2}  &  =\partial b_{1}\partial^{2}b_{1}\partial-\partial b_{1}\partial
b_{0}-b_{0}\partial b_{1}\partial+b_{0}^{2}.
\end{align*}
To put $L^{2}$ into divergence form we repeatedly use the product rule,
$\partial V=V\partial+V^{\prime}.$ Thus
\begin{align*}
\partial b_{1}\partial b_{0}+b_{0}\partial b_{1}\partial &  =\partial
b_{1}b_{0}\partial+\partial b_{1}b_{0}^{\prime}+\partial b_{0}b_{1}%
\partial-b_{0}^{\prime}b_{1}\partial\\
&  =2\partial b_{1}b_{0}\partial+\left(  b_{1}b_{0}^{\prime}\right)  ^{\prime}%
\end{align*}
and
\begin{align*}
\partial b_{1}\partial^{2}b_{1}\partial &  =\partial^{2}b_{1}\partial
b_{1}\partial-\partial b_{1}^{\prime}\partial b_{1}\partial\\
&  =\partial^{2}b_{1}b_{1}\partial^{2}+\partial^{2}b_{1}b_{1}^{\prime}%
\partial-\partial b_{1}^{\prime}b_{1}^{\prime}\partial-\partial b_{1}^{\prime
}b_{1}\partial^{2}\\
&  =\partial^{2}b_{1}b_{1}\partial^{2}+\partial\left(  b_{1}b_{1}^{\prime
}\right)  ^{\prime}\partial-\partial b_{1}^{\prime}b_{1}^{\prime}\partial.
\end{align*}
Combining the last three displayed equations together shows
\begin{align}
L^{2}  &  =\partial^{2}b_{1}^{2}\partial^{2}+\partial\left[  -2b_{1}%
b_{0}+\left(  b_{1}b_{1}^{\prime}\right)  ^{\prime}-\left(  b_{1}^{\prime
}\right)  ^{2}\right]  \partial+b_{0}^{2}-\left(  b_{1}b_{0}^{\prime}\right)
^{\prime}\nonumber\\
&  =\partial^{2}b_{1}^{2}\partial^{2}+\partial\left[  -2b_{1}b_{0}+b_{1}%
b_{1}^{\prime\prime}\right]  \partial+b_{0}^{2}-\left(  b_{1}b_{0}^{\prime
}\right)  ^{\prime}. \label{e.3.6}%
\end{align}

Dropping all terms in Eq. (\ref{e.3.6}) which contain a derivative of $b_{1}$
or $b_{0}$ shows
\begin{equation}
\mathcal{L}_{\mathcal{B}}^{\left(  2\right)  }=\partial^{2}b_{1}^{2}%
\partial^{2}-\partial\left[  2b_{0}b_{1}\right]  \partial+b_{0}^{2}.
\label{e.3.7}%
\end{equation}

\end{example}

\begin{notation}
\label{not.3.6}For $\mathbf{j}\in\mathbb{N}_{0}^{n}$\textbf{ }and
$\mathbf{k}\in\mathbb{N}_{0}^{n},$\textbf{ }let
\[
\dbinom{\mathbf{k}}{\mathbf{j}}:=\prod_{i=1}^{n}\binom{k_{i}}{j_{i}}
\]
where the binomial coefficients are as in Convention \ref{con.2.3}.
\end{notation}

\begin{lemma}
\label{lem.3.7}If $L$ is as in Eq. (\ref{e.2.1}),
\begin{equation}
M_{e^{-i\xi\left(  \cdot\right)  }}LM_{e^{i\xi\left(  \cdot\right)  }}%
=\sum_{l=0}^{m}\left(  -1\right)  ^{l}\left(  \partial+i\xi\right)
^{l}M_{b_{l}\left(  \cdot\right)  }\left(  \partial+i\xi\right)  ^{l}.
\label{e.3.10}%
\end{equation}

\end{lemma}

\begin{proof}
If $f\in C^{\infty}\left(  \mathbb{R}\right)  ,$ the product rule gives,
\[
\partial_{x}\left[  e^{i\xi x}f\left(  x\right)  \right]  =e^{i\xi x}\left(
\partial_{x}+i\xi\right)  f\left(  x\right)  ,
\]
which is to say,
\begin{equation}
M_{e^{-i\xi\left(  \cdot\right)  }}\partial M_{e^{i\xi\left(  \cdot\right)  }%
}=\left(  \partial+i\xi\right)  . \label{e.3.11}%
\end{equation}
Combining Eq. (\ref{e.3.11}) with the fact that
\[
M_{e^{-i\xi\left(  \cdot\right)  }}M_{b_{l}}M_{e^{i\xi\left(  \cdot\right)  }%
}=M_{b_{l}}
\]
allows us to conclude,
\begin{align*}
M_{e^{-i\xi\left(  \cdot\right)  }}LM_{e^{i\xi\left(  \cdot\right)  }}  &
=\sum_{l=0}^{m}\left(  -1\right)  ^{l}M_{e^{-i\xi\left(  \cdot\right)  }%
}\partial^{l}M_{b_{l}}\partial^{l}M_{e^{i\xi\left(  \cdot\right)  }}\\
&  =\sum_{l=0}^{m}\left(  -1\right)  ^{l}M_{e^{-i\xi\left(  \cdot\right)  }%
}\partial^{l}M_{e^{i\xi\left(  \cdot\right)  }}M_{b_{l}}M_{e^{-i\xi\left(
\cdot\right)  }}\partial^{l}M_{e^{i\xi\left(  \cdot\right)  }}\\
&  =\sum_{l=0}^{m}\left(  -1\right)  ^{l}\left(  \partial+i\xi\right)
^{l}M_{b_{l}\left(  \cdot\right)  }\left(  \partial+i\xi\right)  ^{l}.
\end{align*}

\end{proof}

\begin{notation}
\label{not.3.8} For $\mathbf{q,l,j\in}\Lambda_{m}^{n},$ let
\begin{equation}
C_{k}\left(  \mathbf{q,l,j}\right)  :=\left(  -1\right)  ^{\left\vert
\mathbf{q}\right\vert }\binom{\mathbf{q}}{\mathbf{l}}\binom{\mathbf{q}%
}{\mathbf{j}}1_{j_{1}=0}1_{2\left\vert \mathbf{q}\right\vert -k=\left\vert
\mathbf{l}\right\vert +\left\vert \mathbf{j}\right\vert >0}, \label{e.3.12}%
\end{equation}
and for $k\in\Lambda_{2m}$ let
\begin{equation}
T_{k}:=\sum_{\mathbf{q,l,j}\in\Lambda_{m}^{n}}C_{k}\left(  \mathbf{q,l,j}%
\right)  \left(  \partial^{l_{n}}M_{b_{q_{n}}}\partial^{j_{n}}\right)  \left(
\partial^{l_{n-1}}M_{b_{q_{n-1}}}\partial^{j_{n-1}}\right)  \dots\left(
\partial^{l_{1}}M_{b_{q_{1}}}\partial^{j_{1}}\right)  \mathbf{1}
\label{e.3.13}%
\end{equation}
where $\mathbf{1}$ is a function constantly equal to $1.$ We will often abuse
notation and write this last equation as,
\[
T_{k}\left(  x\right)  :=\sum_{\mathbf{q,l,j}\in\Lambda_{m}^{n}}C_{k}\left(
\mathbf{q,l,j}\right)  \left(  \partial_{x}^{l_{n}}b_{q_{n}}\left(  x\right)
\partial_{x}^{j_{n}}\right)  \dots\left(  \partial_{x}^{l_{2}}b_{q_{2}%
}\partial_{x}^{j_{2}}\right)  \partial_{x}^{l_{1}}b_{q_{1}}\left(  x\right)
.
\]

\end{notation}

\begin{proposition}
[$A_{k}=A_{k}\left(  \left\{  b_{l}\right\}  _{l=0}^{m}\right)  $%
]\label{pro.3.9}If $L$ is given as in Eq. (\ref{e.2.1}), then coefficients
$\left\{  A_{k}\right\}  _{k=0}^{2mn}$ of $L^{n}$ in Eq. (\ref{e.3.1}) are
given by
\begin{equation}
A_{k}=1_{k\in2\mathbb{N}_{0}}\cdot\left(  -1\right)  ^{k/2}\mathcal{B}%
_{k/2}+T_{k}, \label{e.3.14}%
\end{equation}
where $\mathcal{B}_{\ell}$ and $T_{k}$ are as in Notations \ref{not.3.2} and
\ref{not.3.8} respectively. Moreover, if we further assume $\left\{
b_{l}\right\}  _{l=0}^{m}$ are polynomial functions such that
\begin{equation}
\deg\left(  b_{l}\right)  \leq\max\left\{  \deg\left(  b_{0}\right)
,0\right\}  \text{ for }1\leq\ell\leq m, \label{e.3.15}%
\end{equation}
then $\left\{  \mathcal{B}_{\ell}\right\}  _{\ell=0}^{mn}$ and $\left\{
T_{k}\right\}  _{k=0}^{2mn}$ are polynomials such that
\[
\deg\left(  T_{k}\right)  <\max\left\{  n\deg\left(  b_{0}\right)  ,0\right\}
=\max\left\{  \deg\left(  \mathcal{B}_{0}\right)  ,0\right\}  \text{ for
}0\leq k\leq2mn.
\]

\end{proposition}

\begin{proof}
First observe that if $L^{n}$ is described as in Eq. (\ref{e.3.1}), then
\[
\sum_{k=0}^{2mn}A_{k}\left(  x\right)  \left(  i\xi\right)  ^{k}=\sigma
_{n}\left(  x,\xi\right)  =e^{-i\xi x}L_{x}^{n}\left(  e^{i\xi x}\right)
\]
where $\sigma_{n}:=\sigma_{L^{n}}$ is a symbol of $L^{n}$ defined in Eq.
(\ref{e.1.3}) and $L_{x}^{n}$ is a differential operator with respect to $x.$
To compute the right side of this equation, take the $n^{\text{th}}$ -- power
of Eq. (\ref{e.3.10}) to learn
\begin{align*}
&  M_{e^{-i\xi\left(  \cdot\right)  }}L^{n}M_{e^{i\xi\left(  \cdot\right)  }%
}=\left(  M_{e^{-i\xi\left(  \cdot\right)  }}LM_{e^{i\xi\left(  \cdot\right)
}}\right)  ^{n}\\
&  =\sum_{q_{1},\dots q_{n}=0}^{m}\left(  \partial+i\xi\right)  ^{q_{n}%
}\left(  -1\right)  ^{q_{n}}M_{b_{q_{n}}}\left(  \partial+i\xi\right)
^{q_{n}}\dots\left(  \partial+i\xi\right)  ^{q_{1}}\left(  -1\right)  ^{q_{1}%
}M_{b_{q_{1}}}\left(  \partial+i\xi\right)  ^{q_{1}}\\
&  =\sum_{\mathbf{q\in\Lambda}_{m}}\left(  -1\right)  ^{\left\vert
\mathbf{q}\right\vert }\left(  \partial+i\xi\right)  ^{q_{n}}M_{b_{q_{n}}%
}\left(  \partial+i\xi\right)  ^{q_{n}}\dots\left(  \partial+i\xi\right)
^{q_{1}}M_{b_{q_{1}}}\left(  \partial+i\xi\right)  ^{q_{1}}.
\end{align*}
Applying this result to the constant function $\mathbf{1}$ then shows
\begin{align*}
\sigma_{n}\left(  x,\xi\right)   &  =e^{-i\xi x}L_{x}^{n}\left(  e^{i\xi
x}\right)  =M_{e^{-i\xi x}}L_{x}^{n}M_{e^{i\xi x}}\mathbf{1}\\
&  =\sum_{\mathbf{q\in\Lambda}_{m}}\left(  -1\right)  ^{\left\vert
\mathbf{q}\right\vert }\left(  \partial+i\xi\right)  ^{q_{n}}M_{b_{q_{n}}%
}\left(  \partial+i\xi\right)  ^{q_{n}}\dots\left(  \partial+i\xi\right)
^{q_{1}}M_{b_{q_{1}}}\left(  \partial+i\xi\right)  ^{q_{1}}\mathbf{1.}%
\end{align*}
Making repeatedly used of the binomial formula to expand out all the terms
$\left(  \partial+i\xi\right)  ^{q}$ appearing above then gives,
\begin{align*}
\sum_{k=0}^{2mn}  &  A_{k}\left(  x\right)  \left(  i\xi\right)  ^{k}%
=\sigma_{n}\left(  x,\xi\right) \\
&  =\sum_{\mathbf{q,l,j\in\Lambda}_{m}}\left(  -1\right)  ^{\left\vert
\mathbf{q}\right\vert }\binom{\mathbf{q}}{\mathbf{l}}\binom{\mathbf{q}%
}{\mathbf{j}}\left(  i\xi\right)  ^{2\left\vert \mathbf{q}\right\vert
-\left\vert \mathbf{l}\right\vert -\left\vert \mathbf{j}\right\vert }%
\partial_{x}^{l_{n}}b_{q_{n}}\left(  x\right)  \partial_{x}^{j_{n}}%
\dots\partial_{x}^{l_{1}}b_{q_{1}}\left(  x\right)  \partial_{x}^{j_{1}%
}\mathbf{1.}%
\end{align*}
Looking the coefficient of $\left(  i\xi\right)  ^{k}$ on the right side of
this expression shows,
\begin{align*}
A_{k}\left(  x\right)  :=  &  \sum_{\mathbf{q,l,j\in\Lambda}_{m}}%
\binom{\mathbf{q}}{\mathbf{l}}\binom{\mathbf{q}}{\mathbf{j}}\left(  -1\right)
^{\left\vert \mathbf{q}\right\vert }1_{2\left\vert \mathbf{q}\right\vert
-\left\vert \mathbf{l}\right\vert -\left\vert \mathbf{j}\right\vert
=k}\partial_{x}^{l_{n}}b_{q_{n}}\left(  x\right)  \partial_{x}^{j_{n}}%
\dots\partial_{x}^{l_{1}}b_{q_{1}}\left(  x\right)  \partial_{x}^{j_{1}%
}\mathbf{1}\\
=  &  \sum_{\mathbf{q,l,j\in\Lambda}_{m}}\binom{\mathbf{q}}{\mathbf{l}}%
\binom{\mathbf{q}}{\mathbf{j}}\left(  -1\right)  ^{\left\vert \mathbf{q}%
\right\vert }1_{j_{1}=0}1_{2\left\vert \mathbf{q}\right\vert -\left\vert
\mathbf{l}\right\vert -\left\vert \mathbf{j}\right\vert =k}\partial_{x}%
^{l_{n}}b_{q_{n}}\left(  x\right)  \partial_{x}^{j_{n}}\dots\partial
_{x}^{l_{1}}b_{q_{1}}\left(  x\right) \\
=  &  \sum_{\mathbf{q\in\Lambda}_{m}}1_{2\left\vert \mathbf{q}\right\vert
=k}\left(  -1\right)  ^{\left\vert \mathbf{q}\right\vert }b_{q_{n}}\left(
x\right)  \dots b_{q_{1}}\left(  x\right) \\
&  +\sum_{\mathbf{q,l,j\in\Lambda}_{m}}\binom{\mathbf{q}}{\mathbf{l}}%
\binom{\mathbf{q}}{\mathbf{j}}\left(  -1\right)  ^{\left\vert \mathbf{q}%
\right\vert }1_{j_{1}=0}1_{2\left\vert \mathbf{q}\right\vert -k=\left\vert
\mathbf{l}\right\vert +\left\vert \mathbf{j}\right\vert >0}\partial_{x}%
^{l_{n}}b_{q_{n}}\left(  x\right)  \partial_{x}^{j_{n}}\dots\partial
_{x}^{l_{1}}b_{q_{1}}\left(  x\right)
\end{align*}
which completes the proof of Eq. (\ref{e.3.14}). The remaining assertions now
easily follow from the formulas for $\left\{  \mathcal{B}_{\ell}\right\}
_{\ell=0}^{mn}$ and $\left\{  T_{k}\right\}  _{k=0}^{2mn}$ in Notations
\ref{not.3.2} and \ref{not.3.8} and the assumption in Eq. (\ref{e.3.15}).
\end{proof}

As we can see from Example \ref{exa.3.4}, computing the coefficients $\left\{
B_{\ell}\right\}  _{\ell=0}^{nm}$ in Eq. (\ref{e.3.1}) can be tedious in terms
of in terms of the coefficients $\left\{  b_{l}\right\}  _{l=0}^{m}$ defining
$L$ as in Eq. (\ref{e.2.1}). Although we do not need the explicit formula for
the $\left\{  B_{\ell}\right\}  _{\ell=0}^{nm},$ we will need some general
properties of these coefficients which we develop below.

\begin{proposition}
\label{pro.3.10} Given $B_{\ell}=B_{\ell}\left(  \left\{  b_{l}\right\}
_{l=0}^{m}\right)  $ of $L^{n}$ in Eq. (\ref{e.3.1}). There are constants
$\hat{C}\left(  n,\ell,\mathbf{k},\mathbf{p}\right)  $ for $n\in
\mathbb{\mathbb{N}}_{0}$, $0\leq\ell\leq mn$, $\mathbf{k}\in\Lambda_{m}^{n}$
and $\mathbf{p\in}\Lambda_{2m}^{n}$ such that;

\begin{enumerate}
\item $\hat{C}\left(  n,\ell,\mathbf{k},\mathbf{p}\right)  =0$ unless
$0<\left\vert \mathbf{p}\right\vert =2\left\vert \mathbf{k}\right\vert -2\ell$ and

\item if $L=\sum_{l=0}^{m}\left(  -1\right)  ^{l}\partial^{l}b_{l}\partial
^{l}$ then $L^{n}=\sum_{\ell=0}^{mn}\left(  -1\right)  ^{\ell}\partial^{\ell
}B_{\ell}\partial^{\ell},$ where
\begin{equation}
B_{\ell}=\mathcal{B}_{\ell}+R_{\ell} \label{e.3.16}%
\end{equation}
with $\mathcal{B}_{\ell}$ as in Eq. (\ref{e.3.2}) and $R_{\ell}$ is defined
by
\begin{equation}
R_{\ell}=\sum_{\mathbf{k}\in\Lambda_{m}^{n},\mathbf{p\in}\Lambda_{2m}^{n}}%
\hat{C}\left(  n,\ell,\mathbf{k},\mathbf{p}\right)  \left(  \partial^{p_{1}%
}b_{k_{1}}\right)  \left(  \partial^{p_{2}}b_{k_{2}}\right)  \dots\left(
\partial^{p_{n}}b_{k_{n}}\right)  . \label{e.3.17}%
\end{equation}
[Notice that $2\left\vert \mathbf{k}\right\vert -2\ell\leq2mn-2\ell$ and so if
$\ell=mn,$ we must have $\left\vert \mathbf{p}\right\vert =2\left\vert
\mathbf{k}\right\vert -2\ell=0$ and so $\hat{C}\left(  n,\ell,\mathbf{k}%
,\mathbf{p}\right)  =0.$ This shows that $R_{mn}=0$ which can easily be
verified independently if the reader so desires.]
\end{enumerate}
\end{proposition}

\begin{proof}
By Theorem \ref{the.2.7}, we know that $L^{n}=\sum_{\ell=0}^{mn}\left(
-1\right)  ^{\ell}\partial^{\ell}B_{\ell}\partial^{\ell}$ where
\begin{equation}
\left(  -1\right)  ^{\ell}B_{\ell}=A_{2\ell}+\sum_{\ell<s\leq mn}K_{mn}%
(\ell,s)\partial^{2(s-\ell)}A_{2s}\text{ }\forall~0\leq\ell\leq mn
\label{e.3.18}%
\end{equation}
and $\left\{  A_{k}\right\}  _{k=0}^{2mn}$ are the coefficients in Eq.
(\ref{e.3.1}). Using the formula for the $\left\{  A_{k}\right\}  $ from
Proposition \ref{pro.3.9} in Eq. (\ref{e.3.18}) implies,
\[
\left(  -1\right)  ^{\ell}B_{\ell}=\left(  -1\right)  ^{\ell}\mathcal{B}%
_{\ell}+T_{2\ell}+\sum_{\ell<s\leq mn}K_{mn}(\ell,s)\partial^{2(s-\ell
)}\left[  \left(  -1\right)  ^{s}\mathcal{B}_{s}+T_{2s}\right]  ,
\]
i.e. $B_{\ell}=\mathcal{B}_{\ell}+R_{\ell}$ where
\[
R_{\ell}=\left(  -1\right)  ^{\ell}T_{2\ell}+\sum_{\ell<s\leq mn}\left(
-1\right)  ^{\ell}K_{mn}(\ell,s)\partial^{2(s-\ell)}\left[  \left(  -1\right)
^{s}\mathcal{B}_{s}+T_{2s}\right]  .
\]
It now only remains to see that this remainder term may be written as in Eq.
(\ref{e.3.17}).

Recall from Eq. (\ref{e.3.2}) that $\mathcal{B}_{s}:=\sum_{\mathbf{j}%
\in\Lambda_{m}^{n}}1_{\left\vert \mathbf{j}\right\vert =s}b_{j_{1}}\dots
b_{j_{n}}$ and so
\[
\partial^{2(s-\ell)}\mathcal{B}_{s}=\sum_{\mathbf{j}\in\Lambda_{m}^{n}%
}1_{\left\vert \mathbf{j}\right\vert =s}\partial^{2(s-\ell)}\left[  b_{j_{1}%
}\dots b_{j_{n}}\right]  .
\]
For $s>\ell,$ $\partial^{2(s-\ell)}\left[  b_{j_{1}}\dots b_{j_{n}}\right]  $
is a linear combination of terms of the form,
\[
\left(  \partial^{p_{1}}b_{j_{1}}\right)  \left(  \partial^{p_{2}}b_{j_{2}%
}\right)  \dots\left(  \partial^{p_{n}}b_{j_{n}}\right)
\]
where $0<\left\vert \mathbf{p}\right\vert =2(s-\ell)=2\left\vert
\mathbf{j}\right\vert -2\ell$ as desired. Similarly, from Eq. (\ref{e.3.13}),
$T_{2s}$ is a linear combination of monomials of the form,
\[
\partial^{l_{n}}M_{b_{q_{n}}}\partial^{j_{n}}\dots\partial^{l_{2}}M_{b_{q_{n}%
}}\partial^{j_{2}}\partial^{l_{1}}b_{q_{1}}\text{ with }2\left\vert
\mathbf{q}\right\vert -2s=\left\vert \mathbf{l}\right\vert +\left\vert
\mathbf{j}\right\vert >0\text{ and }j_{1}=0.
\]
It then follows that
\[
\partial^{2(s-\ell)}\partial^{l_{n}}M_{b_{q_{n}}}\partial^{j_{n}}\dots
\partial^{l_{2}}M_{b_{q_{n}}}\partial^{j_{2}}\partial^{l_{1}}b_{q_{1}}
\]
is a linear combination of monomials of the form,
\[
\left(  \partial^{p_{1}}b_{q_{1}}\right)  \left(  \partial^{p_{2}}b_{q_{2}%
}\right)  \dots\left(  \partial^{p_{n}}b_{q_{n}}\right)
\]
where
\[
0<\left\vert \mathbf{p}\right\vert =2(s-\ell)+\left\vert \mathbf{l}\right\vert
+\left\vert \mathbf{j}\right\vert =2(s-\ell)+2\left\vert \mathbf{q}\right\vert
-2s=2\left\vert \mathbf{q}\right\vert -2\ell.
\]
Putting all of these comments together completes the proof.
\end{proof}

\section{The Essential Self Adjointness Proof\label{sec.4}}

This section is devoted to the proof of Theorem \ref{the.1.10}. Lemma
\ref{lem.4.2} records a simple sufficient condition for showing a symmetric
operator on a Hilbert space is in fact essentially self-adjoint. For the
remainder of this paper, we assume that the coefficients, $\left\{
a_{k}\right\}  _{k=0}^{d},$ of $L$ in Eq. (\ref{e.1.2}) are all in
$\mathbb{R}\left[  x\right]  $ and we now restrict $L$ to $\mathcal{S}$ as
described in Notation \ref{not.1.9}. The operators, $L^{n},$ are then defined
for all $n\in\mathbb{N}$ and we still have $\mathcal{D}\left(  L^{n}\right)
=\mathcal{S}$, see Remark \ref{rem.1.8}.

\begin{lemma}
[Self-Adjointness Criteria]\label{lem.4.2}Let $\mathbb{L}:\mathcal{K}%
\rightarrow\mathcal{K}$ be a densely defined symmetric operator on a Hilbert
space $\mathcal{K}$ and let $S=\mathcal{D}\left(  \mathbb{L}\right)  $ be the
domain of $\mathbb{L}.$ Assume there exists linear operators $T_{\mu
}:S\rightarrow S$ and bounded operators $R_{\mu}:\mathcal{K}\rightarrow
\mathcal{K}$ for $\mu\in\mathbb{R}$ such that;

\begin{enumerate}
\item $(\mathbb{L}+i\mu)T_{\mu}u=(I+R_{\mu})u$ for all $u\in S$, and

\item there exists $M>0$ such that $\left\Vert R_{\mu}\right\Vert _{op}<1$ for
$\left\vert \mu\right\vert >M$.
\end{enumerate}

Under these assumptions, $\mathbb{L}|_{S}$ is essentially self-adjoint.
\end{lemma}

\begin{proof}
$\left\Vert R_{\mu}\right\Vert _{op}<1$ for $\left\vert \mu\right\vert >M$ is
assumed in condition $2$ which implies $I+R_{\mu}$ is invertible. Therefore,
if $f\in\mathcal{K},$ then $g:=\left(  I+R_{\mu}\right)  ^{-1}f\in\mathcal{K}$
satisfies $\left(  I+R_{\mu}\right)  g=f.$ We may then choose $\left\{
g_{n}\right\}  _{n=1}^{\infty}\subset S$ such that $g_{n}\rightarrow g$ in
$\mathcal{K}.$ Let $s_{n}=T_{\mu}g_{n}\in S.$ We have, by condition $1$, that
\begin{align*}
\left\Vert \left(  \mathbb{L}+i\mu\right)  s_{n}-f\right\Vert  &  =\left\Vert
(I+R_{\mu})g_{n}-f\right\Vert \\
&  \leq||(I+R_{\mu})\left(  g_{n}-g\right)  ||\\
&  \leq\left\Vert I+R_{\mu}\right\Vert _{op}||g_{n}-g||\rightarrow0\text{ as
}n\rightarrow\infty.
\end{align*}
We have thus verified that $\operatorname*{Ran}(\mathbb{L}+i\mu)|_{S}$ is
dense in $K$ for all $\left\vert \mu\right\vert >M$ from which it follows that
$\mathbb{L}|_{S}$ is essentially self-adjoint, see for example the corollary
on p.257 in Reed and Simon \citep{Reed1980}.
\end{proof}

\begin{notation}
\label{not.4.4}Let $\left\{  \mathcal{B}_{\ell}\left(  x\right)  \right\}
_{\ell=0}^{mn}$ be the coefficients defined in Eq. (\ref{e.3.2}) and define
\begin{equation}
\Sigma\left(  x,\xi\right)  :=\sum_{\ell=0}^{mn}\mathcal{B}_{\ell}\left(
x\right)  \xi^{2\ell}. \label{e.4.2}%
\end{equation}

\end{notation}

From Eqs. (\ref{e.3.14}) and (\ref{e.1.3}) the symbol, $\sigma_{n}\left(
x,\xi\right)  :=\sigma_{L^{n}}\left(  x,\xi\right)  ,$ of $L^{n}$ presented as
in Eq. (\ref{e.3.1}) may be written as
\begin{align}
\sigma_{n}\left(  x,\xi\right)   &  =\sum_{\ell=0}^{mn}\left[  \mathcal{B}%
_{\ell}\left(  x\right)  +\left(  -1\right)  ^{\ell}T_{2\ell}\left(  x\right)
\right]  \xi^{2\ell}-i\sum_{\ell=1}^{mn}\left(  -1\right)  ^{\ell}T_{2\ell
-1}\left(  x\right)  \xi^{2\ell-1}\label{e.4.3}\\
&  =\Sigma\left(  x,\xi\right)  +\sum_{\ell=0}^{mn}\left[  \left(  -1\right)
^{\ell}T_{2\ell}\left(  x\right)  \right]  \xi^{2\ell}-i\sum_{\ell=1}%
^{mn}\left(  -1\right)  ^{\ell}T_{2\ell-1}\left(  x\right)  \xi^{2\ell-1},
\label{e.4.4}%
\end{align}
where the coefficients $\left\{  T_{k}\right\}  _{k=0}^{2mn}$ are as in Eq.
(\ref{e.3.13}). More importantly, for our purposes,
\begin{equation}
\operatorname{Re}\sigma_{n}\left(  x,\xi\right)  =\Sigma\left(  x,\xi\right)
+\sum_{\ell=0}^{mn}\left[  \left(  -1\right)  ^{\ell}T_{2\ell}\left(
x\right)  \right]  \xi^{2\ell}. \label{e.4.5}%
\end{equation}
The following lemma will be useful in estimating all of these functions of
$\left(  x,\xi\right)  .$

\begin{lemma}
\label{lem.4.5}Let $0\leq k_{1}<k_{2}<\infty$ and $p(x)$, $q(x)$ and $r(x)$ be
real polynomials such that $\deg p\leq\deg q,$ $q>0,$ and $r$ is bounded from below.

\begin{enumerate}
\item If $\deg\left(  p\right)  <\deg\left(  r\right)  $ or $p$ is a constant
function, then, for any $k_{1}<k_{2}$ and $\lambda>0$, there exists
$c_{\lambda}$ such that
\begin{equation}
\left\vert p\left(  x\right)  \xi^{k_{1}}\right\vert \leq\lambda\left(
q\left(  x\right)  \left\vert \xi\right\vert ^{k_{2}}+r\left(  x\right)
\right)  +c_{\lambda}. \label{e.4.6}%
\end{equation}

\item If $\deg\left(  p\right)  \leq\deg\left(  r\right)  ,$ then for any
$k_{1}<k_{2}$ and $\lambda>0$, there exists constants $c_{\lambda}$ and
$d_{\lambda}$ such that
\begin{equation}
\left\vert p\left(  x\right)  \xi^{k_{1}}\right\vert \leq\lambda q\left(
x\right)  \left\vert \xi\right\vert ^{k_{2}}+c_{\lambda}r\left(  x\right)
+d_{\lambda}. \label{e.4.7}%
\end{equation}

\end{enumerate}
\end{lemma}

\begin{proof}
Since $\deg p\leq\deg q$ and $q>0,$ $\deg q\in2\mathbb{N}_{0}$ and
$K:=\sup_{x\in\mathbb{R}}\left\vert p\left(  x\right)  \right\vert /q\left(
x\right)  <\infty,$ i.e. $p\left(  x\right)  \leq Kq\left(  x\right)  .$ One
also has for every $\tau>0,$ there exists $0<a_{\tau}<\infty$ such $\left\vert
\xi^{k_{1}}\right\vert \leq\tau\left\vert \xi\right\vert ^{k_{2}}+a_{\tau}.$
Combining these estimates shows,
\begin{align}
\left\vert p\left(  x\right)  \xi^{k_{1}}\right\vert  &  \leq\tau\left\vert
p\left(  x\right)  \right\vert \left\vert \xi\right\vert ^{k_{2}}+a_{\tau
}\left\vert p\left(  x\right)  \right\vert \label{e.4.8}\\
&  \leq\tau Kq\left(  x\right)  \left\vert \xi\right\vert ^{k_{2}}+a_{\tau
}\left\vert p\left(  x\right)  \right\vert . \label{e.4.9}%
\end{align}

If $\deg p<\deg r,$ for every $\delta>0$ there exists $0<b_{\delta}<\infty$
such that $\left\vert p\left(  x\right)  \right\vert \leq\delta r\left(
x\right)  +b_{\delta}$ which combined with Eq. (\ref{e.4.9}) implies
\[
\left\vert p\left(  x\right)  \xi^{k_{1}}\right\vert \leq\tau Kq\left(
x\right)  \left\vert \xi\right\vert ^{k_{2}}+a_{\tau}\left(  \delta r\left(
x\right)  +b_{\delta}\right)
\]
and Eq. (\ref{e.4.6}) follows by choosing $\tau=\lambda/K$ and then
$\delta=\lambda/a_{\tau}$ so that $c_{\lambda}=a_{\tau}b_{\delta}.$

If $\deg p\leq\deg r,$ then there exists $C_{1},C_{2}<\infty$ such that
$\left\vert p\left(  x\right)  \right\vert \leq C_{1}r\left(  x\right)
+C_{2}$ which combined with Eq. (\ref{e.4.9}) with $\tau=\lambda/K$ shows
\[
\left\vert p\left(  x\right)  \xi^{k_{1}}\right\vert \leq\lambda q\left(
x\right)  \left\vert \xi\right\vert ^{k_{2}}+a_{\lambda/K}\left(
C_{1}r\left(  x\right)  +C_{2}\right)
\]
from which Eq. (\ref{e.4.7}) follows.
\end{proof}

With the use of Lemma \ref{lem.4.5}, the following Lemma helps us to estimate
the growth of $T_{k}\left(  x\right)  $ (see Notation \ref{not.3.8}) and its
derivatives of $L^{n}$ in Eq. (\ref{e.3.1}) for $0\leq k\leq2mn$

\begin{lemma}
\label{lem.4.6}Suppose that $\left\{  b_{l}\right\}  _{l=0}^{m}$ are
polynomials satisfying the assumptions in Theorem \ref{the.1.10}. Then for
each $0\leq k\leq2mn,$ $\beta\in\mathbb{N}_{0},$ and $\delta>0$, there exists
$C=C\left(  k,\beta,\delta\right)  <\infty$ such that
\begin{equation}
\left(  1+\left\vert \xi\right\vert ^{k}\right)  \left(  1+\left\vert
x\right\vert ^{\beta}\right)  \left\vert \partial_{x}^{\beta}T_{k}\left(
x\right)  \right\vert \leq\delta\Sigma\left(  x,\xi\right)  +C\left(
k,\beta,\delta\right)  . \label{e.4.10}%
\end{equation}

\end{lemma}

\begin{proof}
If $\deg b_{0}\leq0,$ then by condition 2 in Theorem \ref{the.1.10} it follows
that $\left\{  b_{l}\right\}  _{l=0}^{m}$ are all constant in which case
$T_{k}\equiv0$ and the Lemma is trivial. So for the rest of the proof we
assume $\deg b_{0}>0.$

According to Eq. (\ref{e.3.13}), $T_{k}$ may be expressed as a linear
combination of terms of the form, $\left(  \partial^{p_{1}}b_{j_{1}}\right)
\dots\left(  \partial^{p_{n}}b_{j_{n}}\right)  ,$ where $\mathbf{j}$ and
$\mathbf{p}$ are multi-indices such that $2\left\vert \mathbf{j}\right\vert
-k=\left\vert \mathbf{p}\right\vert >0.$ If $\mathbf{j}$ and $\mathbf{p}$ are
multi-indices such that $2\left\vert \mathbf{j}\right\vert -k=\left\vert
\mathbf{p}\right\vert >0$ and $\left(  \partial^{p_{1}}b_{j_{1}}\right)
\dots\left(  \partial^{p_{n}}b_{j_{n}}\right)  \neq0,$ then $b_{j_{1}}\dots
b_{j_{n}}$ is strictly positive and
\[
\deg\left(  b_{j_{1}}\dots b_{j_{n}}\right)  \geq\deg\left[  \left(
\partial^{p_{1}}b_{j_{1}}\right)  \dots\left(  \partial^{p_{n}}b_{j_{n}%
}\right)  \right]  +\left\vert \mathbf{p}\right\vert >0.
\]
Given the term, $b_{j_{1}}\dots b_{j_{n}},$ appears in $\mathcal{B}%
_{\left\vert \mathbf{j}\right\vert },$ we conclude that $\mathcal{B}%
_{\left\vert \mathbf{j}\right\vert }$ is strictly positive and
\[
\deg\left(  \left(  \partial^{p_{1}}b_{j_{1}}\right)  \dots\left(
\partial^{p_{n}}b_{j_{n}}\right)  \right)  <\deg\left(  b_{j_{1}}\dots
b_{j_{n}}\right)  \leq\deg\left(  \mathcal{B}_{\left\vert \mathbf{j}%
\right\vert }\right)  .
\]
Moreover from condition 2 in Theorem \ref{the.1.10}, $\deg b_{j}\leq\deg
b_{0}$ for all $j$ and therefore we also have
\[
\deg\left(  \left(  \partial^{p_{1}}b_{j_{1}}\right)  \dots\left(
\partial^{p_{n}}b_{j_{n}}\right)  \right)  <\deg\left(  b_{j_{1}}\dots
b_{j_{n}}\right)  \leq\deg\left(  b_{0}^{n}\right)  =\deg\mathcal{B}_{0}.
\]

Moreover, for any $r,\beta\in\mathbb{N}_{0}$ with $r\leq\beta$ we still have
\begin{align*}
\text{$\deg$}\left\{  x^{r}\partial_{x}^{\beta}\left[  \left(  \partial
^{p_{1}}b_{j_{1}}\right)  \dots\left(  \partial^{p_{n}}b_{j_{n}}\right)
\right]  \right\}   &  \leq\deg\left(  \left(  \partial^{p_{1}}b_{j_{1}%
}\right)  \dots\left(  \partial^{p_{n}}b_{j_{n}}\right)  \right) \\
&  <\min\left\{  \deg\left(  \mathcal{B}_{\left\vert \mathbf{j}\right\vert
}\right)  ,\deg\left(  \mathcal{B}_{0}\right)  \right\}  .
\end{align*}
Hence by substituting
\[
p\left(  x\right)  =x^{r}\partial_{x}^{\beta}\left[  \left(  \partial^{p_{1}%
}b_{j_{1}}\right)  \dots\left(  \partial^{p_{n}}b_{j_{n}}\right)  \right]
\text{, }q\left(  x\right)  =\mathcal{B}_{\left\vert \mathbf{j}\right\vert
}\left(  x\right)  \text{ and }r\left(  x\right)  =\mathcal{B}_{0}\left(
x\right)
\]
in Lemma \ref{lem.4.5}, for every $\lambda>0$ there exists $C_{\lambda}%
<\infty$ such that
\begin{align*}
\left\vert x^{r}\partial_{x}^{\beta}\left[  \left(  \partial^{p_{1}}b_{j_{1}%
}\right)  \dots\left(  \partial^{p_{n}}b_{j_{n}}\right)  \right]  \xi
^{k}\right\vert  &  \leq\lambda\left[  \mathcal{B}_{\left\vert \mathbf{j}%
\right\vert }\left(  x\right)  \xi^{2\left\vert \mathbf{j}\right\vert
}+\mathcal{B}_{0}\left(  x\right)  \right]  +C_{\lambda}\\
&  \leq\lambda\cdot\Sigma\left(  x,\xi\right)  +C_{\lambda}%
\end{align*}
and similarly,
\[
\left\vert x^{r}\partial_{x}^{\beta}\left[  \left(  \partial^{p_{1}}b_{j_{1}%
}\right)  \dots\left(  \partial^{p_{n}}b_{j_{n}}\right)  \right]  \right\vert
\leq\lambda\Sigma\left(  x,\xi\right)  +C_{\lambda}.
\]
These last two equations with $r=0$ and $r=\beta$ combine to show, for all
$\lambda>0,$ there exists $C_{\lambda}<\infty$ such that
\[
\left(  1+\left\vert \xi\right\vert ^{k}\right)  \left(  1+\left\vert
x\right\vert ^{\beta}\right)  \left\vert \partial_{x}^{\beta}\left[  \left(
\partial^{p_{1}}b_{j_{1}}\right)  \dots\left(  \partial^{p_{n}}b_{j_{n}%
}\right)  \right]  \right\vert \leq4\left(  \lambda\Sigma\left(  x,\xi\right)
+C_{\lambda}\right)  .
\]
By using this result in Eq. (\ref{e.3.13}), one then sees there is a constant
$K<\infty$ such that
\[
\left(  1+\left\vert \xi\right\vert ^{k}\right)  \left(  1+\left\vert
x\right\vert ^{\beta}\right)  \left\vert \partial_{x}^{\beta}T_{k}\left(
x\right)  \right\vert \leq K\lambda\Sigma\left(  x,\xi\right)  +KC_{\lambda}.
\]
Equation (\ref{e.4.10}) now follows by replacing $\lambda$ by $\delta/K$ in
the above equation.
\end{proof}

The following Lemma is to study the growth of $\mathcal{B}_{\ell}\left(
x\right)  $ (see Notation \ref{not.3.2}) and its derivatives of $L^{n}$ in Eq.
(\ref{e.3.1}) for $0\leq l\leq mn$

\begin{lemma}
\label{lem.4.7}Again suppose that $\left\{  b_{l}\right\}  _{l=0}^{m}$ are
polynomials satisfying the assumptions in Theorem \ref{the.1.10}. For all
$\ell\in\Lambda_{mn},$ and $\beta\in\mathbb{N}_{0},\ $there exists $C<\infty$
such that
\begin{equation}
\left\vert \partial_{x}^{\beta}\mathcal{B}_{\ell}\left(  x\right)  \right\vert
\left(  \left\vert \xi\right\vert ^{2\ell}+1\right)  \left(  \left\vert
x\right\vert ^{\beta}+1\right)  \leq C\Sigma\left(  x,\xi\right)  +C.
\label{e.4.11}%
\end{equation}
Moreover, if we assume $b_{0}$ is not the zero polynomial, then we may drop
the second $C$ in Eq. (\ref{e.4.11}), i.e. there exists $C<\infty$ such that
\begin{equation}
\left\vert \partial_{x}^{\beta}\mathcal{B}_{\ell}\left(  x\right)  \right\vert
\left(  \left\vert \xi\right\vert ^{2\ell}+1\right)  \left(  \left\vert
x\right\vert ^{\beta}+1\right)  \leq C\Sigma\left(  x,\xi\right)  .
\label{e.4.12}%
\end{equation}

\end{lemma}

\begin{proof}
Case 1. If $b_{0}=0$ then by the assumption 2 of Theorem \ref{the.1.10} each
$b_{l}$ is a constant for $1\leq l\leq m$ and therefore $\partial_{x}^{\beta
}\mathcal{B}_{\ell}\left(  x\right)  =0$ for all $\beta>0,$ i.e.
$\mathcal{B}_{\ell}$ are constant for all $\ell.$ Moreover, if $\beta=0,$ from
the definition of $\Sigma\left(  x,\xi\right)  $ in Eq. (\ref{e.4.2}) it
follows that $\mathcal{B}_{\ell}\xi^{2\ell}\leq\Sigma\left(  x,\xi\right)  $
and hence
\[
\left\vert \mathcal{B}_{\ell}\left(  x\right)  \right\vert \left(  \left\vert
\xi\right\vert ^{2\ell}+1\right)  =\mathcal{B}_{\ell}\left(  \xi^{2\ell
}+1\right)  \leq\Sigma\left(  x,\xi\right)  +\mathcal{B}_{\ell}
\]
and so Eq. (\ref{e.4.11}) holds with $C=\max_{1\leq\ell\leq mn}\max\left(
1,\mathcal{B}_{\ell}\right)  .$

Case 2. If $b_{0}\neq0,$ let us assume $\mathcal{B}_{\ell}$ is not the zero
polynomial for otherwise there is nothing to prove. Since $x^{\beta}%
\partial_{x}^{\beta}\mathcal{B}_{\ell}\left(  x\right)  $ and $\partial
_{x}^{\beta}\mathcal{B}_{\ell}$ $\left(  x\right)  $ both have degree no more
than $\deg\mathcal{B}_{\ell}$ and $\mathcal{B}_{\ell}>0,$ we may conclude
there exists $C<\infty$ such that
\begin{equation}
\left(  1+\left\vert x\right\vert ^{\beta}\right)  \left\vert \partial
_{x}^{\beta}\mathcal{B}_{\ell}\left(  x\right)  \right\vert =\left\vert
\partial_{x}^{\beta}\mathcal{B}_{\ell}\left(  x\right)  \right\vert
+\left\vert x^{\beta}\partial_{x}^{\beta}\mathcal{B}_{\ell}\left(  x\right)
\right\vert \leq C\mathcal{B}_{\ell}\left(  x\right)  . \label{e.4.13}%
\end{equation}
Multiplying this equation by $\xi^{2\ell}$ then shows,
\begin{equation}
\left(  1+\left\vert x\right\vert ^{\beta}\right)  \left\vert \partial
_{x}^{\beta}\mathcal{B}_{\ell}\left(  x\right)  \right\vert \xi^{2\ell}\leq
C\mathcal{B}_{\ell}\left(  x\right)  \xi^{2\ell}\leq C\Sigma\left(
x,\xi\right)  \label{e.4.14}%
\end{equation}
while $\deg\left(  \mathcal{B}_{\ell}\right)  \leq\deg\left(  \mathcal{B}%
_{0}\right)  $ and $\mathcal{B}_{0}>0$ , then there exists $C_{1}<\infty$ such
that
\[
\mathcal{B}_{\ell}\left(  x\right)  \leq C_{1}\mathcal{B}_{0}\left(  x\right)
\leq C_{1}\Sigma\left(  x,\xi\right)
\]
which combined with Eq. (\ref{e.4.13}) shows
\[
\left(  1+\left\vert x\right\vert ^{\beta}\right)  \left\vert \partial
_{x}^{\beta}\mathcal{B}_{\ell}\left(  x\right)  \right\vert \leq C_{1}%
\Sigma\left(  x,\xi\right)  .
\]
This estimate along with Eq. (\ref{e.4.14}) then completes the proof of Eq.
(\ref{e.4.11}) with no second $C.$
\end{proof}

\begin{notation}
\label{not.4.8}For any non-negative real-valued functions $f$ and $g$ on some
domain $U$, we write $f\lesssim g$ to mean there exists $C>0$ such that
$f\left(  y\right)  \leq Cg\left(  y\right)  $ for all $y\in U$.
\end{notation}

The following result is an immediate corollary of Proposition \ref{pro.3.9}
and Lemmas \ref{lem.4.6} and \ref{lem.4.7}.

\begin{corollary}
\label{cor.4.9}Suppose that $\left\{  b_{l}\right\}  _{l=0}^{m}$ are
polynomials satisfying the assumptions in Theorem \ref{the.1.10}. If $\left\{
A_{k}\right\}  _{k=0}^{2mn}$ are the coefficients of $L^{n}$ as in Eq.
(\ref{e.3.1}), then for all $\beta\in\mathbb{N}_{0}$ and $0\leq k\leq2mn,$
\begin{equation}
\left\vert \partial_{x}^{\beta}A_{k}\left(  x\right)  \right\vert \left(
1+\left\vert \xi\right\vert ^{k}\right)  \left(  1+\left\vert x\right\vert
^{\beta}\right)  \lesssim\Sigma\left(  x,\xi\right)  +1. \label{e.4.15}%
\end{equation}

\end{corollary}

The next lemma is a direct consequence of Lemmas \ref{lem.4.6}.

\begin{lemma}
\label{lem.4.10}Let $L$ be the operator in Eq. (\ref{e.1.6}), where we now
assume that $\left\{  b_{l}\right\}  _{l=0}^{m}$ are polynomials satisfying
the assumptions of Theorem \ref{the.1.10}. Then there exists $c>0$ such that
the following hold:
\begin{align}
\sum_{\ell=0}^{mn}\left\vert T_{2\ell}\left(  x\right)  \xi^{2\ell
}\right\vert  &  \leq\sum_{k=0}^{2mn}\text{$\left\vert T_{k}\left(  x\right)
\xi^{k}\right\vert $}\leq\frac{1}{2}\left(  \Sigma\left(  x,\xi\right)
+c\right)  ,\text{ and}\label{e.4.16}\\
\frac{3}{2}\Sigma\left(  x,\xi\right)  +\frac{1}{2}c  &  \geq\operatorname{Re}%
\sigma_{n}\left(  x,\xi\right)  \geq\frac{1}{2}\Sigma\left(  x,\xi\right)
-\frac{1}{2}c. \label{e.4.17}%
\end{align}
Alternatively, adding $c$ to both sides of Eq. (\ref{e.4.17}) shows
\begin{equation}
\frac{3}{2}\left(  \Sigma\left(  x,\xi\right)  +c\right)  \geq
\operatorname{Re}\sigma_{L^{n}+c}\left(  x,\xi\right)  \geq\frac{1}{2}\left(
\Sigma\left(  x,\xi\right)  +c\right)  . \label{e.4.18}%
\end{equation}
A key point is that $\operatorname{Re}\sigma_{n}\left(  x,\xi\right)
:=\operatorname{Re}\sigma_{L^{n}}\left(  x,\xi\right)  $ (see Notation
\ref{not.4.4}) is bounded from below.
\end{lemma}

\begin{notation}
\label{not.4.11}For the rest of this section, we fix a $c>0$ as in Lemma
\ref{lem.4.10} and then define $\mathbb{L}:=L^{n}+c$ where $\mathcal{D}\left(
\mathbb{L}\right)  =$ $\mathcal{S}$ and $\left\{  b_{l}\right\}  _{l=1}^{m}$
in Eq. (\ref{e.1.6}) satisfies the assumptions in Theorem \ref{the.1.10}.
According to the definition of symbol in Eq. (\ref{e.1.3}),
\[
\sigma_{\mathbb{L}}\left(  x,\xi\right)  :=\sigma_{L^{n}+c}\left(
x,\xi\right)  =\sigma_{n}\left(  x,\xi\right)  +c
\]

\end{notation}

Because of our choice of $c>0$ we know that
\begin{equation}
\kappa:=\inf_{\left(  x,\xi\right)  }\operatorname{Re}\sigma_{\mathbb{L}%
}\left(  x,\xi\right)  >0. \label{e.4.19}%
\end{equation}

\begin{corollary}
\label{cor.4.12}For all $l,k\in\mathbb{N}_{0},$
\begin{equation}
\left\vert \partial_{\xi}^{l}\partial_{x}^{k}\sigma_{\mathbb{L}}\left(
x,\xi\right)  \right\vert \left(  1+\left\vert \xi\right\vert \right)
^{l}\left(  1+\left\vert x\right\vert \right)  ^{k}\lesssim\Sigma\left(
x,\xi\right)  +1 \label{e.4.20}%
\end{equation}
or equivalently,
\begin{equation}
\left\vert \frac{\partial_{\xi}^{l}\partial_{x}^{k}\sigma_{\mathbb{L}}\left(
x,\xi\right)  }{\operatorname{Re}\sigma_{\mathbb{L}}\left(  x,\xi\right)
}\right\vert \lesssim\frac{1}{\left(  1+\left\vert \xi\right\vert \right)
^{l}\left(  1+\left\vert x\right\vert \right)  ^{k}} \label{e.4.21}%
\end{equation}
if Eq. (\ref{e.4.18}) is applied.
\end{corollary}

\begin{proof}
We have
\[
\partial_{\xi}^{l}\partial_{x}^{k}\sigma_{\mathbb{L}}\left(  x,\xi\right)
=\sum_{j=l}^{2mn}\left(  i\right)  ^{j}\frac{j!}{\left(  j-l\right)
!}\partial_{x}^{k}A_{j}\left(  x\right)  \xi^{j-l}
\]
and therefore,
\begin{align*}
\left\vert \partial_{\xi}^{l}\partial_{x}^{k}\sigma_{\mathbb{L}}\left(
x,\xi\right)  \right\vert  &  \left(  1+\left\vert \xi\right\vert \right)
^{l}\left(  1+\left\vert x\right\vert \right)  ^{k}\\
&  \leq\sum_{j=l}^{2mn}\frac{j!}{\left(  j-l\right)  !}\left\vert \partial
_{x}^{k}A_{j}\left(  x\right)  \right\vert \left\vert \xi^{j-l}\right\vert
\left(  1+\left\vert \xi\right\vert \right)  ^{l}\left(  1+\left\vert
x\right\vert \right)  ^{k}\\
&  \lesssim\sum_{j=l}^{2mn}\left\vert \partial_{x}^{k}A_{j}\left(  x\right)
\right\vert \left(  1+\left\vert \xi\right\vert ^{j}\right)  \left(
1+\left\vert x\right\vert ^{k}\right) \\
&  \lesssim\Sigma\left(  x,\xi\right)  +1.
\end{align*}
The last step is asserted by Corollary \ref{cor.4.9}. Equation (\ref{e.4.21})
follows directly from Eqs. (\ref{e.4.18}) and (\ref{e.4.20}).
\end{proof}

By the Fourier inversion formula, if $\psi\in\mathcal{S},$ then
\begin{equation}
\psi\left(  x\right)  =\int_{\mathbb{R}}\widehat{\psi}\left(  \xi\right)
e^{ix\xi}d\xi, \label{e.4.22}%
\end{equation}
where $\widehat{\psi}$ is the Fourier transform of $\psi$ defined by
\begin{equation}
\widehat{\psi}\left(  \xi\right)  =\frac{1}{2\pi}\int_{\mathbb{R}}e^{-iy\xi
}\psi(y)dy. \label{e.4.23}%
\end{equation}
Recall that, with these normalizations, that
\begin{equation}
\left\Vert \psi\right\Vert =\sqrt{2\pi}\left\Vert \hat{\psi}\right\Vert
~\forall~\psi\in L^{2}\left(  \mathbb{R}\right)  . \label{e.4.24}%
\end{equation}

Letting $\mu\in\mathbb{R}$ and then applying $\mathbb{L}+i\mu$ to Eq.
(\ref{e.4.22}) gives the following pseudo-differential operator representation
of $\left(  \mathbb{L}+i\mu\right)  \psi,$
\begin{equation}
\left(  \mathbb{L}+i\mu\right)  \psi\left(  x\right)  =\int_{\mathbb{R}%
}\left[  \sigma_{\mathbb{L}}\left(  x,\xi\right)  +i\mu\right]  e^{ix\xi
}\widehat{\psi(\xi)}d\xi. \label{e.4.25}%
\end{equation}

Let $\kappa$ be as in Eq. (\ref{e.4.19}), it follows that for any $\mu
\in\mathbb{R},$
\[
\left\vert \sigma_{\mathbb{L}}\left(  x,\xi\right)  +i\mu\right\vert
\geq\left\vert \operatorname{Re}\sigma_{\mathbb{L}}\left(  x,\xi\right)
\right\vert \geq\kappa>0
\]
for all $\left(  x,\xi\right)  \in\mathbb{R}^{2}.$ Therefore, the following
integrand in Eq. (\ref{e.4.26}) is integrable for $u\in\mathcal{S}$ and we may
define
\begin{equation}
\left(  T_{\mu}u\right)  \left(  x\right)  =\int_{\mathbb{R}}\frac{1}%
{\sigma_{\mathbb{L}}\left(  x,\xi\right)  +i\mu}e^{ix\xi}\hat{u}(\xi)d\xi.
\label{e.4.26}%
\end{equation}
Furthermore, we will show that $T_{\mu}$ actually preserves $\mathcal{S}$
later in this section (see Proposition \ref{pro.4.17}).

\begin{notation}
\label{not.4.13}If $\left\{  q_{k}\left(  x\right)  \right\}  _{k=0}^{j}$ is a
collection of smooth functions and
\begin{equation}
q\left(  x,\theta\right)  =\sum_{k=0}^{j}q_{k}\left(  x\right)  \theta^{k},
\label{e.4.27}%
\end{equation}
then $q\left(  x,\partial\right)  $ is defined to be the $j^{\text{th}}$ --
order differential operator given by
\begin{equation}
q\left(  x,\partial\right)  :=\sum_{k=0}^{j}q_{k}\left(  x\right)
\partial_{x}^{k}. \label{e.4.28}%
\end{equation}
Similarly, for $\xi\in\mathbb{R},$ we let
\begin{equation}
q\left(  x,\frac{1}{i}\partial_{x}+\xi\right)  :=\sum_{k=0}^{j}q_{k}\left(
x\right)  \left(  \frac{1}{i}\partial_{x}+\xi\right)  ^{k}. \label{e.4.29}%
\end{equation}

\end{notation}

For the proofs below, recall from Eq. (\ref{e.3.11}) that
\begin{equation}
q\left(  \partial_{x}\right)  M_{e^{i\xi\cdot x}}=M_{e^{i\xi\cdot x}}q\left(
\partial_{x}+i\xi\right)  . \label{e.4.30}%
\end{equation}
whenever $q\left(  \theta\right)  $ is a polynomial in $\theta.$

\begin{lemma}
\label{lem.4.14}Let $q\left(  x,\theta\right)  $ be as in Eq. (\ref{e.4.27})
where the coefficients $\left\{  q_{k}\left(  x\right)  \right\}  _{k=0}^{j}$
are now assumed to be polynomials in $x.$ Further let
\begin{equation}
\left(  Su\right)  \left(  x\right)  :=\int_{\mathbb{R}}\Gamma\left(
x,\xi\right)  e^{ix\xi}\hat{u}(\xi)d\xi\text{ }\forall~u\in\mathcal{S},
\label{e.4.31}%
\end{equation}
where $\Gamma\left(  x,\xi\right)  $ is a smooth function such that
$\Gamma\left(  x,\xi\right)  $ and all of its derivatives in both $x$ and
$\xi$ have at most polynomial growth in $\xi$ for any fixed $x.$ Then
\begin{equation}
\left(  q\left(  x,\partial_{x}\right)  Su\right)  \left(  x\right)
=\int_{\mathbb{R}}e^{i\xi\cdot x}q\left(  i\partial_{\xi},\partial_{x}%
+i\xi\right)  \left[  \Gamma\left(  x,\xi\right)  \hat{u}\left(  \xi\right)
\right]  d\xi, \label{e.4.32}%
\end{equation}
where
\begin{equation}
q\left(  i\partial_{\xi},\partial_{x}+i\xi\right)  :=\sum_{k=0}^{j}%
q_{k}\left(  i\partial_{\xi}\right)  \left(  \partial_{x}+iM_{\xi}\right)
^{k}. \label{e.4.33}%
\end{equation}

\end{lemma}

\begin{proof}
Using Eq. (\ref{e.4.30}) we find,
\begin{align*}
\left(  q\left(  x,\partial\right)  Su\right)  \left(  x\right)   &
=\int_{\mathbb{R}}q\left(  x,\partial_{x}\right)  \left[  \Gamma\left(
x,\xi\right)  \hat{u}\left(  \xi\right)  e^{i\xi\cdot x}\right]  d\xi\\
&  =\sum_{k=0}^{j}\int_{\mathbb{R}}q_{k}\left(  x\right)  \partial_{x}%
^{k}\left[  \Gamma\left(  x,\xi\right)  \hat{u}\left(  \xi\right)
e^{i\xi\cdot x}\right]  d\xi\\
&  =\sum_{k=0}^{j}\int_{\mathbb{R}}q_{k}\left(  x\right)  e^{i\xi\cdot
x}\left(  \partial_{x}+i\xi\right)  ^{k}\left[  \Gamma\left(  x,\xi\right)
\hat{u}\left(  \xi\right)  \right]  d\xi\\
&  =\sum_{k=0}^{j}\int_{\mathbb{R}}\left[  q_{k}\left(  -i\partial_{\xi
}\right)  e^{i\xi\cdot x}\right]  \left(  \partial_{x}+i\xi\right)
^{k}\left[  \Gamma\left(  x,\xi\right)  \hat{u}\left(  \xi\right)  \right]
d\xi\\
&  =\sum_{k=0}^{j}\int_{\mathbb{R}}e^{i\xi\cdot x}q_{k}\left(  i\partial_{\xi
}\right)  \left(  \partial_{x}+i\xi\right)  ^{k}\left[  \Gamma\left(
x,\xi\right)  \hat{u}\left(  \xi\right)  \right]  d\xi\\
&  =\int_{\mathbb{R}}e^{i\xi\cdot x}q\left(  i\partial_{\xi},\partial_{x}%
+i\xi\right)  \left[  \Gamma\left(  x,\xi\right)  \hat{u}\left(  \xi\right)
\right]  d\xi.
\end{align*}
We have used the assumptions on $\Gamma$ to show; (1) that $\partial_{x}$
commutes with the integral giving the first equality above, and (2) that
\[
\xi\rightarrow\left(  \partial_{x}+i\xi\right)  ^{k}\left[  \Gamma\left(
x,\xi\right)  \hat{u}\left(  \xi\right)  \right]  \in\mathcal{S}%
\]
which is used to justify the integration by parts used in the in the second to
last equality.
\end{proof}

\begin{lemma}
\label{lem.4.15}Suppose that $f\in\mathbb{C}^{\infty}\left(  \mathbb{R}%
^{j},\left(  0,\infty\right)  \right)  ,$ then for every multi-index,
$\alpha=\left(  \alpha_{1},\dots,\alpha_{j}\right)  \in\mathbb{N}_{0}^{j}$
with $\alpha\neq0$ there exists a polynomial function, $P_{\alpha},$ with no
constant term such that
\[
\partial^{\alpha}\frac{1}{f}=\frac{1}{f}P_{\alpha}\left(  \left\{
\frac{\partial^{\beta}f}{f}:0<\beta\leq\alpha\right\}  \right)
\]
where $\partial^{\alpha}:=\partial_{1}^{\alpha_{1}}\dots\partial_{j}%
^{\alpha_{j}}.$ Moreover, $P_{\alpha}\left(  \left\{  \frac{\partial^{\beta}%
f}{f}:0<\beta\leq\alpha\right\}  \right)  $ is a linear combination of
monomials of the form $\prod_{i=1}^{k}\frac{\partial^{\beta^{\left(  i\right)
}}f}{f}$ where $\beta^{\left(  i\right)  }\in\mathbb{N}_{0}^{j}$ for $1\leq
i\leq k$ and $1\leq k\leq\left\vert \alpha\right\vert $ such that $\sum
_{i=1}^{k}\beta^{\left(  i\right)  }=\alpha,$ and $\beta^{\left(  i\right)
}\neq0$ for all $i.$
\end{lemma}

\begin{proof}
The proof is a straight forward induction argument which will be left to the
reader. However, by way of example one easily shows,
\[
\partial_{1}\partial_{2}\frac{1}{f}=\frac{1}{f}\cdot\left[  \frac{\partial
_{1}f}{f}\frac{\partial_{2}f}{f}-\frac{\partial_{1}\partial_{2}f}{f}\right]
.
\]

\end{proof}

\begin{corollary}
\label{cor.4.16} Let $\mu\in\mathbb{R}.$ If
\begin{equation}
\Gamma\left(  x,\xi\right)  :=\frac{1}{\sigma_{\mathbb{L}}\left(
x,\xi\right)  +i\mu}, \label{e.4.34}%
\end{equation}
then
\begin{equation}
\left\vert \Gamma\left(  x,\xi\right)  \right\vert \leq\frac{1}{\left\vert
\operatorname{Re}\sigma_{\mathbb{L}}\left(  x,\xi\right)  \right\vert
}\lesssim\frac{1}{b_{m}^{n}\left(  x\right)  \xi^{2mn}+b_{0}^{n}\left(
x\right)  +c}\lesssim1 \label{e.4.35}%
\end{equation}
and for any $\alpha,\beta\in\mathbb{N}_{0}$ with $\alpha+\beta>0,$ there
exists a constant $c_{\alpha,\beta}>0$ such that
\begin{equation}
\left\vert \partial_{x}^{\alpha}\partial_{\xi}^{\beta}\Gamma\left(
x,\xi\right)  \right\vert \leq\left\vert \Gamma\left(  x,\xi\right)
\right\vert \cdot c_{\alpha,\beta}\left(  1+\left\vert \xi\right\vert \right)
^{-\beta}\left(  1+\left\vert x\right\vert \right)  ^{-\alpha}. \label{e.4.36}%
\end{equation}

\end{corollary}

\begin{proof}
The estimate in Eq. (\ref{e.4.35}) is elementary from Eq.(\ref{e.4.18}) in
Lemma \ref{lem.4.10} and will be left to the reader. From Lemma \ref{lem.4.15}%
,
\begin{align*}
\partial_{x}^{\alpha}\partial_{\xi}^{\beta}\Gamma\left(  x,\xi\right)   &
=\partial_{x}^{\alpha}\partial_{\xi}^{\beta}\frac{1}{\sigma_{\mathbb{L}%
}\left(  x,\xi\right)  +i\mu}=\frac{1}{\sigma_{\mathbb{L}}\left(
x,\xi\right)  +i\mu}\cdot\Theta_{\left(  \alpha,\beta\right)  }\left(
x,\xi\right)  \\
&  =\Gamma\left(  x,\xi\right)  \cdot\Theta_{\left(  \alpha,\beta\right)
}\left(  x,\xi\right)
\end{align*}
where $\Theta_{\left(  \alpha,\beta\right)  }\left(  x,\xi\right)  $ is a
linear combination of the following functions,
\[
\prod_{j=1}^{J}\frac{\partial_{x}^{k_{j}}\partial_{\xi}^{l_{j}}\sigma
_{\mathbb{L}}\left(  x,\xi\right)  }{\sigma_{\mathbb{L}}\left(  x,\xi\right)
+i\mu}%
\]
where $0\leq k_{j}\leq\alpha,~0\leq l_{j}\leq\beta,~k_{j}+l_{j}>0,~\sum
_{j=1}^{J}k_{j}=\alpha,~\sum_{j=1}^{J}l_{j}=\beta$ and $1\leq J\leq
\alpha+\beta$. The estimate in Eq. (\ref{e.4.21}) implies,
\begin{align*}
\prod_{j=1}^{J}\left\vert \frac{\partial_{x}^{k_{j}}\partial_{\xi}^{l_{j}%
}\sigma_{\mathbb{L}}\left(  x,\xi\right)  }{\sigma_{\mathbb{L}}\left(
x,\xi\right)  +i\mu}\right\vert  & \leq\prod_{j=1}^{J}\left\vert
\frac{\partial_{x}^{k_{j}}\partial_{\xi}^{l_{j}}\sigma_{\mathbb{L}}\left(
x,\xi\right)  }{\operatorname{Re}\sigma_{\mathbb{L}}\left(  x,\xi\right)
}\right\vert \\
& \lesssim\prod_{j=1}^{J}\frac{1}{\left(  1+\left\vert \xi\right\vert \right)
^{l_{j}}\left(  1+\left\vert x\right\vert \right)  ^{k_{j}}}\lesssim\left(
1+\left\vert \xi\right\vert \right)  ^{-\beta}\left(  1+\left\vert
x\right\vert \right)  ^{-\alpha}%
\end{align*}
which altogether gives the estimated in Eq. (\ref{e.4.36}).
\end{proof}

\begin{proposition}
[$T_{\mu}\mathcal{\ }$preserves $\mathcal{S}$]\label{pro.4.17} If $T_{\mu}$ is
as defined in Eq. (\ref{e.4.26}), then $T_{\mu}\left(  \mathcal{S}\right)
\subset\mathcal{S}$ for all $\mu\in\mathbb{R}.$
\end{proposition}

\begin{proof}
Let $\Gamma$ be as in Eq. (\ref{e.4.34}) so that $T_{\mu}=S$ where $S$ is as
in Lemma \ref{lem.4.14}. According Corollary \ref{cor.4.16}, for all
$\alpha,\beta\in\mathbb{N}_{0}$, we know that $\left\vert \partial_{x}%
^{\alpha}\partial_{\xi}^{\beta}\Gamma\left(  x,\xi\right)  \right\vert \leq
C_{\alpha,\beta}$ for some constants $C_{\alpha,\beta}$ and hence, from Lemma
\ref{lem.4.14}, if $q\left(  x,\theta\right)  $ is as in Eq. (\ref{e.4.27}),
then
\begin{equation}
\left\vert q\left(  x,\partial_{x}\right)  T_{\mu}u\left(  x\right)
\right\vert \leq\int_{\mathbb{R}}\left\vert q\left(  i\partial_{\xi}%
,\partial_{x}+i\xi\right)  \left[  \Gamma\left(  x,\xi\right)  \hat{u}\left(
\xi\right)  \right]  \right\vert d\xi. \label{e.4.37}%
\end{equation}
The integrand in Eq. (\ref{e.4.37}) may be bounded by a finite linear
combination of terms of the form
\[
\left\vert \partial_{x}^{\alpha}\partial_{\xi}^{\beta}\Gamma\left(
x,\xi\right)  \right\vert \cdot\left\vert \xi^{j}\partial_{\xi}^{l}\hat
{u}\right\vert \left(  \xi\right)  \lesssim\left\vert \xi^{j}\partial_{\xi
}^{l}\hat{u}\right\vert \left(  \xi\right)  .
\]
Since $\hat{u}\in\mathcal{S},$ $\left\vert \xi^{j}\partial_{\xi}^{l}\hat
{u}\right\vert \left(  \xi\right)  $ is integrable and therefore we may
conclude that
\[
\sup_{x\in\mathbb{R}}\left\vert \left(  q\left(  x,\partial_{x}\right)
T_{\mu}u\right)  \left(  x\right)  \right\vert <\infty.
\]
As $q\left(  x,\theta\right)  $ was an arbitrary polynomial in $\left(
x,\theta\right)  $ we conclude that $T_{\mu}u\in\mathcal{S}.$
\end{proof}

We assume $\Gamma$ is as in Eq. (\ref{e.4.34}) for the remainder of this paper.

\begin{lemma}
\label{lem.4.18}For all $\mu\in\mathbb{R}$ and $u\in\mathcal{S},$
\begin{equation}
\left[  \mathbb{L}+i\mu\right]  T_{\mu}u=\left[  I+R_{\mu}\right]  u
\label{e.4.38}%
\end{equation}
where
\begin{equation}
\left(  R_{\mu}u\right)  \left(  x\right)  =\int_{\mathbb{R}}\rho_{\mu}\left(
x,\xi\right)  \hat{u}(\xi)e^{ix\xi}d\xi, \label{e.4.39}%
\end{equation}
\begin{equation}
\rho_{\mu}\left(  x,\xi\right)  :=\left(  \left[  \sigma_{\mathbb{L}}\left(
x,\frac{1}{i}\partial_{x}+\xi\right)  -\sigma_{\mathbb{L}}\left(
x,\xi\right)  \right]  \frac{1}{\sigma_{\mathbb{L}}\left(  x,\xi\right)
+i\mu}\right)  , \label{e.4.40}%
\end{equation}
and $\sigma_{\mathbb{L}}\left(  x,\frac{1}{i}\partial_{x}+\xi\right)  $ is as
in Eq. (\ref{e.4.29}).
\end{lemma}

\begin{proof}
As $\sigma_{\mathbb{L}}\left(  x,\xi\right)  $ is a polynomial in the $\xi$ --
variables with smooth coefficients in the $x$ -- variables, there is no
problem justifying the identity,
\begin{equation}
\left(  \left[  \mathbb{L}+i\mu\right]  T_{\mu}u\right)  \left(  x\right)
=\int_{\mathbb{R}}\left[  \mathbb{L}_{x}+i\mu\right]  \left(  \Gamma\left(
x,\xi\right)  e^{ix\xi}\right)  \hat{u}(\xi)d\xi, \label{e.4.41}%
\end{equation}
where the subscript $x$ on $\mathbb{L}$ indicates that $\mathbb{L}$ acts on
$x$ -- variables only. Using $\mathbb{L}=\sigma_{\mathbb{L}}\left(  x,\frac
{1}{i}\partial_{x}\right)  $ along with Eq. (\ref{e.4.30}) shows,
\begin{align*}
\left[  \mathbb{L}_{x}+i\mu\right]   &  \left(  \Gamma\left(  x,\xi\right)
e^{ix\xi}\right) \\
&  =\left[  \sigma_{\mathbb{L}}\left(  x,\frac{1}{i}\partial_{x}\right)
+i\mu\right]  \left(  e^{ix\xi}\Gamma\left(  x,\xi\right)  \right) \\
&  =e^{ix\xi}\left[  \sigma_{\mathbb{L}}\left(  x,\frac{1}{i}\partial_{x}%
+\xi\right)  +i\mu\right]  \Gamma\left(  x,\xi\right) \\
&  =e^{ix\xi}\left[  \sigma_{\mathbb{L}}\left(  x,\frac{1}{i}\partial_{x}%
+\xi\right)  -\sigma_{\mathbb{L}}\left(  x,\xi\right)  +\left(  \sigma
_{\mathbb{L}}\left(  x,\xi\right)  +i\mu\right)  \right]  \Gamma\left(
x,\xi\right) \\
&  =e^{ix\xi}\left[  \sigma_{\mathbb{L}}\left(  x,\frac{1}{i}\partial_{x}%
+\xi\right)  -\sigma_{\mathbb{L}}\left(  x,\xi\right)  \right]  \Gamma\left(
x,\xi\right)  +e^{ix\xi}%
\end{align*}
which combined with Eq. (\ref{e.4.41}) gives Eq. (\ref{e.4.38}).
\end{proof}

\begin{lemma}
\label{lem.4.19} $\rho_{\mu}\left(  x,\xi\right)  $ in Eq. (\ref{e.4.40}) can
be explicitly written as
\begin{equation}
\rho_{\mu}\left(  x,\xi\right)  =\sum_{k=1}^{2mn}\sum_{j=1}^{k}\binom{k}%
{j}A_{k}\left(  x\right)  \left(  i\xi\right)  ^{k-j}\partial_{x}^{j}%
\Gamma\left(  x,\xi\right)  . \label{e.4.42}%
\end{equation}
where $A_{k}\left(  x\right)  $ and $\Gamma\left(  x,\xi\right)  $ as in Eq.
(\ref{e.3.1}) and Eq. (\ref{e.4.34}) respectively. Moreover, there exists
$C<\infty$ independent of $\mu$ so that
\begin{equation}
\left\vert \rho_{\mu}\left(  x,\xi\right)  \right\vert \leq C\frac
{1}{1+\left\vert x\right\vert }\cdot\frac{1}{1+\left\vert \xi\right\vert }.
\label{e.4.43}%
\end{equation}

\end{lemma}

\begin{proof}
Using Eq. (\ref{e.1.7}) and the formula of $\sigma_{\mathbb{L}}$ in Notation
\ref{not.4.11}, we may write Eq. (\ref{e.4.40}) more explicitly as,
\begin{align*}
\rho_{\mu}\left(  x,\xi\right)   &  =\sum_{k=0}^{2mn}A_{k}\left(  x\right)
\left[  \left(  \partial_{x}+i\xi\right)  ^{k}-\left(  i\xi\right)
^{k}\right]  \Gamma\left(  x,\xi\right)  \\
&  =\sum_{k=1}^{2mn}\sum_{j=1}^{k}\binom{k}{j}A_{k}\left(  x\right)  \left(
i\xi\right)  ^{k-j}\partial_{x}^{j}\Gamma\left(  x,\xi\right)  .
\end{align*}
Therefore, using the estimate in Eq. (\ref{e.4.36}) of Corollary
\ref{cor.4.16} with $\beta=0$ and $\alpha=j,$ we learn
\begin{align*}
\left\vert \rho_{\mu}\left(  x,\xi\right)  \right\vert  &  \leq\sum
_{k=1}^{2mn}\sum_{j=1}^{k}\binom{k}{j}\left\vert A_{k}\left(  x\right)
\right\vert \left\vert \xi\right\vert ^{k-j}\left\vert \partial_{x}^{j}%
\Gamma\left(  x,\xi\right)  \right\vert \\
&  \lesssim\sum_{k=1}^{2mn}\sum_{j=1}^{k}\binom{k}{j}\left\vert A_{k}\left(
x\right)  \right\vert \left\vert \xi\right\vert ^{k-j}\left\vert \Gamma\left(
x,\xi\right)  \right\vert \frac{1}{1+\left\vert x\right\vert }.
\end{align*}
Moreover, for any $1\leq j\leq k,$
\begin{align*}
\left\vert A_{k}\left(  x\right)  \right\vert \left\vert \xi\right\vert
^{k-j}\left\vert \Gamma\left(  x,\xi\right)  \right\vert  &  \lesssim
\frac{\Sigma\left(  x,\xi\right)  +1}{1+\left\vert \xi\right\vert ^{k}%
}\left\vert \xi\right\vert ^{k-j}\left\vert \Gamma\left(  x,\xi\right)
\right\vert \\
&  \leq\frac{1}{1+\left\vert \xi\right\vert ^{j}}\frac{\Sigma\left(
x,\xi\right)  +1}{\left\vert \operatorname{Re}\sigma_{\mathbb{L}}\left(
x,\xi\right)  \right\vert }\\
&  \lesssim\frac{1}{1+\left\vert \xi\right\vert ^{j}}\lesssim\frac
{1}{1+\left\vert \xi\right\vert },
\end{align*}
wherein we have used the estimates in Eq. (\ref{e.4.15}) with $\beta=0$ in the
first step, and the left inequality in Eq. (\ref{e.4.35}) in the second step,
and Eq. (\ref{e.4.18}) in the third step.
\end{proof}

We are now prepared to complete the proof of Theorem \ref{the.1.10}. The
following notation will be used in the proof.

\begin{notation}
\label{not.4.20}If $g:\mathbb{R}^{2}\rightarrow\mathbb{C}$ is a measurable
function we let
\begin{align*}
\left\Vert g\left(  x,\xi\right)  \right\Vert _{L^{2}(d\xi)}  &  :=\left(
\int_{\mathbb{R}}\left\vert g\left(  x,\xi\right)  \right\vert ^{2}%
d\xi\right)  ^{1/2}\text{ and }\\
\left\Vert g\left(  x,\xi\right)  \right\Vert _{L^{2}(dx\otimes d\xi)}  &
:=\left(  \int_{\mathbb{R}^{2}}\left\vert g\left(  x,\xi\right)  \right\vert
^{2}dxd\xi\right)  ^{1/2}.
\end{align*}

\end{notation}

\begin{proof}
[Proof of Theorem \ref{the.1.10}]The only thing left to show is that condition
2 in Lemma \ref{lem.4.2} is verified. Thus we have to estimate the operator
norm of the error term,
\[
\left(  R_{\mu}u\right)  \left(  x\right)  =\int_{\mathbb{R}}\rho_{\mu}\left(
x,\xi\right)  e^{ix\xi}\hat{u}(\xi)d\xi.
\]

Using the Cauchy--Schwarz inequality and the isometry property (see Eq.
(\ref{e.4.24}) of the Fourier transform it follows that
\[
\left\Vert R_{\mu}u\right\Vert _{L^{2}\left(  dx\right)  }\leq\frac{1}%
{\sqrt{2\pi}}\left\Vert \rho_{\mu}\right\Vert _{L^{2}\left(  dx\otimes
d\xi\right)  }\cdot\left\Vert u\right\Vert _{L^{2}\left(  dx\right)  }%
\]
where $\rho_{\mu}$ is the symbol of $R_{\mu}$ as defined in Eq. (\ref{e.4.40}%
). Since, by Lemma \ref{lem.4.15} and Eq. (\ref{e.4.42}), $\lim_{\mu
\rightarrow\pm\infty}\rho_{\mu}\left(  x,\xi\right)  =0$ and, from Eq.
(\ref{e.4.43}), $\rho_{\mu}$ is dominated by
\[
C\left(  1+\left\vert x\right\vert \right)  ^{-1}\left(  1+\left\vert
\xi\right\vert \right)  ^{-1}\in L^{2}\left(  dx\otimes d\xi\right)  ,
\]
it follows that $\left\Vert R_{\mu}\right\Vert _{op}\leq\frac{1}{\sqrt{2\pi}%
}\left\Vert \rho_{\mu}\right\Vert _{L^{2}\left(  dx\otimes d\xi\right)
}\rightarrow0$ as $\mu\rightarrow\pm\infty$ and in particular, $\left\Vert
R_{\mu}\right\Vert _{op}<1$ when $\left\vert \mu\right\vert $ is sufficiently
large. Therefore, $\mathbb{L}|_{\mathcal{S}}$ is essentially self-adjoint from
Lemma \ref{lem.4.2} and hence $L^{n}|_{\mathcal{S}}=\left(  \mathbb{L}%
-c\right)  |_{\mathcal{S}}$ ($c$ from Notation \ref{not.4.11}) is also
essentially self-adjoint.
\end{proof}

\section{The Divergence Form of $L^{n}$ and $L_{\hbar}^{n}$\label{sec.5}}

Suppose now that $L$ in Eq. (\ref{e.2.1}) with polynomial coefficients
$\left\{  b_{l}\right\}  _{l=0}^{m}$ is a symmetric differential operator on
$\mathcal{S}.$ In section \ref{sec.3}, we have expressed the symmetric
differential operators on $\mathcal{S}$, $L^{n},$ in the divergence form with
the polynomial coefficients $\left\{  B_{\ell}\right\}  $ as in Eq.
(\ref{e.3.1}) for $n\in\mathbb{N}$. The goal of this section is to derive some
basic properties of the polynomial coefficients $\left\{  B_{\ell}\right\}  $
and generalize coefficients properties for a scaled version $L_{\hbar}^{n}$
where $L_{\hbar}$ is in Eq. (\ref{e.1.11}).

\begin{proposition}
\label{pro.5.1}Suppose that $\left\{  b_{l}\right\}  _{l=0}^{m}$ are real
polynomials. Let $\mathcal{B}_{\ell}$ and $R_{\ell}$ are in Eqs. (\ref{e.3.2})
and (\ref{e.3.17}) respectively.

\begin{enumerate}
\item If $\deg b_{l}\leq\deg b_{l-1}$ for $1\leq l\leq m,$ then
\[
\deg\left(  R_{\ell}\right)  \leq\deg\mathcal{B}_{\ell}-2\text{ and
}\operatorname{deg}B_{\ell}=\operatorname{deg}\mathcal{B}_{\ell}\text{ for
}0\leq\ell\leq mn.
\]

\item If we only assume that $\deg b_{l}\leq\deg b_{l-1}+2$ for $1\leq l\leq
m,$ then
\[
\deg R_{\ell}\leq\deg\mathcal{B}_{\ell}\quad\text{for }0\leq\ell\leq mn.
\]

\end{enumerate}
\end{proposition}

\begin{proof}
From Eq. (\ref{e.3.16}), $\deg\left(  B_{\ell}\right)  =\deg\mathcal{B}_{\ell
}$ follows automatically if
\begin{equation}
\deg\left(  R_{\ell}\right)  \leq\deg\mathcal{B}_{\ell}-2 \label{e.5.1}%
\end{equation}
holds. Therefore, the only thing to prove in the item 1 is Eq. (\ref{e.5.1}).
From Proposition \ref{pro.3.10}, $R_{\ell}$ is a linear combination of
$\left(  \partial^{p_{1}}b_{k_{1}}\right)  \left(  \partial^{p_{2}}b_{k_{2}%
}\right)  \dots\left(  \partial^{p_{n}}b_{k_{n}}\right)  $ with $0<\left\vert
\mathbf{p}\right\vert =2\left\vert \mathbf{k}\right\vert -2\ell.$ For each
index $\mathbf{k},$ there exists $\mathbf{j}$ with $\mathbf{j\leq k}$ such
that $\left\vert \mathbf{j}\right\vert =\ell$ and for this $\mathbf{j}$ we
have
\begin{align*}
\deg\left(  \left(  \partial^{p_{1}}b_{k_{1}}\right)  \left(  \partial^{p_{2}%
}b_{k_{2}}\right)  \dots\left(  \partial^{p_{n}}b_{k_{n}}\right)  \right)   &
\leq\sum_{i=1}^{n}\deg\left(  b_{k_{i}}\right)  -\left\vert \mathbf{p}%
\right\vert \\
&  \leq\sum_{i=1}^{n}\deg\left(  b_{j_{i}}\right)  -\left\vert \mathbf{p}%
\right\vert \\
&  \leq\deg\left(  \mathcal{B}_{\ell}\right)  -2,
\end{align*}
wherein we have used $\left\vert \mathbf{p}\right\vert \geq2$ ($\left\vert
\mathbf{p}\right\vert $ is positive even) and
\[
\deg\left(  \mathcal{B}_{\ell}\right)  =\max_{\left\vert \mathbf{j}\right\vert
=\ell}\deg\left(  b_{j_{1}}\dots b_{j_{n}}\right)  =\max_{\left\vert
\mathbf{j}\right\vert =\ell}\sum_{i=1}^{n}\deg\left(  b_{j_{i}}\right)  .
\]

Now suppose that we only assume $\deg\left(  b_{k+1}\right)  \leq\deg\left(
b_{k}\right)  +2$ (which then implies $\deg\left(  b_{k+r}\right)  \leq
\deg\left(  b_{k}\right)  +2r$ for $0\leq r\leq m-k).$ Working as above and
remember that $0<\left\vert \mathbf{p}\right\vert =2\left\vert \mathbf{k}%
\right\vert -2\ell$ and $\left|  j\right|  =\ell$ we find
\begin{align*}
\sum_{i=1}^{n}\deg\left(  b_{k_{i}}\right)  -\left\vert \mathbf{p}\right\vert
&  \leq\sum_{i=1}^{n}\left[  \deg\left(  b_{j_{i}}\right)  +2\left(
k_{i}-j_{i}\right)  \right]  -\left\vert \mathbf{p}\right\vert \\
&  =\sum_{i=1}^{n}\deg\left(  b_{j_{i}}\right)  +2\left(  \left\vert
\mathbf{k}\right\vert -\left\vert \mathbf{j}\right\vert \right)  -\left\vert
\mathbf{p}\right\vert \\
&  =\sum_{i=1}^{n}\deg\left(  b_{j_{i}}\right)  \leq\deg\left(  \mathcal{B}%
_{\ell}\right)  .
\end{align*}

\end{proof}

\subsection{Scaled Version of Divergence Form}

We now take $\hbar>0$ and let $L_{\hbar}$ be defined as in Eq. (\ref{e.1.11})
where the $\hbar$ -- dependent coefficients, $\left\{  b_{l,\hbar}\left(
x\right)  \right\}  _{l=0}^{m},$ satisfy Assumption \ref{ass.1}. To apply the
previous formula already developed (for $\hbar=1)$ we need only make the
replacements,
\begin{equation}
b_{l}\left(  x\right)  \rightarrow\hbar^{l}b_{l,\hbar}\left(  \sqrt{\hbar
}x\right)  \text{ for }0\leq l\leq m. \label{e.5.2}%
\end{equation}
The result of this transformation on $L^{n}$ is recorded in the following lemma.

\begin{notation}
\label{not.5.2} Let $x_{1},\dots,x_{j}$ be variables on $\mathbb{R}.$ We
denote $\mathbb{R}\left[  x_{1},\dots,x_{j}\right]  $ be a collection of
polynomials in $x_{1},\dots,x_{j}$ with real-valued coefficients.
\end{notation}

\begin{proposition}
\label{pro.5.3} Let $n\in\mathbb{N},$ $\hbar>0,$ and $\left\{  b_{l,\hbar
}\left(  x\right)  \right\}  _{l=0}^{m}\subset\mathbb{R}\left[  x\right]  $
and let $L_{\hbar}$ be as in Eq. (\ref{e.1.11}) . Then $L_{\hbar}^{n}$ is an
operator on $\mathcal{S}$ and
\begin{equation}
L_{\hbar}^{n}=\sum_{\ell=0}^{mn}\left(  -\hbar\right)  ^{\ell}\partial^{\ell
}B_{\ell,\hbar}\left(  \sqrt{\hbar}\left(  \cdot\right)  \right)
\partial^{\ell}, \label{e.5.3}%
\end{equation}
where\footnote{Below, we use $\hat{C}\left(  n,\ell,\mathbf{k},\mathbf{p}%
\right)  =0$ unless $0<\left\vert \mathbf{p}\right\vert =2\left\vert
\mathbf{k}\right\vert -2\ell,$ i.e. $\left\vert \mathbf{k}\right\vert
=\ell+\left\vert \mathbf{p}\right\vert /2.$}
\begin{equation}
B_{\ell,\hbar}:=\mathcal{B}_{\ell,\hbar}+R_{\ell,\hbar}\in\mathbb{R}\left[
x\right]  , \label{e.5.4}%
\end{equation}
\begin{align}
\mathcal{B}_{\ell,\hbar}  &  =\sum_{\mathbf{k}\in\Lambda_{m}^{n}}1_{\left\vert
\mathbf{k}\right\vert =\ell}b_{k_{1},\hbar}b_{k_{2},\hbar}\dots b_{k_{n}%
,\hbar},\label{e.5.5}\\
R_{\ell,\hbar}  &  =\underset{\mathbf{p\in}\Lambda_{2m}^{n}}{\sum
_{\mathbf{k}\in\Lambda_{m}^{n},}}\hat{C}\left(  n,\ell,\mathbf{k}%
,\mathbf{p}\right)  \hbar^{\left\vert \mathbf{p}\right\vert }\left(
\partial^{p_{1}}b_{k_{1},\hbar}\right)  \dots\left(  \partial^{p_{n}}%
b_{k_{n},\hbar}\right)  , \label{e.5.6}%
\end{align}

\end{proposition}

\begin{proof}
Making the replacements $b_{l}\left(  x\right)  \rightarrow\hbar^{l}%
b_{l,\hbar}\left(  \sqrt{\hbar}x\right)  $ in Eqs. (\ref{e.3.17}) and
(\ref{e.3.2}) shows
\begin{equation}
\mathcal{B}_{\ell}\left(  x\right)  \rightarrow\sum_{\mathbf{j}\in\Lambda
_{m}^{n}}1_{\left\vert \mathbf{j}\right\vert =\ell}\left[  \hbar^{j_{1}%
}b_{j_{1},\hbar}\left(  \sqrt{\hbar}x\right)  \dots\hbar^{j_{n}}b_{j_{n}%
,\hbar}\left(  \sqrt{\hbar}x\right)  \right]  =\hbar^{\ell}\mathcal{B}%
_{\ell,\hbar}\left(  \sqrt{\hbar}x\right)  \label{e.5.7}%
\end{equation}
and
\begin{align}
R_{\ell}\left(  x\right)   &  \rightarrow\sum_{\mathbf{k}\in\Lambda_{m}%
^{n},\mathbf{p\in}\Lambda_{2m}^{n}}\hat{C}\left(  n,\ell,\mathbf{k}%
,\mathbf{p}\right)  \left(  -1\right)  ^{\left\vert \mathbf{k}\right\vert
}\hbar^{\left\vert \mathbf{k}\right\vert +\frac{\left\vert \mathbf{p}%
\right\vert }{2}}\left[  \left(  \partial^{p_{1}}b_{k_{1},\hbar}\right)
\dots\left(  \partial^{p_{n}}b_{k_{n},\hbar}\right)  \right]  \left(
\sqrt{\hbar}x\right) \nonumber\\
&  =\sum_{\mathbf{k}\in\Lambda_{m}^{n},\mathbf{p\in}\Lambda_{2m}^{n}}\hat
{C}\left(  n,\ell,\mathbf{k},\mathbf{p}\right)  \left(  -1\right)
^{\left\vert \mathbf{k}\right\vert }\hbar^{\ell+\left\vert \mathbf{p}%
\right\vert }\left[  \left(  \partial^{p_{1}}b_{k_{1},\hbar}\right)
\dots\left(  \partial^{p_{n}}b_{k_{n},\hbar}\right)  \right]  \left(
\sqrt{\hbar}x\right) \nonumber\\
&  =\hbar^{\ell}\sum_{\mathbf{k}\in\Lambda_{m}^{n},\mathbf{p\in}\Lambda
_{2m}^{n}}\hat{C}\left(  n,\ell,\mathbf{k},\mathbf{p}\right)  \left(
-1\right)  ^{\left\vert \mathbf{k}\right\vert }\hbar^{\left\vert
\mathbf{p}\right\vert }\left[  \left(  \partial^{p_{1}}b_{k_{1},\hbar}\right)
\dots\left(  \partial^{p_{n}}b_{k_{n},\hbar}\right)  \right]  \left(
\sqrt{\hbar}x\right) \nonumber\\
&  =\hbar^{\ell}R_{\ell,\hbar}\left(  \sqrt{\hbar}x\right)  . \label{e.5.8}%
\end{align}
Therefore it follows that
\[
B_{\ell}\left(  x\right)  =\mathcal{B}_{\ell}\left(  x\right)  +R_{\ell
}\left(  x\right)  \rightarrow\hbar^{\ell}\left[  \mathcal{B}_{\ell,\hbar
}+R_{\ell,\hbar}\right]  \left(  \sqrt{\hbar}x\right)  =\hbar^{\ell}%
B_{\ell,\hbar}\left(  \sqrt{\hbar}x\right)
\]
where $L_{\hbar}^{n}$ is then given as in Eq. (\ref{e.5.3}).
\end{proof}

\begin{notation}
\label{not.5.4}Let
\begin{align}
\mathcal{L}_{\hbar}^{\left(  n\right)  }  &  =\sum_{\ell=0}^{mn}\left(
-\hbar\right)  ^{\ell}\partial^{\ell}\mathcal{B}_{\ell,\hbar}\left(
\sqrt{\hbar}\left(  \cdot\right)  \right)  \partial^{\ell},\text{
and}\label{e.5.9}\\
\mathcal{R}_{\hbar}^{\left(  n\right)  }  &  =\sum_{\ell=0}^{mn-1}\left(
-\hbar\right)  ^{\ell}\partial^{\ell}R_{\ell,\hbar}\left(  \sqrt{\hbar}\left(
\cdot\right)  \right)  \partial_{}^{\ell}. \label{e.5.10}%
\end{align}
as operators on $\mathcal{S}.$ Then $L_{\hbar}^{n}$ can also be written as
\begin{equation}
L_{\hbar}^{n}=\mathcal{L}_{\hbar}^{\left(  n\right)  }+\mathcal{R}_{\hbar
}^{\left(  n\right)  }\text{ on }\mathcal{S}. \label{e.5.11}%
\end{equation}

\end{notation}

\section{Operator Comparison\label{sec.6}}

The main purpose of this section is to prove Theorem \ref{the.1.18}. First
off, since the inequality symbol $\preceq_{\mathcal{S}}$ appears very often in
this section which is defined in Notation \ref{not.1.14}, let us recall its
definition. If $A$ and $B$ are symmetric operators on $\mathcal{S}$ (see
Definition \ref{def.1.12}), then we say $A\preceq_{\mathcal{S}}B$ if
\[
\left\langle A\psi,\psi\right\rangle \preceq_{\mathcal{S}}\left\langle
B\psi,\psi\right\rangle \text{ for all }\psi\in\mathcal{S}%
\]
where $\left\langle \cdot,\cdot\right\rangle $ is the usual $L^{2}\left(
m\right)  $ -- inner product as in Eq. (\ref{e.1.1}).

\begin{lemma}
\label{lem.6.1}Let $\left\{  c_{\ell}\right\}  _{\ell=0}^{M-1}\subset
\mathbb{R}$ be given constants. Then for any $\delta>0,$ there exists
$C_{\delta}<\infty$ such that
\begin{equation}
\sum_{\ell=0}^{M-1}c_{\ell}\left(  -\hbar\partial^{2}\right)  ^{\ell}%
\preceq_{\mathcal{S}}\delta\left(  -\hbar\partial^{2}\right)  ^{M}+C_{\delta
}I~\forall~\hbar>0. \label{e.6.1}%
\end{equation}

\end{lemma}

\begin{proof}
By conjugating Eq. (\ref{e.6.1}) by the Fourier transform in Eq.(\ref{e.4.23})
(so that $\frac{1}{i}\partial\rightarrow\xi$) and Eq.(\ref{e.4.24}), we may
reduce Eq. (\ref{e.6.1}) to the easily verified statement; for all $\delta>0,$
there exists $C_{\delta}<\infty$ such that
\begin{equation}
\sum_{\ell=0}^{M-1}c_{\ell}w^{\ell}\leq\delta w^{M}+C_{\delta}~\forall~w\geq0.
\label{e.6.2}%
\end{equation}
Here, $w$ is shorthand for $\hbar\xi^{2}.$
\end{proof}

\begin{lemma}
\label{lem.6.2}Let $I\subset\mathbb{R}$ be a compact interval. Suppose
$\left\{  p_{k}\left(  \cdot\right)  \right\}  _{k=0}^{m_{p}}$ and $\left\{
q_{k}\left(  \cdot\right)  \right\}  _{k=0}^{m_{q}}\subset C\left(
I,\mathbb{R}\right)  $ where $m_{p}\in2\mathbb{N}_{0}$ and $m_{q}\in
\mathbb{N}_{0}$ such that
\[
p\left(  x,y\right)  =\sum_{k=0}^{m_{p}}p_{k}\left(  y\right)  x^{k}\text{ and
}q\left(  x,y\right)  =\sum_{k=0}^{m_{q}}q_{k}\left(  y\right)  x^{k}%
\]
and $\delta:=\min_{y\in I}p_{m_{p}}\left(  y\right)  >0.$ If we further assume
$m_{p}>m_{q}$ then for any $\epsilon>0,$ there exists $C_{\epsilon}>0$ such
that we have
\begin{equation}
\left\vert q\left(  x,y\right)  \right\vert \leq\epsilon p\left(  x,y\right)
+C_{\epsilon}\text{ for all }y\in I\text{ and }x\in\mathbb{R}. \label{e.6.3}%
\end{equation}

If $m_{p}=m_{q}$ there exists $D$ and $E>0$ such that we have
\begin{equation}
\left\vert q\left(  x,y\right)  \right\vert \leq Dp\left(  x,y\right)
+E\text{ for all }y\in I\text{ and }x\in\mathbb{R}. \label{e.6.4}%
\end{equation}

\end{lemma}

\begin{proof}
Let $M$ be an upper bound for $\left\vert p_{k}\left(  y\right)  \right\vert $
and $\left\vert q_{l}\left(  y\right)  \right\vert $ for all $y\in I,$ $0\leq
k\leq m_{p}$ and $0\leq l\leq m_{q}.$ Then for any $D>0$ we have,%
\begin{equation}
\left\vert q\left(  x,y\right)  \right\vert -Dp\left(  x,y\right)  \leq
\rho_{D}\left(  x\right)  \label{e.6.5}%
\end{equation}
where
\[
\rho_{D}\left(  x\right)  :=M\sum_{k=0}^{m_{q}}\left\vert x\right\vert
^{k}-D\delta\left\vert x\right\vert ^{m_{p}}+DM\sum_{k=0}^{m_{p}-1}\left\vert
x\right\vert ^{k}.
\]
If $m_{p}>m_{q}$ we see $\lim_{x\rightarrow\pm\infty}\rho_{D}\left(  x\right)
=-\infty$ for all $D=\varepsilon>0$ and hence $C_{\varepsilon}:=\max
_{x\in\mathbb{R}}\rho_{\varepsilon}\left(  x\right)  <\infty$ which combined
with Eq. (\ref{e.6.5}) proves Eq. (\ref{e.6.3}). If $m_{p}=m_{q}$ and $D$ is
chosen so that $D\delta>M,$ we again will have $\lim_{x\rightarrow\pm\infty
}\rho_{D}\left(  x\right)  =-\infty$ and so $E:=\max_{x\in\mathbb{R}}\rho
_{D}\left(  x\right)  <\infty$ which combined with Eq. (\ref{e.6.5}) proves
Eq. (\ref{e.6.4}).
\end{proof}

\begin{lemma}
\label{lem.6.3}Suppose that $\left\{  b_{l,\hbar}(\cdot)\right\}  _{l=0}^{m}$
and $\eta>0$ satisfy Assumption \ref{ass.1} and $c_{b_{m}}>0$ is the constant
in Eq. (\ref{e.1.13}). Let $n\in\mathbb{N}$ and $\left\{  B_{\ell,\hbar
}\right\}  _{\ell=0}^{mn}$ and $\left\{  \mathcal{B}_{\ell,\hbar}\right\}
_{\ell=0}^{mn}$ be the polynomials defined in Eqs. (\ref{e.5.4}) and
(\ref{e.5.5}) respectively. Then $\left\{  B_{\ell,\hbar}\right\}  _{\ell
=0}^{mn}$ and $\left\{  \mathcal{B}_{\ell,\hbar}\right\}  _{\ell=0}^{mn}$
satisfy items 1 and 3 of Assumption \ref{ass.1} and in particular,
\begin{equation}
B_{mn,\hbar}=\mathcal{B}_{mn,\hbar}=b_{m,\hbar}^{n}\geq\left(  c_{b_{m}%
}\right)  ^{n}. \label{e.6.11}%
\end{equation}
Moreover, if $R_{\ell,\hbar}$ is the polynomial in Eq. (\ref{e.5.6}), then for
any $\epsilon>0$ there exists $C_{\epsilon}>0$ such that
\begin{equation}
\left\vert R_{\ell,\hbar}\left(  x\right)  \right\vert \leq\epsilon
\mathcal{B}_{\ell,\hbar}\left(  x\right)  +C_{\epsilon}\text{ }\forall\text{
}x\in\mathbb{R},\mathbb{~}0<\hbar<\eta,~\&~\text{ }0\leq\ell\leq mn.
\label{e.6.12}%
\end{equation}

\end{lemma}

\begin{proof}
From Eq. (\ref{e.5.5})
\[
\mathcal{B}_{\ell,\hbar}=\sum_{\mathbf{k}\in\Lambda_{m}^{n}}1_{\left\vert
\mathbf{k}\right\vert =\ell}b_{k_{1},\hbar}b_{k_{2},\hbar}\dots b_{k_{n}%
,\hbar}
\]
from which it easily follows that $\mathcal{B}_{\ell,\hbar}$ is a real
polynomial with real valued coefficients depending continuously on $\hbar.$
Thus we have verified that the $\left\{  \mathcal{B}_{\ell,\hbar}\right\}
_{\ell=0}^{mn}$ satisfy item 1. of Assumption \ref{ass.1}.

The highest order coefficient of the polynomial $\mathcal{B}_{\ell,\hbar}$ is
a linear combination of $n$-fold products among the highest order coefficients
of $\left\{  b_{l,\hbar}(x)\right\}  _{l=0}^{m}$ and hence is still bounded
from below by a positive constant independent of $\hbar\in\left(
0,\eta\right)  .$ This observation along with the estimate, $\mathcal{B}%
_{mn,\hbar}=b_{m,\hbar}^{n}\geq c_{b_{m}}^{n},$ shows $\left\{  \mathcal{B}%
_{\ell,\hbar}\right\}  _{\ell=0}^{mn}$ also satisfies item 3 of Assumption
\ref{ass.1}.

Applying Proposition \ref{pro.5.1} with $b_{l}\left(  x\right)  \rightarrow
\hbar^{l}b_{l,\hbar}\left(  \sqrt{\hbar}x\right)  ,$ it follows (making use of
Eqs. (\ref{e.5.7}) and (\ref{e.5.8})) that
\begin{equation}
\deg R_{\ell,\hbar}\leq\deg\mathcal{B}_{\ell,\hbar}-2\text{ for }0\leq\ell\leq
mn\text{ and }0<\hbar<\eta. \label{e.6.13}%
\end{equation}

Since the leading order coefficient of $B_{\ell,\hbar}$ is a continuous
function of $\hbar$ which satisfies condition $3$ of Assumption \ref{ass.1},
we may conclude that the degree estimate above also holds at $\hbar=0$ and
$\hbar=\eta.$ We now apply Lemma \ref{lem.6.2} with $p\left(  x,\hbar\right)
=\mathcal{B}_{\ell,\hbar}\left(  x\right)  $ (note $\mathcal{B}_{\ell,\hbar}$
satisfies items 1 and 3 of Assumption \ref{ass.1}), $q\left(  x,\hbar\right)
=R_{\ell,\hbar}\left(  x\right)  ,$ and $I=\left[  0,\eta\right]  $ to
conclude Eq. (\ref{e.6.12}) holds.

Finally, let $0<\hbar<\eta$ be fixed. We learn $B_{\ell,\hbar}=\mathcal{B}%
_{\ell,\hbar}+R_{\ell,\hbar}$ from Eqs.(\ref{e.5.4}). It is clear that
$\left\{  B_{\ell,\hbar}\right\}  _{\ell=0}^{mn}$ satisfies the item $1$ of
Assumption \ref{ass.1}. From Eq. (\ref{e.6.13}), the highest order coefficient
of $B_{\ell,\hbar}$ and $\mathcal{B}_{\ell,\hbar}$ are the same and
Proposition \ref{pro.3.10} shows that $R_{mn,\hbar}=0$ which implies
$B_{mn,\hbar}=\mathcal{B}_{mn,\hbar}.$ Therefore $\left\{  B_{\ell,\hbar
}\right\}  _{\ell=0}^{mn}$ also satisfies the item 3 of Assumption \ref{ass.1}.
\end{proof}

\subsection{Estimating the quadratic form associated to $L_{\hbar}^{n}%
$\label{sec.6.1}}

\begin{theorem}
\label{the.6.4}Supposed $\left\{  b_{l,\hbar}\left(  x\right)  \right\}  $ and
$\eta>0$ satisfies Assumption \ref{ass.1} and let $L_{\hbar}$ and
$\mathcal{L}_{\hbar}^{\left(  n\right)  }$ be the operators in Eqs.
(\ref{e.1.11}) and (\ref{e.5.9}) respectively. Then for any $n\in\mathbb{N},$
there exists $C_{n}<\infty$ so that for all $0<\hbar<\eta$ and $c>C_{n}%
$\thinspace$;$%
\[
\frac{3}{2}\left(  \mathcal{L}_{\hbar}^{\left(  n\right)  }+c\right)
\succeq_{\mathcal{S}}L_{\hbar}^{n}+c\succeq_{\mathcal{S}}\frac{1}{2}\left(
\mathcal{L}_{\hbar}^{\left(  n\right)  }+c\right)
\]
and both $\mathcal{L}_{\hbar}^{\left(  n\right)  }+c$ and $L_{\hbar}^{n}+c$
are positive operators.
\end{theorem}

\begin{proof}
Let $\psi\in\mathcal{S}$ and $0<\hbar<\eta$. From Eqs. (\ref{e.5.10}) and
(\ref{e.5.11}) we can conclude
\[
|\langle\left(  L_{\hbar}^{n}-\mathcal{L}_{\hbar}^{\left(  n\right)  }\right)
\psi,\psi\rangle|=\left\vert \left\langle \mathcal{R}_{\hbar}^{\left(
n\right)  }\psi,\psi\right\rangle \right\vert \leq\sum_{\ell=0}^{nm-1}%
\hbar^{\ell}\left\vert \left\langle R_{\ell,\hbar}\left(  \sqrt{\hbar}\left(
\cdot\right)  \right)  \partial^{\ell}\psi,\partial^{\ell}\psi\right\rangle
\right\vert .
\]
From Eq. (\ref{e.6.12}) in Lemma \ref{lem.6.3} by taking $\epsilon=\frac{1}%
{2}$ and $C_{\epsilon}=C_{\frac{1}{2}}$ we have
\begin{equation}
\left\vert R_{\ell,\hbar}\left(  x\right)  \right\vert \leq\frac{1}%
{2}\mathcal{B}_{\ell,\hbar}\left(  x\right)  +C_{\frac{1}{2}} \label{e.6.14}%
\end{equation}
for all $0\leq\ell\leq mn-1$ and $\hbar\in\left(  0,\eta\right)  .$ With the
use of Eq.(\ref{e.6.14}), we learn
\begin{align}
&  |\langle\left(  L_{\hbar}^{n}-\mathcal{L}_{\hbar}^{\left(  n\right)
}\right)  \psi,\psi\rangle|\nonumber\\
&  \quad\leq\sum_{\ell=0}^{nm-1}\hbar^{\ell}\left\langle \left(  \frac{1}%
{2}\mathcal{B}_{\ell,\hbar}\left(  \sqrt{\hbar}\left(  \cdot\right)  \right)
+C_{\frac{1}{2}}\right)  \partial^{\ell}\psi,\partial^{\ell}\psi\right\rangle
\nonumber\\
&  \quad=\frac{1}{2}\sum_{\ell=0}^{nm-1}\left\langle \left(  -\hbar\right)
^{\ell}\partial^{\ell}\mathcal{B}_{\ell,\hbar}\left(  \sqrt{\hbar}\left(
\cdot\right)  \right)  \partial^{\ell}\psi,\psi\right\rangle +C_{\frac{1}{2}%
}\left\langle \sum_{\ell=0}^{nm-1}\left(  -\hbar\right)  ^{\ell}%
\partial^{2\ell}\psi,\psi\right\rangle . \label{e.6.15}%
\end{align}
By Eq. (\ref{e.6.11}) in Lemma \ref{lem.6.3}, we have
\[
\mathcal{B}_{mn,\hbar}=b_{m,\hbar}^{n}\geq c_{b_{m}}^{n}>0.
\]
So making use of Lemma \ref{lem.6.1} by taking $\delta=c_{b_{m}}^{n}/2$ and
$c_{\ell}=C_{\frac{1}{2}}$, there exists $C_{\delta}<\infty$ such that for all
$0<\hbar<\eta$ and $\psi\in\mathcal{S},$
\begin{align*}
C_{\frac{1}{2}}  &  \left\langle \sum_{\ell=0}^{nm-1}\left(  -\hbar\right)
^{\ell}\partial^{2\ell}\psi,\psi\right\rangle \\
&  \quad\leq\left(  -\hbar\right)  ^{mn}\left\langle \delta\partial^{mn}%
\psi,\partial^{mn}\psi\right\rangle +\frac{1}{2}C_{\delta}\left\langle
\psi,\psi\right\rangle \\
&  \quad\leq\frac{1}{2}\left(  -\hbar\right)  ^{mn}\left\langle \mathcal{B}%
_{mn,\hbar}\left(  \sqrt{\hbar}\left(  \cdot\right)  \right)  \partial
^{mn}\psi,\partial^{mn}\psi\right\rangle +\frac{1}{2}C_{\delta}\left\langle
\psi,\psi\right\rangle .
\end{align*}
By combining Eqs. (\ref{e.6.15}) and (\ref{e.6.1}), we get
\begin{align}
&  |\langle\psi,\left(  L_{\hbar}^{n}-\mathcal{L}_{\hbar}^{\left(  n\right)
}\right)  \psi\rangle|\nonumber\\
&  \quad\leq\frac{1}{2}\left\langle \left(  \sum_{\ell=0}^{nm-1}\left(
-\hbar\right)  ^{\ell}\partial^{\ell}\mathcal{B}_{\ell,\hbar}\left(
\sqrt{\hbar}\left(  \cdot\right)  \right)  \partial^{\ell}+\left(
-\hbar\right)  ^{mn}\partial^{mn}\mathcal{B}_{mn,\hbar}\left(  \sqrt{\hbar
}\left(  \cdot\right)  \right)  \partial^{mn}\right)  \psi,\psi\right\rangle
\nonumber\\
&  \qquad+\frac{1}{2}C_{\delta}\left\langle \psi,\psi\right\rangle \nonumber\\
&  \quad=\frac{1}{2}\left\langle \left(  \mathcal{L}_{\hbar}^{\left(
n\right)  }+C_{\delta}\right)  \psi,\psi\right\rangle . \label{e.6.16}%
\end{align}
It is easy to conclude that
\[
\langle\left(  \mathcal{L}_{\hbar}^{\left(  n\right)  }+C_{\delta}\right)
\psi,\psi\rangle\geq0\text{.}%
\]
As a result, for all $c>C_{\delta}$, $0<\hbar<\eta$, by the Eq. (\ref{e.6.16}%
), we get
\[
\left\vert \langle\left[  \left(  L_{\hbar}^{n}+c\right)  -\left(
\mathcal{L}_{\hbar}^{\left(  n\right)  }+c\right)  \right]  \psi,\psi
\rangle\right\vert =\left\vert \langle\left(  L_{\hbar}^{n}-\mathcal{L}%
_{\hbar}^{\left(  n\right)  }\right)  \psi,\psi\rangle\right\vert \leq\frac
{1}{2}\langle\left(  \mathcal{L}_{\hbar}^{\left(  n\right)  }+c\right)
\psi,\psi\rangle
\]
and the desired result follows.
\end{proof}

\subsection{Proof of the operator comparison Theorem \ref{the.1.18}}

The purpose of this subsection is to prove Theorem \ref{the.1.18}. We begin
with a preparatory lemma whose proof requires the following notation.

\begin{notation}
\label{not.6.5}For any divergence form differential operator $L$ on
$\mathcal{S}$ described as in Eq. (\ref{e.2.1}) we may decompose $L$ into its
top order and lower order pieces, $L=L^{top}+L^{<}$ where
\begin{equation}
L^{top}:=\left(  -1\right)  ^{m}\partial^{m}M_{b_{m}}\partial^{m}\text{ and
}L^{<}:=\sum_{l=0}^{m-1}\left(  -1\right)  ^{l}\partial^{l}M_{b_{l}}%
\partial^{l}. \label{e.6.17}%
\end{equation}

\end{notation}

\begin{lemma}
\label{lem.6.6}Let $\left\{  \mathcal{B}_{\ell,\hbar}\left(  x\right)
\right\}  _{\ell=0}^{M}$ be polynomial functions depending continuously on
$\hbar$ which satisfies the conditions $1$ and $3$ of Assumption \ref{ass.1}
and so in particular,
\begin{equation}
c_{\mathcal{B}_{M}}:=\inf\left\{  \mathcal{B}_{M,\hbar}\text{$\left(
\sqrt{\hbar}x\right)  $}:x\in\mathbb{R}\text{ and }0<\hbar<\eta\right\}  >0.
\label{e.6.18}%
\end{equation}
If
\[
K_{\hbar}=\sum_{\ell=0}^{M}\left(  -\hbar\right)  ^{\ell}\partial^{2\ell
}\text{ and }\mathcal{L}_{\hbar}=\sum_{\ell=0}^{M}\left(  -\hbar\right)
^{\ell}\partial^{\ell}M_{\mathcal{B}_{\ell,\hbar}\left(  \sqrt{\hbar}\left(
\cdot\right)  \right)  }\partial^{\ell}
\]
are operators on $\mathcal{S}$ then for all $\gamma>\frac{1}{c_{\mathcal{B}%
_{M}}},$ there exists $C_{\gamma}<\infty$ such that
\begin{equation}
K_{\hbar}\preceq_{\mathcal{S}}\gamma\mathcal{L}_{\hbar}+C_{\gamma}I.
\label{e.6.19}%
\end{equation}

\end{lemma}

\begin{proof}
Using the conditions $1$ and $3$ of Assumption \ref{ass.1} on $\left\{
\mathcal{B}_{\ell,\hbar}\left(  x\right)  \right\}  _{\ell=0}^{M}$ where
$b_{l,\hbar}$ is replaced by $\mathcal{B}_{\ell,\hbar},$ we may choose $E>0$
such that for all $0\leq\ell<M,$
\[
c_{\ell}:=\inf\left\{  \mathcal{B}_{\ell,\hbar}\text{$\left(  \sqrt{\hbar
}x\right)  $}+E:x\in\mathbb{R}\text{ and }0<\hbar<\eta\right\}  >0
\]
and therefore,
\[
\left(  -\hbar\right)  ^{\ell}\partial^{\ell}M_{\mathcal{B}_{\ell,\hbar
}\left(  \sqrt{\hbar}\left(  \cdot\right)  \right)  }\partial^{\ell}+E\left(
-\hbar\right)  ^{\ell}\partial^{2\ell}=\left(  -\hbar\right)  ^{\ell}%
\partial^{\ell}M_{\left[  \mathcal{B}_{\ell,\hbar}\left(  \sqrt{\hbar}\left(
\cdot\right)  \right)  +E\right]  }\partial^{\ell}\geq0~\forall~\ell
\]
and in particular $\mathcal{L}_{\hbar}^{<}+EK_{\hbar}^{<}\succeq_{\mathcal{S}%
}0$ which in turn implies
\[
\mathcal{L}_{\hbar}^{top}\preceq_{\mathcal{S}}\mathcal{L}_{\hbar}%
^{top}+\mathcal{L}_{\hbar}^{<}+EK_{\hbar}^{<}=\mathcal{L}_{\hbar}+EK_{\hbar
}^{<}.
\]
Using this observation and Eq. (\ref{e.6.18}) we find,
\begin{equation}
K_{\hbar}^{top}=\left(  -\hbar\right)  ^{M}\partial^{2M}\preceq_{\mathcal{S}%
}\frac{1}{c_{\mathcal{B}_{M}}}\mathcal{L}_{\hbar}^{top}\preceq_{\mathcal{S}%
}\frac{1}{c_{\mathcal{B}_{M}}}\left(  \mathcal{L}_{\hbar}+EK_{\hbar}%
^{<}\right)  . \label{e.6.20}%
\end{equation}

By Lemma \ref{lem.6.1} and Eq. (\ref{e.6.20}), for any $\delta>0,$ there
exists $C_{\delta}<\infty$ such that for all $\hbar>0,$
\begin{equation}
K_{\hbar}^{<}\preceq_{\mathcal{S}}\delta K_{\hbar}^{top}+C_{\delta}%
I\preceq_{\mathcal{S}}\delta\frac{1}{c_{\mathcal{B}_{M}}}\left(
\mathcal{L}_{\hbar}+EK_{\hbar}^{<}\right)  +C_{\delta}I. \label{e.6.21}%
\end{equation}
Given $\varepsilon>0$ small we may use the previous equation with $\delta>0$
chosen so that $\varepsilon\geq\frac{\delta}{c_{\mathcal{B}_{M}}-\delta E}$ to
learn there exists $C_{\varepsilon}^{\prime}<\infty$ such that
\begin{equation}
K_{\hbar}^{<}\preceq_{\mathcal{S}}\varepsilon\mathcal{L}_{\hbar}%
+C_{\varepsilon}^{\prime}I. \label{e.6.22}%
\end{equation}
Combining this inequality with Eq. (\ref{e.6.20}) then shows,
\[
K_{\hbar}=K_{\hbar}^{top}+K_{\hbar}^{<}\preceq_{\mathcal{S}}\frac
{1}{c_{\mathcal{B}_{M}}}\left(  \mathcal{L}_{\hbar}+E\left(  \varepsilon
\mathcal{L}_{\hbar}+C_{\varepsilon}^{\prime}I\right)  \right)  +\varepsilon
\mathcal{L}_{\hbar}+C_{\varepsilon}^{\prime}I.
\]
Thus choosing $\varepsilon>0$ sufficiently small in this inequality allows us
to conclude for every $\gamma>\frac{1}{c_{\mathcal{B}_{M}}}$ there exists
$C_{\gamma}<\infty$ such that Eq. (\ref{e.6.19}) holds.
\end{proof}

\begin{comments}
The next lemma has been replaced by lemma \ref{lem.6.2}.

\begin{lemma}
\label{lem.6.7}Suppose $\left\{  b_{l,\hbar}\left(  x\right)  \right\}
_{l=0}^{m}\subset\mathbb{R}\left[  x\right]  $ and $\eta>0$ satisfies
Assumption \ref{ass.1}, $\left(  k_{1},k_{2},\dots,k_{n}\right)  \in
\Lambda_{m}^{n},$ and $N:=\sum_{i=1}^{n}\deg b_{k_{i},\hbar}\in2\mathbb{N}.$
Then there exists finite constants, $\,C_{1}$ and $C_{2}$ independent of
$\hbar\in\left(  0,\eta\right)  $ such that
\begin{equation}
\left\vert x\right\vert ^{\ell}\leq\left(  1+x^{N}\right)  \leq C_{1}\left(
b_{k_{1},\hbar}\dots b_{k_{n},\hbar}\right)  \left(  x\right)  +C_{2}
\label{e.6.23}%
\end{equation}
holds for all $x\in\mathbb{R},$ $0\leq\ell\leq N,$ and $0<\hbar<\eta.$
\end{lemma}

\begin{proof}
By the assumptions on $\left\{  b_{l,\hbar}(x)\right\}  _{l=0}^{m}$ we easily
deduce there exists $C<\infty$ independent of $\hbar$ such that
\[
\left\vert \left(  b_{k_{1},\hbar}\dots b_{k_{n},\hbar}\right)  \left(
x\right)  \right\vert \geq\alpha x^{N}-C\left(  1+\left\vert x\right\vert
^{N-1}\right)  ,
\]
where $\alpha:=c_{k_{1}}\dots c_{k_{n}}>0$ and the $\left\{  c_{l}\right\}
_{l}$ are\marginpar{There are not $c_{l}$ in Assumption \ref{ass.1}?} as in
Assumption \ref{ass.1}. Simple calculus then shows there exists $C^{\prime
}<\infty$ such that $1+\left\vert x\right\vert ^{N-1}\leq\alpha\left\vert
x\right\vert ^{N}/2+C^{\prime}.$Combining this estimate with the last
displayed equation shows,
\[
\left\vert \left(  b_{k_{1},\hbar}\dots b_{k_{n},\hbar}\right)  \left(
x\right)  \right\vert \geq\frac{\alpha}{2}x^{N}-C\left(  1+C^{\prime}\right)
\]
which is equivalent to Eq. (\ref{e.6.23}).\qed

\end{proof}
\end{comments}

We are now ready to give the proof of Theorem \ref{the.1.18}.

\begin{proof}
[Proof of Theorem \ref{the.1.18}]Recall $\eta:=\min\left\{  \eta
_{\widetilde{L}},\eta_{L}\right\}  $ defined in Theorem \ref{the.1.18}. By the
assumption in Eq. (\ref{e.1.16}) of Theorem \ref{the.1.18},
\[
\left\vert \tilde{b}_{l,\hbar}\left(  x\right)  \right\vert \leq c_{1}\left(
b_{l,\hbar}\left(  x\right)  +c_{2}\right)  ~\forall~0\leq l\leq m_{\tilde{L}%
}~and~0<\hbar<\eta.
\]
Moreover, using items 1 and 3 of Assumption \ref{ass.1}, by increasing the
size of $c_{2}$ if necessary, we may further assume that $b_{l,\hbar}\left(
x\right)  +c_{2}\geq0$ for all $x\in\mathbb{R},$ $0\leq l\leq m_{L}$, and
$0<\hbar<\eta.$ Without loss of generality, we may define $\tilde{b}_{l,\hbar
}\left(  \cdot\right)  \equiv0$ for all $l>m_{\tilde{L}}$ and hence
$\mathcal{\tilde{B}}_{\ell,\hbar}\left(  \cdot\right)  =0$ for all
$\ell>m_{\tilde{L}}n.$ It then follows that there exists $E_{1},E_{2}<\infty$
such that for $0\leq\ell\leq m_{L}n$,
\begin{align}
\mathcal{\tilde{B}}_{\ell,\hbar}  &  \leq\left\vert \mathcal{\tilde{B}}%
_{\ell,\hbar}\right\vert \leq\sum_{\mathbf{k}\in\Lambda_{m_{\tilde{L}}}^{n}%
}1_{\left\vert \mathbf{k}\right\vert =\ell}\left\vert \tilde{b}_{k_{1},\hbar
}\dots\tilde{b}_{k_{n},\hbar}\right\vert \nonumber\\
&  \leq\sum_{\mathbf{k}\in\Lambda_{m_{\tilde{L}}}^{n}}1_{\left\vert
\mathbf{k}\right\vert =\ell}c_{1}^{n}(b_{k_{1},\hbar}+c_{2})\dots
(b_{k_{n},\hbar}+c_{2})\nonumber\\
&  \leq\sum_{\mathbf{k}\in\Lambda_{m_{L}}^{n}}1_{\left\vert \mathbf{k}%
\right\vert =\ell}c_{1}^{n}(b_{k_{1},\hbar}+c_{2})\dots(b_{k_{n},\hbar}%
+c_{2})\nonumber\\
&  =E_{1}\mathcal{B}_{\ell,\hbar}+E_{2}, \label{e.6.24}%
\end{align}
wherein we have used Eq.(\ref{e.6.4}) in Lemma \ref{lem.6.2} for the last
inequality by taking
\begin{align*}
p\left(  x,\hbar\right)   &  =\mathcal{B}_{\ell,\hbar}=\sum_{\mathbf{k}%
\in\Lambda_{m_{L}}^{n}}1_{\left\vert \mathbf{k}\right\vert =\ell}%
(b_{k_{1},\hbar}\dots b_{k_{n},\hbar})\text{ and}\\
q\left(  x,\hbar\right)   &  =\sum_{\mathbf{k}\in\Lambda_{m_{L}}^{n}%
}1_{\left\vert \mathbf{k}\right\vert =\ell}c_{1}^{n}(b_{k_{1},\hbar}%
+c_{2})\dots(b_{k_{n},\hbar}+c_{2}),
\end{align*}
where (by Lemma \ref{lem.6.3}) $\mathcal{B}_{\ell,\hbar}$ is an even degree
polynomial with a positive leading order coefficient. Hence if we let
$\mathcal{\tilde{L}}_{\hbar}^{\left(  n\right)  }$ and $\mathcal{L}_{\hbar
}^{\left(  n\right)  }$ be as in Eq. (\ref{e.5.9}), i.e.
\[
\mathcal{\tilde{L}}_{\hbar}^{\left(  n\right)  }=\sum_{\ell=0}%
^{m_{\widetilde{L}}n}\left(  -\hbar\right)  ^{\ell}\partial^{\ell
}\widetilde{\mathcal{B}}_{\ell,\hbar}\left(  \sqrt{\hbar}\left(  \cdot\right)
\right)  \partial^{\ell}\text{ and }\mathcal{L}_{\hbar}^{\left(  n\right)
}=\sum_{\ell=0}^{m_{L}n}\left(  -\hbar\right)  ^{\ell}\partial^{\ell
}\mathcal{B}_{\ell,\hbar}\left(  \sqrt{\hbar}\left(  \cdot\right)  \right)
\partial^{\ell},
\]
then it follows directly from Eq.(\ref{e.6.24}) that
\[
\mathcal{\tilde{L}}_{\hbar}^{\left(  n\right)  }\preceq_{\mathcal{S}}%
E_{1}\mathcal{L}_{\hbar}^{\left(  n\right)  }+E_{2}K_{\hbar}\text{ where
}K_{\hbar}:=\sum_{\ell=0}^{nm_{L}}\left(  -\hbar\right)  ^{\ell}%
\partial^{2\ell}.
\]
Because of Lemma \ref{lem.6.3}, we may apply Lemma \ref{lem.6.6} with
$M=nm_{L}$ and $\mathcal{L}_{\hbar}=\mathcal{L}_{\hbar}^{\left(  n\right)  }$
to conclude there exists $\gamma>0$ and $C<\infty$ such that $K_{\hbar}%
\preceq_{\mathcal{S}}\gamma\mathcal{L}_{\hbar}^{\left(  n\right)  }+CI$ and
thus,
\[
\mathcal{\tilde{L}}_{\hbar}^{\left(  n\right)  }\preceq_{\mathcal{S}}\left(
E_{1}+\gamma E_{2}\right)  \mathcal{L}_{\hbar}^{\left(  n\right)  }+E_{2}CI.
\]

By Theorem \ref{the.6.4}, there exists $C_{L}$ and $C_{\tilde{L}}$ such that
\begin{align*}
\frac{1}{2}\mathcal{L}_{\hbar}^{\left(  n\right)  }  &  \preceq_{\mathcal{S}%
}L_{\hbar}^{n}+C_{L}\text{ and}\\
\tilde{L}_{\hbar}^{n}  &  \preceq_{\mathcal{S}}\frac{3}{2}\left(
\tilde{\mathcal{L}}_{\hbar}^{\left(  n\right)  }+C_{\tilde{L}}\right)
\preceq_{\mathcal{S}}\frac{3}{2}\left(  \left(  E_{1}+\gamma E_{2}\right)
\mathcal{L}_{\hbar}^{\left(  n\right)  }+E_{2}CI+C_{\tilde{L}}\right)  .
\end{align*}
From these last two inequalities, it follows that $\tilde{L}_{\hbar}%
^{n}\preceq_{\mathcal{S}}C_{1}\left(  L_{\hbar}^{n}+C_{2}\right)  $ for
appropriately chosen constants $C_{1}$ and $C_{2}.$
\end{proof}

\subsection{Proof of Corollary \ref{cor.1.20}\label{sec.6.3}}

For the reader's convenience let us restated Corollary \ref{cor.1.20} here.

\begin{corollary*}
[\ref{cor.1.20}]Supposed $\left\{  b_{l,\hbar}\left(  x\right)  \right\}
_{l=0}^{m}\subset\mathbb{R}\left[  x\right]  $ and $\eta>0$ satisfies
Assumption \ref{ass.1}, $L_{\hbar}$ is the operator in the Eq. (\ref{e.1.11}),
and suppose that $C\geq0$ has been chosen so that $0\preceq_{\mathcal{S}%
}L_{\hbar}+CI$ for all $0<\hbar<\eta.$ (The existence of $C$ is guaranteed by
Corollary \ref{cor.1.19}.) Then for any $0<\hbar<\eta,$ $\bar{L}_{\hbar}+CI$
is a non-negative self-adjoint operator on $L^{2}\left(  m\right)  $ and
$\mathcal{S}$ is a core for $\left(  \bar{L}_{\hbar}+C\right)  ^{r}$ for all
$r\geq0.$
\end{corollary*}

Before proving this corollary we need to develop a few more tools. From Lemma
\ref{lem.6.3}, $\left\{  B_{\ell,\hbar}\right\}  _{\ell=0}^{mn}\subset
\mathbb{R}\left[  x\right]  $ in Eq. (\ref{e.3.1}) satisfies both items 1 and
3 of Assumption \ref{ass.1}. Therefore, $B_{\ell,\hbar}$ is bounded below for
$0\leq\ell\leq mn-1$ and $B_{mn,\hbar}>0.$ We may choose $C>0$ sufficiently
large so that
\begin{equation}
B_{\ell,\hbar}+C>0\text{ for }0\leq\ell\leq mn-1\text{ and }0<\hbar<\eta.
\label{e.6.25}%
\end{equation}

\begin{notation}
\label{not.6.8}Let $C>0$ be chosen so that Eq. (\ref{e.6.25}) holds and then
define the operator, $\hat{L}_{\hbar},$ by
\begin{align*}
\hat{L}_{\hbar}:=  &  \sum_{\ell=0}^{mn}\left(  -\hbar\right)  ^{\ell}%
\partial^{\ell}\left(  B_{\ell,\hbar}\left(  \sqrt{\hbar}\left(  \cdot\right)
\right)  +C1_{\ell<mn}\right)  \partial^{\ell}\\
=  &  \left(  -\hbar\right)  ^{mn}\partial^{mn}B_{mn,\hbar}\left(  \sqrt
{\hbar}\left(  \cdot\right)  \right)  \partial^{mn}+\sum_{\ell=0}%
^{mn-1}\left(  -\hbar\right)  ^{\ell}\partial^{\ell}\left(  B_{\ell,\hbar
}\left(  \sqrt{\hbar}\left(  \cdot\right)  \right)  +C\right)  \partial^{\ell}%
\end{align*}
with domain, $\mathcal{D}\left(  \hat{L}_{\hbar}\right)  =\mathcal{S}.$
\end{notation}

\begin{lemma}
\label{lem 6.6.9}There exists $\tilde{C}_{1}$ and $\tilde{C}_{2}>0$ such that
\[
\left\Vert \hbar\partial^{2M}\psi\right\Vert \leq\tilde{C}_{1}\left\Vert
\hat{L}_{\hbar}\psi\right\Vert +\tilde{C}_{2}\left\Vert \psi\right\Vert
\]
holds for all $0\leq M\leq mn$, $0<\hbar<\eta,$ and $\psi\in\mathcal{S}.$
\end{lemma}

\begin{proof}
As in Eq. (\ref{e.4.23}), let $\hat{\psi}$ denote the Fourier transform of
$\psi\in\mathcal{S}$ and recall that $\left\Vert \psi\right\Vert =\sqrt{2\pi
}\left\Vert \hat{\psi}\right\Vert .$ Hence it follows,
\begin{align}
\left\Vert \hbar^{M}\partial^{2M}\psi\right\Vert  &  =\sqrt{2\pi}\left\Vert
\hbar^{M}\xi^{2M}\hat{\psi}\left(  \xi\right)  \right\Vert \nonumber\\
&  \leq\sqrt{2\pi}\left\Vert \left(  \sum_{\ell=0}^{M}\hbar^{\ell}\xi^{2\ell
}\right)  \hat{\psi}\left(  \xi\right)  \right\Vert =\left\Vert \left(
\sum_{\ell=0}^{M}\left(  -\hbar\right)  ^{\ell}\partial^{2\ell}\right)
\psi\right\Vert . \label{e.6.26}%
\end{align}
With the same $C$ in Notation \ref{not.6.8} and using Eq. (\ref{e.6.25}), we
can see that
\[
1\leq\left(  B_{\ell,\hbar}+C1_{\ell<mn}\right)  +1~\forall~0\leq\ell\leq
mn\text{ \& }0<\hbar<\eta.
\]
Therefore applying the operator comparison Theorem \ref{the.1.18} with
$\widetilde{L}_{\hbar}=\sum_{\ell=0}^{M}\left(  -\hbar\right)  ^{\ell}%
\partial^{2\ell},$ $L_{\hbar}=\hat{L}_{\hbar},$ and $n=2,$ there exists
$C_{1}$ and $C_{2}>0$ such that for%
\[
\left\langle \left(  \sum_{\ell=0}^{M}\left(  -\hbar\right)  ^{\ell}%
\partial^{2\ell}\right)  ^{2}\psi,\psi\right\rangle \leq C_{1}\left\langle
\hat{L}_{\hbar}^{2}\psi,\psi\right\rangle +C_{2}\left\langle \psi
,\psi\right\rangle \text{ }\forall~\text{ }\psi\in\mathcal{S}\text{ \&
}0<\hbar<\eta.
\]
Combining this inequality with Eq. (\ref{e.6.26}), shows there exists other
constants $\widetilde{C}_{1}$ and $\widetilde{C}_{2}>0$ such that
\[
\left\Vert \hbar^{M}\partial^{2M}\psi\right\Vert \leq\left\Vert \left(
\sum_{\ell=0}^{M}\left(  -\hbar\right)  ^{\ell}\partial^{2\ell}\right)
\psi\right\Vert \leq\widetilde{C}_{1}\left\Vert \hat{L}_{\hbar}\psi\right\Vert
+\widetilde{C}_{2}\left\Vert \psi\right\Vert .
\]

\end{proof}

\begin{lemma}
\label{lem.6.10}Let $A$ and $B$ be closed operators on a Hilbert space
$\mathcal{K}$ and suppose there exists a subspace, $S\subseteq\mathcal{D}%
\left(  A\right)  \cap\mathcal{D}\left(  B\right)  ,$ such that $S$ is dense
and $S$ is a core of $B.$ If there exists a constant $C>0$ such that
\begin{equation}
\left\Vert A\psi\right\Vert \leq\left\Vert B\psi\right\Vert +C\left\Vert
\psi\right\Vert \text{ }\forall\text{ }\psi\in S, \label{e.6.27}%
\end{equation}
then $\mathcal{D}\left(  B\right)  \subseteq\mathcal{D}\left(  A\right)  $
and
\begin{equation}
\left\Vert A\psi\right\Vert \leq\left\Vert B\psi\right\Vert +C\left\Vert
\psi\right\Vert \text{ }\forall\text{ }\psi\in\mathcal{D}\left(  B\right)  .
\label{e.6.28}%
\end{equation}

\end{lemma}

\begin{proof}
If $\psi\in\mathcal{D}\left(  B\right)  $, there exists $\psi_{k}\in S$ such
that $\psi_{k}\rightarrow\psi$ and $B\psi_{k}\rightarrow B\psi$ as
$k\rightarrow\infty$. Because of Eq. (\ref{e.6.27}) $\left\{  A\psi
_{k}\right\}  _{k=1}^{\infty}$ is Cauchy in $\mathcal{K}$ and hence
convergent. As $A$ is closed we may conclude that $\psi\in\mathcal{D}\left(
A\right)  $ and that $\lim_{k\rightarrow\infty}A\psi_{k}=A\psi.$ Therefore Eq.
(\ref{e.6.28}) holds by replacing $\psi$ in Eq. (\ref{e.6.27}) by $\psi_{k}$
and then passing to the limit as $k\rightarrow\infty.$
\end{proof}

\begin{proposition}
\label{pro.6.11}Suppose $\left\{  b_{l,\hbar}\left(  x\right)  \right\}
_{l=0}^{m}\subset\mathbb{R}\left[  x\right]  $ and $\eta>0$ satisfies
Assumption \ref{ass.1} and $L_{\hbar}$ is defined by Eq. (\ref{e.1.11}) with
$\mathcal{D}\left(  L_{\hbar}\right)  =\mathcal{S}$ for $0<\hbar<\eta.$ Then
$\bar{L}_{\hbar}$ is self-adjoint and $\mathcal{S}$ is a core for $\bar
{L}_{\hbar}^{n}$ for all $n\in\mathbb{N}$ and $0<\hbar<\eta.$ [Note $\bar
{L}_{\hbar}^{n}$ is a well defined self-adjoint operator by the spectral theorem.]
\end{proposition}

\begin{proof}
Recall that $L_{\hbar}^{n}$ may be written in divergence form as in Eq.
(\ref{e.5.3}) where $B_{\ell,\hbar}=\mathcal{B}_{\ell,\hbar}+R_{\ell,\hbar}$
and $\mathcal{B}_{\ell,\hbar}\in\mathbb{R}\left[  x\right]  $ and
$R_{\ell,\hbar}\in\mathbb{R}\left[  x\right]  $ are as in Eqs. (\ref{e.5.5})
and (\ref{e.5.6}) respectively. By Assumption \ref{ass.1}, $\deg\left(
b_{l-1}\right)  \leq\deg\left(  b_{l}\right)  ,$ which used in combination
with the item 1 in Proposition \ref{pro.5.1} and the definition of
$\mathcal{B}_{\ell}$ in the Eq. (\ref{e.5.5}) implies,
\begin{align*}
\text{$\deg$}\left(  B_{\ell,\hbar}\right)   &  =\text{$\deg$}\left(
\mathcal{B}_{\ell,\hbar}\right)  \leq\max\left\{  \text{$\deg$}\left(
b_{0,\hbar}^{n}\right)  ,0\right\} \\
&  \leq\max\left\{  \text{$\deg$}\left(  \mathcal{B}_{0,\hbar}\right)
,0\right\}  =\max\left\{  \text{$\deg$}\left(  B_{0,\hbar}\right)  ,0\right\}
.
\end{align*}
Each term in $\hat{L}_{\hbar}$ defined in Notation \ref{not.6.8} is a positive
operator and by Theorem \ref{the.1.10}, $\overline{\hat{L}_{\hbar}}$ is
self-adjoint. [Recall that $\mathcal{D}\left(  \hat{L}_{\hbar}\right)
:=\mathcal{S}$.] Moreover by Lemma \ref{lem.6.1}, for all $\delta>0$ there
exists $C_{\delta}<\infty$ such that
\begin{equation}
\left(  \hat{L}_{\hbar}-L_{\hbar}^{n}\right)  ^{2}=\left(  \sum_{\ell
=0}^{nm-1}(-\hbar)^{\ell}C\partial^{2\ell}\right)  ^{2}\preceq_{\mathcal{S}%
}\delta\left(  -\hbar\right)  ^{mn}\partial^{4mn}+C_{\delta}I \label{e.6.29}%
\end{equation}
which implies,%
\[
\left\Vert \left(  \hat{L}_{\hbar}-L_{\hbar}^{n}\right)  \psi\right\Vert
\leq\delta\left\Vert \left(  \hbar\right)  ^{mn}\partial^{2nm}\psi\right\Vert
+C_{\delta}\left\Vert \psi\right\Vert \text{ }\forall~\psi\in\mathcal{S}.
\]
This inequality along with Lemma \ref{lem 6.6.9} then gives%
\begin{align*}
\left\Vert \left(  \hat{L}_{\hbar}-L_{\hbar}^{n}\right)  \psi\right\Vert  &
\leq\delta\left(  C_{1}\left\Vert \hat{L}_{\hbar}\psi\right\Vert
+C_{2}\left\Vert \psi\right\Vert \right)  +C_{\delta}\left\Vert \psi
\right\Vert \\
&  \leq\delta C_{1}\left\Vert \hat{L}_{\hbar}\psi\right\Vert +\left(  \delta
C_{2}+C_{\delta}\right)  \left\Vert \psi\right\Vert \text{ }\forall~\psi
\in\mathcal{S}.
\end{align*}
Therefore for any $a>0$ we may take $\delta>0$ so that $a:=\delta C_{1}$ and
then let $C_{a}:=\left(  \delta C_{2}+C_{\delta}\right)  <\infty$ in the
previous estimate in order to show,%
\begin{equation}
\left\Vert \left(  \hat{L}_{\hbar}-L_{\hbar}^{n}\right)  \psi\right\Vert \leq
a\left\Vert \hat{L}_{\hbar}\psi\right\Vert +C_{a}\left\Vert \psi\right\Vert
\text{ }\forall~\psi\in\mathcal{S}. \label{e.6.30}%
\end{equation}
As a consequence of this inequality with $a<1$ and a variant of the
Kato-Rellich theorem (see \cite[Theorem X.13, p. 163]{Reed1975}), we may
conclude $\overline{L_{\hbar}^{n}}$ is self-adjoint. As this holds for $n=1,$
we conclude that $\bar{L}_{\hbar}$ is self-adjoint. By the spectral theorem,
$\bar{L}_{\hbar}^{n}$ is also self-adjoint. Since $L_{\hbar}^{n}\subset\bar
{L}_{\hbar}^{n},$ we know that $\overline{L_{\hbar}^{n}}\subset\bar{L}_{\hbar
}^{n}$ and therefore $\overline{L_{\hbar}^{n}}=\bar{L}_{\hbar}^{n}$ as both
operators are self-adjoint. Finally, $L_{\hbar}^{n}=\bar{L}_{\hbar}%
^{n}|_{\mathcal{S}}$ and hence $\overline{\bar{L}_{\hbar}^{n}|_{\mathcal{S}}%
}=\overline{L_{\hbar}^{n}}=\bar{L}_{\hbar}^{n}$ which shows $\mathcal{S}$ is a
core for $\bar{L}_{\hbar}^{n}.$

\begin{comments}
There is a gap in this proof namely it shows $\overline{\hat{L}_{\hbar}%
}-\overline{\hat{L}_{\hbar}-L_{\hbar}^{n}}$ is self-adjoint on $\mathcal{D}%
\left(  \overline{\hat{L}_{\hbar}}\right)  $ but one still has to show
$\overline{\hat{L}_{\hbar}}-\overline{\hat{L}_{\hbar}-L_{\hbar}^{n}}%
=\overline{L_{\hbar}^{n}}.$ Taking $A=\left(  \overline{\hat{L}_{\hbar
}-L_{\hbar}^{n}}\right)  $ and $B=\overline{\hat{L}_{\hbar}}$ in Lemma
\ref{lem.6.10}, Eq. (\ref{e.6.30}) implies $\mathcal{D}(\overline{\hat
{L}_{\hbar}})\subseteq\mathcal{D}(\overline{\hat{L}_{\hbar}-L_{\hbar}^{n}})$
and
\[
\left\Vert \left(  \overline{\hat{L}_{\hbar}-L_{\hbar}^{n}}\right)
\psi\right\Vert \leq a\left\Vert \overline{\hat{L}_{\hbar}}\psi\right\Vert
+C_{a}\left\Vert \psi\right\Vert \text{ for }\psi\in\mathcal{D}(\overline
{\hat{L}_{\hbar}}).
\]
By the Kato-Rellich theorem or Wust's theorem, see \citep[ p.162]{Reed1975}
and \citep[ p.164]{Reed1975} respectively, we may conclude by taking
$A=\overline{\hat{L}_{\hbar}}$ and $B=\overline{\hat{L}_{\hbar}-L_{\hbar}^{n}%
}$ in Wust's theorem in \citep[ p.164]{Reed1975}, we have $\overline{L_{\hbar
}^{n}}$ is self adjoint with a core $\mathcal{S}$.
\end{comments}

\begin{comments}
$n=1$ proof. Let $C>0$ be chosen so that Eq. (\ref{e.6.25}) holds and then
define the operator, $\hat{L}_{\hbar},$ by
\begin{align*}
\hat{L}_{\hbar}:=  &  \sum_{\ell=0}^{m}\left(  -\hbar\right)  ^{\ell}%
\partial^{\ell}\left(  b_{\ell,\hbar}\left(  \sqrt{\hbar}\left(  \cdot\right)
\right)  +C1_{\ell<m}\right)  \partial^{\ell}\\
=  &  \left(  -\hbar\right)  ^{m}\partial^{m}b_{m,\hbar}\left(  \sqrt{\hbar
}\left(  \cdot\right)  \right)  \partial^{m}+\sum_{\ell=0}^{m-1}\left(
-\hbar\right)  ^{\ell}\partial^{\ell}\left(  b_{\ell,\hbar}\left(  \sqrt
{\hbar}\left(  \cdot\right)  \right)  +C\right)  \partial^{\ell}%
\end{align*}
with domain, $\mathcal{D}\left(  \hat{L}_{\hbar}\right)  =\mathcal{S}.$
\end{comments}
\end{proof}

\begin{lemma}
\label{lem.6.12}If $A$ is any essentially self-adjoint operator on a Hilbert
space $K$ and $q:\mathbb{R\rightarrow C}$ is a measurable function such that,
for some constants $C_{1}$ and $C_{2},$%
\[
\left\vert q\left(  x\right)  \right\vert \leq C_{1}\left\vert x\right\vert
+C_{2}~\forall~x\in\mathbb{R},
\]
then $\mathcal{D}\left(  A\right)  $ is a core for $q\left(  \bar{A}\right)
.$
\end{lemma}

\begin{proof}
To prove this we may assume by the spectral theorem that $\mathcal{K}%
=L^{2}\left(  \Omega,\mathcal{B},\mu\right)  $ and $\bar{A}=M_{f}$ where
$\left(  \Omega,\mathcal{B},\mu\right)  $ is a $\sigma$ -- finite measure
space and $f:\Omega\rightarrow\mathbb{R}$ is a measurable function. Of course
in this model, $q\left(  \bar{A}\right)  =M_{q\circ f}.$ In this case,
$\mathcal{D}:=\mathcal{D}\left(  A\right)  \subset\mathcal{D}\left(
M_{f}\right)  $ is a dense subspace of $L^{2}\left(  \mu\right)  $ such that
for all $g\in\mathcal{D}\left(  M_{f}\right)  $ there exists $g_{n}%
\in\mathcal{D}$ such that $g_{n}\rightarrow g$ and $fg_{n}\rightarrow fg$ in
$L^{2}\left(  \mu\right)  $ as $n\rightarrow\infty.$ For this same sequence we
have
\[
\left\Vert q\left(  \bar{A}\right)  g_{n}-q\left(  \bar{A}\right)
g\right\Vert _{2}=\left\Vert q\left(  f\right)  \left[  g_{n}-g\right]
\right\Vert _{2}\leq C_{1}\left\Vert f\left[  g_{n}-g\right]  \right\Vert
_{2}+C_{2}\left\Vert g_{n}-g\right\Vert \rightarrow0
\]
as $n\rightarrow\infty.$ This shows that
\begin{equation}
q\left(  \bar{A}\right)  |_{\mathcal{D}\left(  M_{f}\right)  }\subset
\overline{q\left(  \bar{A}\right)  |_{\mathcal{D}}}\subset q\left(  \bar
{A}\right)  . \label{e.6.31}%
\end{equation}

For $g\in\mathcal{D}\left(  q\left(  \bar{A}\right)  \right)  $ (i.e. both $g$
and $g\cdot q\circ f$ are in $\in L^{2}\left(  \mu\right)  $ ), let
$g_{n}:=g1_{\left\vert f\right\vert \leq n}\in\mathcal{D}\left(  M_{f}\right)
.$ Then $g_{n}\rightarrow g$ in $L^{2}\left(  \mu\right)  $ as $n\rightarrow
\infty$ by DCT. Moreover
\[
\left\vert g_{n}q\circ f-gq\circ f\right\vert =\left(  g1_{\left\vert
f\right\vert \leq n}-g\right)  q\circ f\leq2\left\vert g\right\vert \left\vert
q\circ f\right\vert \in L^{2}\left(  \mu\right)  ,
\]
and so $\left\Vert g_{n}q\circ f-gq\circ f\right\Vert _{2}\rightarrow0$ as
$n\rightarrow\infty$ by DCT as well. This shows that $q\left(  \bar{A}\right)
=\overline{q\left(  \bar{A}\right)  |_{\mathcal{D}\left(  M_{f}\right)  }}$
and hence it now follows from Eq. (\ref{e.6.31}) that
\[
q\left(  \bar{A}\right)  =\overline{q\left(  \bar{A}\right)  |_{\mathcal{D}%
\left(  M_{f}\right)  }}\subset\overline{q\left(  \bar{A}\right)
|_{\mathcal{D}}}\subset q\left(  \bar{A}\right)  .
\]

\end{proof}

\begin{lemma}
\label{lem.6.13}Let $B$ be a non-negative self-adjoint operator on a Hilbert
space, $\mathcal{K}.$ If $S$ is a core for $B^{n}$ for some $n\in
\mathbb{N}_{0},$ then $S$ is a core for $B^{r}$ for any $0\leq r\leq n.$ [By
the spectral theorem, $B^{r}$ is again a non-negative self-adjoint operator on
$\mathcal{K}$ for any $0\leq r<\infty.]$
\end{lemma}

\begin{proof}
Let $A=B^{n}|_{S}$ so that by assumption $\bar{A}=B^{n},$ i.e. $A$ is
essentially self-adjoint. The proof is then finished by applying Lemma
\ref{lem.6.12} with $q\left(  x\right)  =\left\vert x\right\vert ^{r/n}$ upon
noticing, $q\left(  \bar{A}\right)  =q\left(  B^{n}\right)  =\left\vert
B^{n}\right\vert ^{r/n}=B^{r}.$
\end{proof}

\begin{proof}
[Proof of Corollary \ref{cor.1.20}]Let $C\geq0$ be the constant in the
statement of Corollary \ref{cor.1.20}. It is simple to verify that $\left\{
b_{l,\hbar}+C1_{l=0}\right\}  _{l=0}^{m}$ and $\eta>0$ satisfies Assumption
\ref{ass.1}, and therefore applying Proposition \ref{pro.6.11} with $\left\{
b_{l,\hbar}\right\}  _{l=0}^{m}$ replaced by $\left\{  b_{l,\hbar}%
+C1_{l=0}\right\}  _{l=0}^{m}$ shows, $\bar{L}_{\hbar}+C$ is self-adjoint and
$\mathcal{S}$ is core for $\left(  \bar{L}_{\hbar}+C\right)  ^{n}$ for all
$n\in\mathbb{N}$ and $0<\hbar<\eta.$ It then follows from Lemma \ref{lem.6.13}
that $S=\mathcal{S}$ is a core for $\left(  \bar{L}_{\hbar}+C\right)  ^{r}$
for all $0\leq r\leq n$ and $0<\hbar<\eta.$ As $n\in\mathbb{N}$ was arbitrary,
the proof is complete.
\end{proof}

\subsection{Proof of Corollary \ref{cor.1.21}\label{sec.6.4}}

In order to prove Corollary \ref{cor.1.21}, we will need a lemma below.

\begin{lemma}
\label{lem.6.14} Let $A$ and $B$ be non-negative self-adjoint operators on a
Hilbert Space $\mathcal{K}.$ Suppose $S$ is a dense subspace of $\mathcal{K}$
so that $S\subseteq\mathcal{D}\left(  A\right)  \cap\mathcal{D}\left(
B\right)  ,$ $AS\subseteq S$ and $BS\subseteq S.$ If we further assume that
for each $n\in\mathbb{N}_{0},$ $S$ is a core of $B^{n}.$ Then the following
are equivalent:

\begin{enumerate}
\item For any $n\in\mathbb{N}_{0}$ there exists $C_{n}>0$ such that
$A^{n}\preceq_{S}C_{n}B^{n}.$

\item For each $r\geq0,$ there exists $C_{r}$ such that $A^{r}\leq C_{r}%
B^{r}.$

\item For each $v\geq0,$ there exists $C_{v}$ such that $A^{v}\preceq
C_{v}B^{v}.$
\end{enumerate}

Recall the different operator inequality notations, $\preceq_{S},$ $\preceq$
and $\leq,$ were defined in Notation \ref{not.1.14}.
\end{lemma}

\begin{proof}
$\left(  1\Rightarrow2\right)  $ $A^{n}\preceq_{\mathcal{S}}C_{n}B^{n}$
implies for all $\psi\in S$ we have
\[
\left\Vert \sqrt{A^{n}}\psi\right\Vert ^{2}=\left\langle A^{n}\psi
,\psi\right\rangle \leq C_{n}\left\langle B^{n}\psi,\psi\right\rangle
=\left\Vert \sqrt{C_{n}B^{n}}\psi\right\Vert ^{2}.
\]

Note $S$ is a core of $C_{n}B^{n}$ and hence $S$ is also a core of
$\sqrt{C_{n}B^{n}}$ by taking $q\left(  x\right)  =\sqrt{\left\vert
x\right\vert }$ in Lemma \ref{lem.6.12}. By using Lemma \ref{lem.6.10} with
$C=0$ we have $\mathcal{D}\left(  \sqrt{B^{n}}\right)  =\mathcal{D}\left(
\sqrt{C_{n}B^{n}}\right)  \subseteq\mathcal{D}\left(  \sqrt{A^{n}}\right)  $
and
\[
\left\Vert \sqrt{A^{n}}\psi\right\Vert \leq\left\Vert \sqrt{C_{n}B^{n}}%
\psi\right\Vert \text{ for all }\psi\in\mathcal{D}\left(  \sqrt{C_{n}B^{n}%
}\right)  ,
\]
i.e. $A^{n}\leq C_{n}B^{n}.$ It then follows by the L\"{o}wner-Heinz
inequality (Theorem \ref{the.1.17}) that $A^{nr}\leq C_{n}^{r}B^{nr}$ for all
$0\leq r\leq1.$ Since $n\in\mathbb{N}$ was arbitrary, we have verified the
truth of item 2.

$\left(  2\Rightarrow3\right)  $ Given item 2., it is easy to verify that
$\mathcal{D}\left(  B^{v}\right)  =\mathcal{D}\left(  C_{v}B^{v}\right)
\subseteq\mathcal{D}\left(  A^{v}\right)  $ for all $v\geq0.$ In particularly,
we have $\mathcal{D}\left(  B^{v}\right)  \subseteq\mathcal{D}\left(
A^{v}\right)  \cap\mathcal{D}\left(  \sqrt{B^{v}}\right)  $ for any $v\geq0.$
Hence, by taking $r=v$ in item 2,%
\[
\left\langle A^{v}\psi,\psi\right\rangle =\left\Vert \sqrt{A^{v}}%
\psi\right\Vert ^{2}\leq\left\Vert \sqrt{C_{v}B^{v}}\psi\right\Vert
^{2}=\left\langle C_{v}B^{v}\psi,\psi\right\rangle ~\forall~\psi\in
\mathcal{D}\left(  B^{v}\right)  ,
\]
i.e. $A^{v}\preceq C_{v}B^{v}.$

$\left(  3\Rightarrow1\right)  $ The assumption that $S\subseteq
\mathcal{D}\left(  A\right)  \cap\mathcal{D}\left(  B\right)  ,$ $AS\subseteq
S$ and $BS\subseteq S$ follows that $S\subseteq\mathcal{D}\left(
B^{n}\right)  \cap\mathcal{D}\left(  A^{n}\right)  $ for all $n\in
\mathbb{N}_{0}.$ By taking $v=n,$ we learn that $A^{n}\preceq C_{n}B^{n}$
which certainly implies $A^{n}\preceq_{S}C_{n}B^{n}.$
\end{proof}

\begin{proof}
[Proof of the Corollary \ref{cor.1.21}]{We first observe that the
coefficients, }$\left\{  b_{l,\hbar}\left(  \cdot\right)  +C1 _{l=0}\right\}
_{l=0}^{m_{L}}$ and $\left\{  \tilde{b}_{l,\hbar}\left(  \cdot\right)
+\tilde{C}1_{l=0}\right\}  _{l=0}^{m_{\tilde{L}}}$ still satisfy Assumption
\ref{ass.1}. Using this observation along with the inequalities, $L_{\hbar
}+C\succeq_{\mathcal{S}}I$ and $\tilde{L}_{\hbar}+\tilde{C}\succeq
_{\mathcal{S}}0,$ we may use Corollary \ref{cor.1.20} to conclude both
$\bar{L}_{\hbar}+C$ and $\overline{\tilde{L}_{\hbar}}+\tilde{C}$ are
non-negative self adjoint operators and $\mathcal{S}$ is a core for $\left(
\bar{L}_{\hbar}+C\right)  ^{r}$ for all $r\geq0$ and all $0<\hbar<\eta.$ By
the operator comparison Theorem \ref{the.1.18} with $b_{l,\hbar}$ replaced by
$b_{l,\hbar}+C1_{l=0}$ and $\tilde{b}_{l,\hbar}$ replaced by $\tilde
{b}_{l,\hbar}+\tilde{C}1_{l=0},$ for any $n\in\mathbb{N}_{0},$ there exists
$C_{1}$ and $C_{2}>0$ such that
\begin{equation}
\left(  \tilde{L}_{\hbar}+\tilde{C}\right)  ^{n}\preceq_{\mathcal{S}}%
C_{1}\left(  \left(  L_{\hbar}+C\right)  ^{n}+C_{2}\right)  . \label{e.6.32}%
\end{equation}
Because $\left(  L_{\hbar}+C\right)  ^{n}\succeq_{\mathcal{S}}I,$ we may
conclude from Eq. (\ref{e.6.32}) that
\[
\left(  \tilde{L}_{\hbar}+\tilde{C}\right)  ^{n}\preceq_{\mathcal{S}}%
C_{n}\left(  L_{\hbar}+C\right)  ^{n}\text{ ~}\forall~\text{ }n\in
\mathbb{N}_{0},
\]
where $C_{n}=C_{1}\left(  C_{2}+1\right)  .$ By taking $A=\overline{\tilde
{L}_{\hbar}}+\tilde{C}$ and $B=\left(  \bar{L}_{\hbar}+C\right)  $ and
$S=\mathcal{S}$ in Lemma \ref{lem.6.14}, we may conclude that for any
$v\geq0,$ there exists $C_{v}>0$ such that Eq. (\ref{e.1.18}) holds,
i.e.$\left(  \overline{\tilde{L}_{\hbar}}+\tilde{C}\right)  ^{r}\preceq
C_{r}\left(  \bar{L}_{\hbar}+C\right)  ^{r}~\forall~0<\hbar<\eta.$
\end{proof}

\section{Discussion of the 2nd condition in Assumption \ref{ass.1}%
\label{sec.7}}

We try to relax conditions $2$ in Assumption \ref{ass.1}. The degree
restriction Eq. (\ref{e.1.12}) allows the choice of $\eta$ independent of a
power $n$ in both Theorem \ref{the.6.4} and Theorem \ref{the.1.18}. If a
weaker condition of the degree restriction is assumed, which is
\[
\deg(b_{l,\hbar})\leq\deg(b_{l-1,\hbar})+2\text{ for all }0<\hbar<\eta\text{
and }0\leq l\leq m,
\]
then Theorems \ref{the.7.2} and \ref{the.7.3} are resulted where now $\eta$
does depend on $n.$

\begin{lemma}
\label{lem.7.1} Supposed there exists $\eta>0$ such that $\deg(b_{l,\hbar
})\leq\deg(b_{l-1,\hbar})+2$ for all $0<\hbar<\eta$ and $0\leq k\leq m.$ Let
$\mathcal{B}_{\ell,\hbar}\left(  x\right)  $ and $R_{\ell,\hbar}\left(
x\right)  $ be in Eqs. (\ref{e.5.5}) and (\ref{e.5.6}) respectively. Then we
have
\begin{equation}
\operatorname{deg}_{x}\left(  R_{\ell,\hbar}\right)  \leq\operatorname{deg}%
_{x}\left(  \mathcal{B}_{\ell,\hbar}\right)  \text{ for }\hbar\in\left(
0,\eta\right)  \text{ and }0\leq\ell\le mn. \label{e.7.1}%
\end{equation}

\end{lemma}

\begin{proof}
Eq. (\ref{e.7.1}) follows immediately if we apply the item $2$ in Proposition
\ref{pro.5.1} with $b_{l}\left(  x\right)  \rightarrow\hbar^{l}b_{l,\hbar
}\left(  \sqrt{\hbar}x\right)  $ where $\hbar$ is fixed.
\end{proof}

\begin{theorem}
\label{the.7.2} Let $L_{\hbar}$ be an operator in the Eq. (\ref{e.1.11}).
Supposed $b_{l,\hbar}\left(  x\right)  $ satisfies the conditions $1$ and $3$
in Assumption \ref{ass.1} and we assume
\begin{equation}
\deg(b_{l,\hbar})\leq\deg(b_{l-1,\hbar})+2\text{ for all }0<\hbar<\eta\text{
and }0\leq l\leq m, \label{e.7.2}%
\end{equation}
where $\eta$ is the $\eta$ in Assumption \ref{ass.1}. Then for any
$n\in\mathbb{N}_{0}$, there exists $C_{n}$ and $\eta_{n}$ such that for all
$0<\hbar<\eta_{n}$ and $c>C_{n}$
\[
\frac{3}{2}\left(  \mathcal{L}_{\hbar}^{\left(  n\right)  }+c\right)
\succeq_{\mathcal{S}}L_{\hbar}^{n}+c\succeq_{\mathcal{S}}\frac{1}{2}\left(
\mathcal{L}_{\hbar}^{\left(  n\right)  }+c\right)  .
\]

\end{theorem}

\begin{proof}
Let $\psi\in\mathcal{S}$ and $0<\hbar<\eta$, we have
\[
|\langle\left(  L_{\hbar}^{n}-\mathcal{L}_{\hbar}^{\left(  n\right)  }\right)
\psi,\psi\rangle|=\left\vert \left\langle \mathcal{R}_{\hbar}^{\left(
n\right)  }\psi,\psi\right\rangle \right\vert \leq\sum_{\ell=0}^{nm-1}%
\left\vert \left\langle R_{\ell,\hbar}\partial^{\ell}\psi,\partial^{\ell}%
\psi\right\rangle \right\vert .
\]
where $\mathcal{R}_{\hbar}^{\left(  n\right)  }$ and $R_{\ell,\hbar}$ are
still defined in the same way as Eqs. (\ref{e.5.10}) and (\ref{e.5.6})
respectively. From Lemma \ref{lem.7.1}, $\deg_{x}(\mathcal{B}_{\ell,\hbar
})\geq\deg_{x}\left(  R_{\ell,\hbar}\right)  $ where $\mathcal{B}_{\ell,\hbar
}$ is defined in Eq. (\ref{e.5.5}). Note $\left|  p\right|  >0$ in
$R_{\ell,\hbar}$ from Eq. (\ref{e.5.6}). Although $\deg\left(  R_{\ell,\hbar
}\right)  $ can be the same as $\deg(\mathcal{B}_{\ell,\hbar})$, the extra
$\hbar^{\left\vert \mathbf{p}\right\vert }$ factor in the $R_{\ell}\left(
\hbar\right)  $ makes $\left\vert R_{\ell,\hbar}\right\vert $ decrease more
rapidly than $\mathcal{B}_{\ell,\hbar}$ as $\hbar$ decrease to $0$. As a
result, there exist constants $\eta_{n}>0$ and $C$ such that
\[
\left\vert R_{\ell,\hbar}\left(  x\right)  \right\vert \leq\frac{1}%
{2}\mathcal{B}_{\ell,\hbar}\left(  x\right)  +C
\]
for all $0\leq\ell\leq mn-1$ and $0<\hbar<\eta_{n}$. Therefore
\begin{align}
|\langle\left(  L_{\hbar}^{n}-\mathcal{L}_{\hbar}^{\left(  n\right)  }\right)
\psi,\psi\rangle|  &  \leq\sum_{\ell=0}^{nm-1}\hbar^{\ell}\left\langle \left(
\frac{1}{2}\mathcal{B}_{\ell,\hbar}\left(  \sqrt{\hbar}\left(  \cdot\right)
\right)  +C\right)  \partial^{\ell}\psi,\partial^{\ell}\psi\right\rangle
\nonumber\\
&  =\frac{1}{2}\sum_{\ell=0}^{nm-1}\left\langle \left(  -\hbar\right)  ^{\ell
}\partial^{\ell}\mathcal{B}_{\ell,\hbar}\partial^{\ell}\psi,\psi\right\rangle
+C\left\langle \sum_{\ell=0}^{nm-1}\left(  -\hbar\right)  ^{\ell}%
\partial^{2\ell}\psi,\psi\right\rangle . \label{e.7.3}%
\end{align}
Then by following the argument in Theorem \ref{the.6.4}, we can conclude that
there exists $C_{n}>0$ such that for all $0<\hbar<\eta_{n}$ and $c>C_{n}$ we
have
\[
\frac{1}{2}\sum_{\ell=0}^{nm-1}\left\langle \partial^{\ell}\mathcal{B}%
_{\ell,\hbar}\partial^{\ell}\psi,\psi\right\rangle +C\left\langle \sum
_{\ell=0}^{nm-1}\left(  -\hbar\right)  ^{\ell}\partial^{2\ell}\psi
,\psi\right\rangle \leq\frac{1}{2}\left\langle \left(  \mathcal{L}_{\hbar
}^{\left(  n\right)  }+c\right)  \psi,\psi\right\rangle .
\]
The result follows immediately by combing the above inequality and Eq.
(\ref{e.7.3}).
\end{proof}

As a result, the operator comparison theorem now have choice of $\eta$
depending on a power $n$.

\begin{theorem}
\label{the.7.3} Let
\[
\tilde{L}_{\hbar}=\sum_{\ell=0}^{m_{\tilde{L}}}(-\hbar)^{k}\partial^{k}%
\tilde{b}_{k,\hbar}(\sqrt{\hbar}x)\partial^{k}\text{ and }L_{\hbar}=\sum
_{\ell=0}^{m_{L}}(-\hbar)^{k}\partial^{k}b_{k,\hbar}(\sqrt{\hbar}%
x)\partial^{k}%
\]
be operators on $\mathcal{S}$ satisfies conditions in Theorem \ref{the.7.2}.
Denote $\eta_{\tilde{L}}$ and $\eta_{L}$ as the $\eta$ of $\tilde{L}_{\hbar}$
and $L_{\hbar}$ in Assumption \ref{ass.1} respectively. If $m_{\tilde{L}}\leq
m_{L}$ and there exists $c_{1}$ and $c_{2}$ such that
\[
|\tilde{b}_{l,\hbar}\left(  x\right)  |\leq c_{1}\left(  b_{l,\hbar}\left(
x\right)  +c_{2}\right)  \text{ for all }0\leq\ell\leq m_{\tilde{L}}\text{ and
}0<\hbar<\min\{\eta_{\tilde{L}},\eta_{L}\},
\]
then for any $n$, there exists $C_{1}$, $C_{2}$ and $\eta_{n}$ such that
\[
\left(  \tilde{L}_{\hbar}\right)  ^{n}\preceq_{\mathcal{S}}C_{1}\left(
L_{\hbar}^{n}+C_{2}\right)
\]
for all $0<\hbar<\eta_{n}.$
\end{theorem}

\begin{proof}
The exact same proof as Theorem \ref{the.1.18} with the use of Theorem
\ref{the.7.2} instead of Theorem \ref{the.6.4}.
\end{proof}

\appendix

\section{Operators Associated to Quantization\label{sec.A}}

Let $\mathcal{A}$ denote the algebra of linear differential operator on
$\mathcal{S}$ which have polynomial coefficients. Remark \ref{rem.1.7} shows
that the $\dagger$ operation on $\mathcal{A}$ defined in Eq. (\ref{e.1.4}) is
an involution of $\mathcal{A}.$ For $\hbar>0$ (following on p.204 in
\citep{Reed1975} or Hepp \citep{Hepp1974}), let $a_{\hbar}\in\mathcal{A}$ and
its formal adjoint, $a_{\hbar}^{\dag},$ be the annihilation and creation
operators respectively given by%
\begin{equation}
a_{\hbar}=\sqrt{\frac{\hbar}{2}}\left(  M_{x}+\partial_{x}\right)  \text{ and
}a_{\hbar}^{\dagger}:=\sqrt{\frac{\hbar}{2}}\left(  M_{x}-\partial_{x}\right)
\text{ on }\mathcal{S}. \label{e.A.1}%
\end{equation}
These operators satisfy the commutation relation $\left[  a_{\hbar},a_{\hbar
}^{\dag}\right]  =\hbar I$ on $\mathcal{S}.$

Let $\mathbb{\mathbb{R}}\left\langle \theta,\theta^{\ast}\right\rangle $ be
the space of non-commutative polynomials over $\mathbb{R}$ in two
indeterminants $\left\{  \theta,\theta^{\ast}\right\}  .$ Thus, given
$H\left(  \theta,\theta^{\ast}\right)  \in\mathbb{\mathbb{R}}\left\langle
\theta,\theta^{\ast}\right\rangle ,$ there exists $d\in\mathbb{N}$ (the degree
of $H\left(  \theta,\theta^{\ast}\right)  $ in $\theta$ and $\theta^{\ast}$)
and coefficients,
\[
\cup_{k=0}^{d}\left\{  C_{k}\left(  \mathbf{b}\right)  \in\mathbb{R}%
:\mathbf{b}\in\left\{  \theta,\theta^{\ast}\right\}  ^{k}\right\}
\]
such that $H\left(  \theta,\theta^{\ast}\right)  =\sum_{k=0}^{d}H_{k}\left(
\theta,\theta^{\ast}\right)  $ where
\begin{equation}
H_{k}\left(  \theta,\theta^{\ast}\right)  :=\sum_{\mathbf{b}=\left(
b_{1},\dots,b_{k}\right)  \in\left\{  \theta,\theta^{\dagger}\right\}  ^{k}%
}C_{k}\left(  \mathbf{b}\right)  b_{1}\dots b_{k}\in\mathbb{\mathbb{R}%
}\left\langle \theta,\theta^{\ast}\right\rangle . \label{e.A.2}%
\end{equation}
We let $H\left(  \theta,\theta^{\ast}\right)  ^{\ast}\in\mathbb{\mathbb{R}%
}\left\langle \theta,\theta^{\ast}\right\rangle $ be defined by $H\left(
\theta,\theta^{\ast}\right)  ^{\ast}=\sum_{k=0}^{d}H_{k}\left(  \theta
,\theta^{\ast}\right)  ^{\ast}$ where
\begin{equation}
H_{k}\left(  \theta,\theta^{\ast}\right)  ^{\ast}:=\sum_{\mathbf{b}=\left(
b_{1},\dots,b_{k}\right)  \in\left\{  \theta,\theta^{\ast}\right\}  ^{k}}%
C_{k}\left(  \mathbf{b}\right)  b_{k}^{\ast}\dots b_{1}^{\ast} \label{e.A.3}%
\end{equation}
and for $b\in\left\{  \theta,\theta^{\ast}\right\}  ,$%
\[
b^{\ast}:=\left\{
\begin{array}
[c]{ccc}%
\theta^{\ast} & \text{if} & b=\theta\\
\theta & \text{if} & b=\theta^{\ast}%
\end{array}
.\right.
\]
The operation, $H\left(  \theta,\theta^{\ast}\right)  \rightarrow H\left(
\theta,\theta^{\ast}\right)  ^{\ast}$ defines an involution on
$\mathbb{\mathbb{R}}\left\langle \theta,\theta^{\ast}\right\rangle $ and we
say that $H\left(  \theta,\theta^{\ast}\right)  \in$ $\mathbb{\mathbb{R}%
}\left\langle \theta,\theta^{\ast}\right\rangle $ is \textbf{symmetric} if
$H\left(  \theta,\theta^{\ast}\right)  =H\left(  \theta,\theta^{\ast}\right)
^{\ast}.$ If $H\left(  \theta,\theta^{\ast}\right)  \in$ $\mathbb{\mathbb{R}%
}\left\langle \theta,\theta^{\ast}\right\rangle $ is symmetric, then $H\left(
a_{\hbar},a_{\hbar}^{\dagger}\right)  $ is a symmetric linear differential
operator with polynomial coefficients as in Definition \ref{def.1.6}.

In the following lemmas and theorem let $\mathbb{R}\left[  x\right]  $ and
$\mathbb{R}\left[  \sqrt{\hbar},x\right]  $ be as in Notation \ref{not.5.2}.

\begin{lemma}
\label{lem.A.1} If $\hbar>0$ and $H\in\mathbb{R}\left\langle \theta
,\theta^{\ast}\right\rangle $ is a noncommutative polynomial with degree $d$,
then $H\left(  a_{\hbar},a_{\hbar}^{\dagger}\right)  $ can be written as a
linear differential operator
\[
H\left(  a_{\hbar},a_{\hbar}^{\dagger}\right)  =\sum_{l=0}^{d}\hbar^{\frac
{l}{2}}G_{l}\left(  \sqrt{\hbar},\sqrt{\hbar}x\right)  \partial_{x}^{l}%
\]
where $G_{l}\left(  \sqrt{\hbar},x\right)  \in\mathbb{R}\left[  \sqrt{\hbar
},x\right]  $ is a polynomial of $\sqrt{\hbar}$ and $x$ for $0\leq l\leq d.$
\end{lemma}

\begin{proof}
Let $H\left(  \theta,\theta^{*}\right)  =\sum_{k=0}^{d}H_{k}\left(
\theta,\theta^{*}\right)  $ with $H_{k}\left(  \theta,\theta^{*}\right)  $ be
as in Eq. (\ref{e.A.2}). We then have $H\left(  a_{\hbar},a_{\hbar}^{\dagger
}\right)  =\sum_{k=0}^{d}H_{k}\left(  a{}_{\hbar},a_{\hbar}^{\dagger}\right)
$ where
\[
H_{k}\left(  a_{\hbar},a_{\hbar}^{\dagger}\right)  =\left(  \hbar\right)
^{k/2}\sum_{\mathbf{b}\in\left\{  \theta,\theta^{\ast}\right\}  ^{k}}%
C_{k}\left(  \mathbf{b}\right)  \hat{b}_{1}\dots\hat{b}_{k},
\]
and (as in \citep{BruceDriver20152nd}) for $b\in\left\{  \theta,\theta
^{*}\right\}  ,$
\[
\hat{b}:=\left\{
\begin{array}
[c]{lll}%
a & \text{if} & b=\theta\\
a^{\dagger} & \text{if} & b=\theta^{*}.
\end{array}
\right.
\]
Using the definition of $a_{\hbar}$ and $a_{\hbar}^{\dag}$ in Eq.
(\ref{e.A.1}), there exists
\[
\left\{  \tilde{C}_{k}\left(  \mathbf{\varepsilon}\right)  \in\mathbb{R}%
:\mathbf{\varepsilon}=\left(  \varepsilon_{1},\dots,\varepsilon_{k}\right)
\in\left\{  \pm1\right\}  ^{k}\right\}
\]
such that
\[
H_{k}\left(  a_{\hbar},a_{\hbar}^{\dagger}\right)  =\left(  \hbar\right)
^{k/2}\sum_{\mathbf{\varepsilon}\in\left\{  \pm1\right\}  ^{k}}C\left(
\mathbf{\varepsilon}\right)  \left(  x+\varepsilon_{1}\partial_{x}\right)
\dots\left(  x+\varepsilon_{k}\partial_{x}\right)  .
\]
From the previous equation it is easy to see
\begin{equation}
H_{k}\left(  a_{\hbar},a_{\hbar}^{\dagger}\right)  =\left(  \hbar\right)
^{k/2}\sum_{l=0}^{k}g_{l,k}\left(  x\right)  \partial_{x}^{l} \label{e.A.4}%
\end{equation}
where $g_{l,k}\in\mathbb{R}\left[  x\right]  $ with
\begin{equation}
\operatorname{deg}_{x}\left(  g_{l,k}\right)  \leq k-l. \label{e.A.5}%
\end{equation}
Summing Eq. (\ref{e.A.4}) on $k$ and then switching two sums shows
\begin{equation}
H\left(  a_{\hbar},a_{\hbar}^{\dagger}\right)  =\sum_{k=0}^{d}\sum_{l=0}%
^{k}\left(  \hbar\right)  ^{k/2}g_{l,k}\left(  x\right)  \partial_{x}^{l}%
=\sum_{l=0}^{d}\hbar^{\frac{l}{2}}\left(  \sum_{k=l}^{d}\hbar^{\frac{k-l}{2}%
}g_{l,k}\left(  x\right)  \right)  \partial_{x}^{l}. \label{e.A.6}%
\end{equation}
There exists $G_{l}\left(  \sqrt{\hbar},x\right)  \in\mathbb{R}\left[
\sqrt{\hbar},x\right]  $ such that
\[
G_{l}\left(  \sqrt{\hbar},\sqrt{\hbar}x\right)  =\sum_{k=l}^{d}\hbar
^{\frac{k-l}{2}}g_{l,k}\left(  x\right)
\]
because each monomial of $x$ in $g_{l,k}\left(  x\right)  $ can be multiplied
with enough $\sqrt{\hbar}$ from $\hbar^{\frac{k-l}{2}}$ by using Eq.
(\ref{e.A.5}).
\end{proof}

\begin{theorem}
\label{the.A.2}If $\hbar>0$ and $H\in\mathbb{R}\left\langle \theta
,\theta^{\ast}\right\rangle $ is a \textbf{symmetric} noncommutative
polynomial with degree $d$, then there exits $m\in\mathbb{N}_{0}$ and
$\left\{  f_{l}\right\}  _{l=0}^{m}\subset\mathbb{R}\left[  \sqrt{\hbar
},x\right]  $ such that
\[
H\left(  a_{\hbar},a_{\hbar}^{\dagger}\right)  =\sum_{l=0}^{m}\left(
-\hbar\right)  ^{l}\partial^{l}f_{l}\left(  \sqrt{\hbar},\sqrt{\hbar}x\right)
\partial^{l}\text{ on }\mathcal{S}.
\]

\end{theorem}

\begin{proof}
Since $H$ is symmetric, $H\left(  a_{\hbar},a_{\hbar}^{\dagger}\right)  $ is
symmetric, see Definition \ref{def.1.6}. So by Proposition \ref{pro.2.2},
$d=2m$ for some $m\in\mathbb{N}_{0}$ and $H\left(  a_{\hbar},a_{\hbar
}^{\dagger}\right)  $ in Eq. (\ref{e.A.6}) may be written in a divergence
form
\begin{equation}
H\left(  a_{\hbar},a_{\hbar}^{\dagger}\right)  =\sum_{l=0}^{m}\left(
-1\right)  ^{l}\partial_{x}^{l}M_{b_{l}}\partial_{x}^{l}. \label{e.A.7}%
\end{equation}
By substituting $a_{l}\left(  x\right)  =h^{\frac{l}{2}}G_{l}\left(
\sqrt{\hbar,}\sqrt{\hbar}x\right)  $ for $0\leq l\leq d=2m$ and $r=l$ in Eq.
(\ref{e.2.9}) from Theorem \ref{the.2.7}, for all $0\leq l\leq m,$ we have
\begin{align*}
\left(  -1\right)  ^{l}b_{l}  &  \frac{1}{\hbar^{l}}=\left[  h^{l}%
G_{2l}\left(  \sqrt{\hbar,}\sqrt{\hbar}x\right)  +\sum_{l<s\leq m}%
K_{m}(l,s)h^{s}\partial^{2(s-l)}G_{2s}\left(  \sqrt{\hbar,}\sqrt{\hbar
}x\right)  \right]  \times\frac{1}{\hbar^{l}}\\
&  =G_{2l}\left(  \sqrt{\hbar,}\sqrt{\hbar}x\right)  +\sum_{l<s\leq m}%
K_{m}(l,s)\hbar^{2\left(  s-l\right)  }\left(  \partial^{2(s-l)}G_{2s}\right)
\left(  \sqrt{\hbar,}\sqrt{\hbar}x\right)
\end{align*}
By using Lemma \ref{lem.A.1}, it follows that the R.H.S. in the above equation
is a polynomial of $\sqrt{\hbar}$ and $\sqrt{\hbar}x$ . Therefore, there
exists $f_{l}\left(  \sqrt{\hbar},x\right)  \in\mathbb{R}\left[  \sqrt{\hbar
},x\right]  $ such that
\[
\left(  -1\right)  ^{l}b_{l}=\hbar^{l}f_{l}\left(  \sqrt{\hbar},\sqrt{\hbar
}x\right)
\]
and hence, using the above equation along with Eq. (\ref{e.A.7}), we can
conclude that
\[
H\left(  a_{\hbar},a_{\hbar}^{\dagger}\right)  =\sum_{l=0}^{m}\left(
-1\right)  ^{l}\partial_{x}^{l}M_{b_{l}}\partial_{x}^{l}=\sum_{l=0}^{m}%
\hbar^{l}\partial_{x}^{l}f_{l}\left(  \sqrt{\hbar},\sqrt{\hbar}x\right)
\partial_{x}^{l}.
\]

\end{proof}

\begin{remark}
\label{rem.A.3}The functions, $f_{l}\left(  \sqrt{\hbar},x\right)  ,$ in
Theorem \ref{the.A.2} are examples of the functions, $b_{l,\hbar}\left(
x\right)  ,$ appearing in Eq. (\ref{e.1.10}).
\end{remark}

\begin{example}
\label{exa.A.4} Let $H^{\text{cl}}(x,\xi)=x^{2}\xi^{2}$ be a classical
Hamiltonian where $x$ is position and $\xi$ is momentum on a state space
$\mathbb{R}^{2}$. We would like to lift this to a symmetric polynomial in two
symmetric indeterminate $\hat{q}=\frac{\theta+\theta^{*}}{\sqrt{2}}$ and
$\hat{p}=\frac{\theta-\theta^{*}}{i\sqrt{2}}.$ The Weyl lift of $x^{2}\xi^{2}$
is given by
\begin{align*}
H\left(  \theta,\theta^{*}\right)  =  &  \frac{1}{4!}\left(  \hat{q}^{2}%
\hat{p}^{2}+\text{ all permutations}\right) \\
&  =\frac{1}{4!}\cdot2!\cdot2!\left[
\begin{array}
[c]{c}%
\hat{q}^{2}\hat{p}^{2}+\hat{q}\hat{p}^{2}\hat{q}+\hat{p}^{2}\hat{q}^{2}\\
+\hat{p}\hat{q}^{2}\hat{p}+\hat{p}\hat{q}\hat{p}\hat{q}+\hat{q}\hat{p}\hat
{q}\hat{p}%
\end{array}
\right] \\
&  =\frac{1}{3!}\left[
\begin{array}
[c]{c}%
\hat{q}^{2}\hat{p}^{2}+\hat{q}\hat{p}^{2}\hat{q}+\hat{p}^{2}\hat{q}^{2}\\
+\hat{p}\hat{q}^{2}\hat{p}+\hat{p}\hat{q}\hat{p}\hat{q}+\hat{q}\hat{p}\hat
{q}\hat{p}%
\end{array}
\right]  \in\mathbb{R}\left< \theta,\theta^{*}\right>  .
\end{align*}
Making the substitutions
\[
\hat{q}\rightarrow\frac{a_{\hbar}+a_{\hbar}^{\dagger}}{\sqrt{2}}=\sqrt{\hbar
}M_{x}\text{ and }\hat{p}\rightarrow\frac{a_{\hbar}-a_{\hbar}^{\dagger}%
}{i\sqrt{2}}=\frac{\sqrt{\hbar}}{i}\partial.
\]
above gives the Weyl quantization of $x^{2}\xi^{2}$ to be
\[
H\left(  a_{\hbar},a_{\hbar}^{\dagger}\right)  =-\frac{\hbar^{2}}{3!}\left(
x^{2}\partial^{2}+\partial^{2}x^{2}+x\partial x\partial+\partial x\partial
x+x\partial^{2}x+\partial x^{2}\partial\right)  \text{ on }\mathcal{S}
\]
which after a little manipulation using the product rule repeatedly may be
written as
\begin{align*}
H\left(  a_{\hbar},a_{\hbar}^{\dagger}\right)   &  =-\hbar^{2}\partial
x^{2}\partial-\frac{1}{2}\hbar^{2}\\
&  =-\hbar\partial b_{1,\hbar}\left(  \sqrt{\hbar}x\right)  \partial
+b_{0,\hbar}\left(  \sqrt{\hbar}x\right)
\end{align*}
where $b_{1,\hbar}\left(  x\right)  =x^{2}$ and $b_{0,\hbar}\left(  x\right)
=-\frac{1}{2}\hbar^{2}.$
\end{example}

\bibliographystyle{plain}
\bibliography{semi-classical+opinequal}

\end{document}